\documentclass[11pt]{article}
\usepackage{amsmath,enumerate,amsfonts,amssymb,amsthm,mathrsfs,amsfonts}
\usepackage{textcomp}
\usepackage{bbm}
\usepackage{graphicx}
\usepackage[colorinlistoftodos]{todonotes}
\usepackage{bm}
\usepackage{titlesec}
\usepackage{extarrows}
\allowdisplaybreaks[4]
\usepackage{hyperref}
\usepackage{framed}
\usepackage[normalem]{ulem}
\usepackage[margin=1.25in]{geometry}
\usepackage{cite}
\usepackage{soul}
\hypersetup{
	colorlinks,
	citecolor=blue,
	filecolor=blue,
	linkcolor=blue,
	urlcolor=black
}
\numberwithin{equation}{section}
\newtheorem{thm}{Theorem}[section]
\newtheorem{lem}{Lemma}[section]
\newtheorem{prop}{Proposition}[section]

\newtheorem{cor}{Corollary}[section]

\newcommand{\la}{\langle}
\newcommand{\ra}{\rangle}
\newcommand{\lam}{\lambda}
\newcommand{\morse}{\mathsf{morse}\,}

\newcommand{\sgn}{\mathrm{sign}}
\newcommand{\Hess}{\mathsf{Hess}}

\numberwithin{equation}{section}
\newcommand{\FF}{\mathbb{F}}
\newcommand{\rg}{\mathrm{g}}
\newcommand{\cV}{\mathcal{V}}
\newcommand{\cL}{\mathcal{L}}
\newcommand{\cF}{\mathcal{F}}
\newcommand{\cU}{\mathcal{U}}
\newcommand{\cK}{\mathcal{K}}

\newcommand{\N}{\mathbb{N}}

\newcommand{\R}{\mathbb{R}}
\newcommand{\B}{\mathbb{B}}
\newcommand{\kap}{\kappa}
\newcommand{\ff}{\mathbf{F}_{p,q}^V}
\newcommand{\ii}{\mathrm{i}}
\newcommand{\intS}{\mathring\Si}

\newcommand{\var}{\varepsilon}   
\renewcommand{\d}{\mathrm{d}}
\newcommand{\oH}{\mathring H^1}
\newcommand{\Si}{\Sigma}
\newcommand{\al}{\alpha}
\newcommand{\pa}{\partial}  
\newcommand\mytop[2]{\genfrac{}{}{0pt}{}{#1}{#2}}

\begin{document}
	
	\title{\bf 	
		A degree-counting formula  for  a Keller-Segel equation  on a  surface with boundary 
	}      
	\author{\medskip  {\sc Mohameden Ahmedou\footnote{M. Ahmedou is supported by DFG grant AH 156/2-1.}}, \ \
		{\sc Zhengni Hu\footnote{Corresponding author}}, \ \ 
		{\sc Heming Wang}
	}

	\date{\today}
	\maketitle
	
	\begin{abstract}
		In this paper, we consider the following  Keller-Segel equation on a compact Riemann surface $(\Si, g)$ with  smooth boundary $\pa \Si$:
		\[
		\left\{\begin{aligned}
			-\Delta_g u &= \rho  \Big(\frac{V e^u}{\int_{\Si} V e^u \, \d v_g} - \frac{1}{|\Si|_g}\Big) & & \text{ in }\, \mathring{\Si},\\
			\pa _{\nu_g} u &= 0 & & \text{ on }\, \pa  \Si,
		\end{aligned}\right.
		\]
		where $V$ is a  smooth positive    function on $\Si$ and   $\rho > 0$ is a parameter. 
		
		We perform a refined blow-up analysis of bubbling solutions and establish sharper a priori estimates around their concentration points. We then compute the Morse index  of these solutions and use it to derive a counting formula for the Leray-Schauder degree in the non-resonant case (i.e., $\rho \notin 4 \pi  \N$). Our approach follows the strategy suggested by Y. Y. Li \cite{Li1999} and later implemented by C.-S. Lin and C.-C. Chen  \cite{CL2003, CL2002} for the mean field equations on closed surfaces and employs techniques from Bahri's critical points at infinity \cite{Bahri:1989}.   
	\end{abstract}
	
	

	\section{Introduction}
	Let $(\Si, g)$ be a compact Riemann surface with smooth boundary $\pa\Si$. 
	We consider the following mean field type  equations with homogeneous Neumann boundary conditions:
	\begin{equation}\label{maineq} \tag{$P_\rho$}
		\left\{\begin{aligned}
			&-\Delta_g u =\rho \Big( \frac{V e^u}{\int_{\Si } V e^ u \,\d v_g} -1\Big)  && \text{  in }\, \intS,\\
			&	\pa_{\nu_g} u=0 && \text{ on }\, \pa\Si ,
		\end{aligned}
		\right.
	\end{equation}
	where $(\Si,g)$ is normalized with the unit volume $|\Si|_g=1$,  $\intS=\Si \setminus\pa\Si $ denotes the interior of $\Si $, $\Delta_g$ is the Laplace-Beltrami operator, $\nu_g$ is the unit outward normal to the boundary $\pa\Si $, $\d v_g$ is the volume element in $(\Si , g)$,  $V:\Si \to\R$ is a smooth positive potential  function, and $\rho > 0$ is a parameter. 
	The boundary value problem   \eqref{maineq}   arises naturally in the study of stationary solutions to the Keller-Segel model for chemotaxis; see, for instance, \cite{ks1,ks2,ko1, ko2,kel, ns,ssr, hl} and the references therein. 
	
	The  problem \eqref{maineq} is invariant by adding a constant to its solutions. Hence, we may assume without loss of generality that any solution $u$  satisfies $\int_{\Si}u \, \d v_g=0$. Let  $H^1(\Si )$ be the Sobolev space of $L^2$-functions whose first-order weak derivatives are also in $L^2(\Sigma)$. We  seek  weak solutions in     a subspace of $H^1(\Si )$ with average $0$ over $\Sigma$, that is 
	\[ \oH:= \Big\{u\in H^1(\Si ): \int_{\Si } u \,\d v_g =0\Big\}, \]
	which is equipped with the norm $\|\cdot\|$ induced by the inner product 	$\la u, v \ra:=\int_{\Si} \la \nabla u, \nabla v\ra_g \, \d v_g$ for any $u,v\in\oH$,  where $\la \cdot, \cdot\ra_g$ denotes the inner product induced by the Riemannian metric $g$.

	
	The boundary value problem   \eqref{maineq}  exhibits a variational structure, and the corresponding Euler-Lagrange functional is given by
	\[
	J_{\rho }(u)= \frac{1}{2} \int_{\Si } |\nabla_g u|^2 \,\d v_g -\rho  \ln \Big(\int_{\Si } V e^u \,\d v_g\Big),\quad u\in  \oH.
	\]
	The analytical feature on the functional  $J_{\rho}$ depends strongly on the range of values taken by the parameter $\rho$. Indeed, when $ \rho < 4 \pi$, the functional $J_{\rho}$ is lower bounded and lower semi-continuous and hence achieves its minimum. In contrast, the functional is lower bounded for the value $\rho = 4 \pi$, but the minimum is not necessarily achieved. For $\rho > 4 \pi$, the functional is no longer bounded below or above.
	In recent work, J. Li, L. Sun, and Y. Yang \cite{LSY2023} obtained a compactness result for the problem \eqref{maineq}  via blow-up analysis. They also determined the set of critical parameters $4\pi \N$ such that the a priori bounds for solutions of \eqref{maineq} may fail, where  $\N$ denotes the set of positive integers. More recently, A. Ahmedou, T. Bartsch, and Z. Hu \cite{HBA20240829} have established  some sufficient conditions  for the existence of  blow-up solutions to \eqref{maineq} as $\rho$ approaches 
	$4\pi \N$, where blow-ups can occur both in the interior and on the boundary of the surface. See also their subsequent work  \cite{HBA20240831} for related results on steady states of the Keller-Segel system.
	
	Set
	\[
	\cK (u):=(-\Delta_g) ^{-1} \Big(\frac{V e^u}{\int_{\Si} V e^u\, \d v_g}-1\Big),
	\] 
	which maps $\oH$ into itself. The existence of a priori bound in  \cite{LSY2023} implies there is a large $R>0$ such that $(I-\rho \cK) (u) \neq 0$ for any $u \in \pa \mathcal{B}_R$, where $\mathcal{B}_R:=\{u \in \oH: \|u\| \leq R\}$. Therefore, the Leray-Schauder degree
	\begin{equation}\label{def:degree}
		d_\rho:=\operatorname{deg}(I-\rho \cK, \mathcal{B}_R, 0)
	\end{equation}
	is well-defined provided that $\rho \notin 4\pi\N$ and the main goal of this paper is to compute the degree formula in this case.
	
	The degree-counting formula for mean field equations has been extensively investigated in the context of domains in $ \mathbb{R}^2$ under  Dirichlet boundary conditions,  as well as for closed surfaces.   
	Based on the quantization estimates in  \cite{LS1994} (see also the earlier paper \cite{BM1991}),  Y. Y. Li  \cite{Li1999}  proved that    the Leray-Schauder degree $d_\rho$ is well-defined for all
	$\rho\notin8\pi\N$. Furthermore he    proved that the degree $d_{\rho}$ is constant  in each interval    $\rho\in(8 (m-1)\pi, 8m\pi)$ with $m\in \N$  and completely determined by the Euler characteristic of the surface (or of the domain, in the planar case). Since $d_{\rho} = 1$ if $\rho < 8 \pi $,  Y. Y. Li \cite{Li1999}  pointed out that to compute the Leray-Schauder degree $d_\rho$  for $\rho\in (4\pi (m-1), 4\pi m)$ with  $m\in \N$, it is enough to count the jump in the degree as the parameter  $\rho$ crosses each critical value  $8 \pi l$ for $l \in \{ 1, \ldots, m-1\}$. 
	Subsequently, 
	C.-C. Chen and C.-S.  Lin \cite{CL2003, CL2002}, implementing such a strategy,  established an explicit degree-counting formula  in terms of the Euler characteristic in the non-resonant case (i.e.,  $ \rho \notin 8\pi \N$) and generalized their result to singular mean field equations in \cite{CCL2015}. \par 
	A similar feature also appears in our setting. Since the degree $d_\rho$ is homotopy invariant,   $d_\rho$ is a constant in each interval $\rho \in ( 4\pi (m-1), 4\pi m)$ with   $m\in  \N$, independently of the choice of the potential function $V$  and the Riemannian metric $g$. More precisely,  we derive the following degree formula in the non-resonant case, i.e., $\rho \notin 4 \pi \N$:
	\begin{thm}\label{thm:main}
		Let  $\rho\in (4\pi (m-1), 4\pi m)$ with $m\in  \N$, and let $\ell := [\frac{m-1}{2}]$ be  the greatest integer less than or equal to  $\frac{m-1}{2}$. Then the Leray-Schauder degree for  \eqref{maineq} is given
		by 	
		\begin{equation}\label{eq:degree-nonreso}
			d_{\rho} = \left\{\begin{aligned}
				&1&&\text{ if }\,m=1,2,	\\	&\frac{1}{\ell !}  \prod_{i=1}^{\ell} (i- \chi(\Si))&&\text{ if }\,m\geq 3,
			\end{aligned}\right. 
		\end{equation}
		where $\chi(\Si)$ is the Euler characteristic of the surface $\Si$,  given by  
		\[
		\chi(\Si) = 2 - 2\mathrm{g} - \mathrm{b},
		\]
		with $\mathrm{g}$ denoting the genus of $\Si$, and $\mathrm{b}$  the number of connected components of its  boundary $\pa  \Si$.
	\end{thm}
	As a direct consequence of Theorem \ref{thm:main}, we obtain the following existence result:
	\begin{cor}
		Let $\rho\notin 4\pi \N$. If one of the following conditions holds:
		\begin{itemize}
			\item $\rho\in (0, 4\pi)\cup (4\pi, 8\pi)$;
			\item  $\rho\in (4\pi(m-1), 4\pi m)$ for $m\in \N\setminus\{ 1,2\}$ and  $\chi(\Si)\leq 0$;
		\end{itemize}
		then the boundary value problem  \eqref{maineq} admits at least one solution.	
	\end{cor}
	
	The method for deriving the degree counting formula \eqref{eq:degree-nonreso} is based on computing the degree jump of $d_{\rho}$ as the parameter $\rho$ crosses a critical value $4\pi l$, for $l = 1, \ldots, m$. This approach was first developed by Y. Y. Li \cite{Li1999} and then implemented   by C.-C. Chen and C.-S. Lin \cite{CL2002, CL2003}  to deal with  the mean field equation on closed surfaces  To carry it out, we perform a refined blow-up analysis of solutions as $\rho \uparrow 4\pi l$ and $\rho \downarrow 4\pi l$, combined with techniques from the theory of \emph{critical points at infinity}; see \cite{Bahri:1989, BC2, ABL2017, AB2023, AB2017}.
	Although we follow the general scheme introduced in \cite{CL2003}, new difficulties arise from the fact that the counting formula does not involve a single Kirchhoff-Routh function but rather a family of such functions. A further important and striking difference compared to the closed case is that the degree $d_{\rho}$ does not jump when crossing a critical value $4\pi m$ for an odd integer $m$.
	In addition, compared to \cite{CL2003}, some of our arguments are more concise, and we employ a different approximate solution.
	
	As a by-product of our jump-counting formula, we prove the existence of blow-up solutions to the problem \eqref{maineq} as $\rho \to 4\pi m$ with $m\in \N$. Namely, we have: 
	\begin{thm}\label{thm:blowup}
		Let  $\rho\in (4\pi (m-1), 4\pi m)$ with $m\in  \N$. If $\chi(\Si)< 0$ and $m$ is an even number,  then there exists a sequence of blow-up solutions to  \eqref{maineq} with $\rho\to  4\pi m$.  
	\end{thm}

	Now, we consider the following mean field type equations with non-homogeneous Neumann boundary conditions:
	\begin{equation}\label{eq:Geo}
		\left\{
		\begin{aligned}
			-\Delta_g u &= \rho \left( \frac{K e^u}{\int_{\Sigma}K e^u \, \d v_g} - 1 \right) + \int_{\partial\Sigma} \psi \, \d s_g && \text{in } \, \intS, \\
			\partial_{\nu_g} u &= -\psi && \text{on } \partial\Sigma,
		\end{aligned}
		\right.
	\end{equation}
	where $\psi \in C^\infty(\Sigma)$ is a given function, $K \in C^\infty(\Sigma)$ is positive, and $\mathrm{d}s_g$ denotes the line element on $\partial\Sigma$.
	
	Let $K_g$ denote the Gaussian curvature of $\Sigma$, and $k_g$ the geodesic curvature of its boundary $\partial\Sigma$. A particular geometric instance of \eqref{eq:Geo} arises when we choose $\psi = 2k_g$, in which case the problem becomes:
	\begin{equation}\label{eq:Geo1}
		\left\{
		\begin{aligned}
			-\Delta_g v &= \rho \left( \frac{K e^v}{\int_{\Sigma} K e^v \, \d v_g} - 1 \right) + 2 \int_{\partial\Sigma} k_g \, \d s_g && \text{in } \, \intS, \\
			\partial_{\nu_g} v &= -2k_g && \text{on } \partial\Sigma,
		\end{aligned}
		\right.
	\end{equation}
	where $K \in C^\infty(\Sigma)$ is a given positive function. This boundary value problem is naturally connected to classical problems in differential geometry, specifically the prescription of Gaussian and geodesic curvatures. Indeed, it is related to the conformal deformation problem given by:
	\begin{equation*}
		\left\{
		\begin{aligned}
			-\Delta_g u &= -K_g + K e^{2u} && \text{in } \, \intS, \\
			\partial_{\nu_g} u &= -k_g && \text{on } \partial\Sigma,
		\end{aligned}
		\right.
	\end{equation*}
	where $K_g$ is constant, and the goal is to find a conformal metric with prescribed Gaussian curvature $K$ and boundary geodesic curvature $0$ on a Riemann surface $(\Sigma,g)$ with Gaussian curvature $K_g$ and boundary geodesic curvature $k_g$.
	
	To reduce the non-homogeneous Neumann condition in \eqref{eq:Geo1}, we introduce an auxiliary function $\phi \in C^\infty(\Sigma)$ solving
	\[
	\left\{
	\begin{aligned}
		-\Delta_g \phi &= -2 \int_{\partial \Sigma} k_g \, \d v_g && \text{in } \, \intS, \\
		\partial_{\nu_g} \phi &= 2k_g && \text{on } \partial\Sigma.
	\end{aligned}
	\right.
	\]
	Setting $u = v + \phi$ and $V = K e^{-\phi}$, the problem \eqref{eq:Geo1} is transformed into a standard mean field type equation of the form considered in Theorem~\ref{thm:main}. Notably, since $K > 0$ and $\phi$ is smooth, the associated potential $V$ remains strictly positive on $\Sigma$.
	
	As an immediate consequence, applying Theorem~\ref{thm:main} with the substitutions $u = v + \phi$ and $V = K e^{-\phi}$, we obtain the following corollary, which also yields the Leray--Schauder degree for the geometric boundary condition problem \eqref{eq:Geo1}:
	
	\begin{cor}
		Under the assumptions of Theorem~\ref{thm:main}, the Leray--Schauder degree associated to \eqref{eq:Geo} is given by:
		\[
		d_\rho =
		\left\{
		\begin{aligned}
			&	1 && \text{ if }\, m = 1, 2, \\
			&	\frac{1}{\ell!} \prod_{i=1}^{\ell} (i - \chi(\Si)) && \text{ if }\, m \geq 3,
		\end{aligned}
		\right.
		\]
		where $\ell := [\frac{m - 1}{2} ]$ and $\chi(\Si)$ is the Euler characteristic of $\Sigma$.
	\end{cor}


	
	The remainder of the paper is organized as follows. In Section \ref{sec:2}, we perform a refined blow-up analysis for blowing-up solutions; in particular, we establish sharper pointwise estimates near the blow-up points. In Section \ref{sec:3}, we introduce a useful parametrization of the so-called \emph{neighborhood at infinity}, perform a finite-dimensional reduction,  derive precise expansions for the gradient and energy functional in this neighborhood, and  establish  estimates for $\rho_k- 4\pi m $. The main results  are proved in Section \ref{sec:4}. In the appendix, we provide some basic estimates used in the proofs of our main results, as well as detailed proofs that are omitted in Section \ref{sec:3}.


	
	%
	\newpage
	\noindent{\bf  Acknowledgements} 
	
	The first-named author is grateful to Professor Yanyan Li for his valuable comments on an earlier version of this paper, which greatly improved its presentation, and for suggesting the inclusion of the geometric boundary condition. The second-named author gratefully acknowledges Professor Thomas Bartsch for many insightful discussions.



	
	\section*{Notation}
	The main notations used throughout the paper are listed as follows: 
	\begin{itemize}
		\item 	$B(a,b)=\int_0^1 t^{a-1}(1-t)^{b-1} \, \d t$ is the Beta function for $a,b>0$. 
		\item  $C>0$ is a generic constant which can vary from different occurrences or equations. 
		\item   $C(\al, \beta, \ldots)$ means that the positive constant $C$ depends on $\al, \beta, \ldots$.
		\item	$\morse(f,a)$ represents the Morse index of the function $f$ at point $a$. 
		\item $\mathbb{N}=\{1,2,\ldots\}$ denotes the set of positive integers and 	$\mathbb{N}_0=\mathbb{N}\cup \{0\}$. 
		\item $\mathbb{R}_+=(0,+\infty)$ and $\R_+^2=\{ y=(y_1, y_2)\in \R^2: y_2\geq 0 \}$.	
		\item  The index $\ii(\xi)=2$ if $\xi\in \intS$ and $1$ if $\xi\in \pa\Si$.
		
		\item The binomial coefficient is given by $\binom{n}{k} = \frac{n!}{k!(n-k)!}$  for  integers   $0 \leq k \leq n$.
	\end{itemize}

	\section{Bubbling off  analysis}\label{sec:2}
	In this section, we carry out a refined blow-up analysis of the solutions to  \eqref{maineq}. To this end, we begin with several preliminary results.
	
	\subsection{The isothermal coordinates and Green's functions}\label{sec:2.1}
	Every Riemann surface $(\Si,g)$ is locally conformally flat,  and the local coordinates in which $g$ is conformal to the Euclidean metric are called isothermal coordinates (see, e.g., \cite{C1955, HW1955, V1955, L1957}). In this paper, we use a family of isothermal coordinates based on those applied in \cite{EF2014, YZ2021}, which map open neighborhoods in $\Si $ onto open disks or half-disks in $\R^2$. 	We first define 
	\begin{gather*}
		\B_r(y^*):=\{y=(y_1, y_2) \in \R^2:(y_1-y_1^*)^2+(y_2-y_2^*)^2<r^2\}, \\ \B_r:=\B_r(0),\quad  \text { and } \quad 
		\B_r^{+}:=\B_r \cap\{y=(y_1, y_2) \in \R^2: y_2 \geq 0\}.	
	\end{gather*}
	\begin{itemize}
		\item If  $\xi\in\intS$,  there exists an isothermal coordinate  $(U(\xi), y_{\xi})$ such that $y_{\xi}$ maps an open neighborhood $U(\xi)$ around $\xi$ onto an open disk $B^{\xi}:=\B_{2r_{\xi}}$  and transforms $g$ to $e^{\hat{\varphi}_{\xi}}\langle\cdot,\cdot\rangle_{\R^2}$ with $y_{\xi}(\xi)=(0,0)$ and $\overline{U(\xi)} \subset \intS$.
		\item 	If $\xi \in \pa \Si$, there exists an isothermal coordinate  $(U(\xi), y_{\xi})$ around $\xi$ such that the image of $y_{\xi}$ is a half disk $\B_{2r_{\xi}}^+$ and $y_{\xi}(U(\xi) \cap \pa \Si)=\B_{2r_{\xi}}^+ \cap \pa\R^2_+$ and transforms $g$ to $e^{\hat{\varphi}_{\xi}}\langle\cdot,\cdot\rangle_{\R^2}$ with $y_{\xi}(\xi)=(0,0)$.  In this case, we take $B^{\xi}:=\B_{2r_{\xi}}^+$.
	\end{itemize}
	In both cases, the conformal factor $\hat{\varphi}_{\xi}: B^{\xi} \to \R$  satisfies
	$	-\Delta \hat{\varphi}_{\xi}(y)=2 K_g(y_{\xi}^{-1}(y)) e^{\hat{\varphi}_{\xi}(y)}$ for all $ y \in B^{\xi},$
	where $K_g$ is the Gaussian curvature  of $\Si$.
	Additionally, when $\xi\in \pa \Si$, the boundary condition
	$
	\frac{\pa }{\pa y_2}\hat{\varphi}_{\xi}(y)=-2k_g(y)e^{\frac{\hat{\varphi}_{\xi}(y)}{2} } \quad \text { for all }\, y \in \pa B^{\xi}\cap \pa \R_+^{2}
	$
	is satisfied, where	 $k_g$ denotes the geodesic curvature of the boundary $\pa \Si$ with respect to the Riemannian  metric $g$. 
	Both $y_{\xi}$ and $\hat{\varphi}_{\xi}$ are assumed to depend  smoothly on $\xi$  in a local manner as in \cite{EF2014,F2024}. In addition, $\hat{\varphi}_{\xi}$ satisfies   
	\begin{equation}\label{varphixi}
		\hat{\varphi}_{\xi}(0)=0 \quad \text{ and } \quad \nabla\hat{\varphi}_{\xi}(0)=\begin{cases}
			0 & \text{ if }\, \xi\in \intS,\\
			(0,-2k_g(\xi)) & \text{ if }\, \xi\in \pa  \Si.
		\end{cases}
	\end{equation}
	Under a conformal change of metric $\tilde{g}=e^{\varphi} g$,   the Laplace-Beltrami operator transforms as
	\begin{equation}\label{conformalchange}
		\Delta_{\tilde{g}}=e^{-\varphi} \Delta_g.
	\end{equation}
	Moreover, the isothermal coordinates reserve the Neumann boundary conditions, i.e.,  for any $x \in$ $U(\xi)\cap \pa \Si$, we have
	\[	
	(y_{\xi})_*(\nu_g(x))= -\exp\Big( -\frac{\hat{\varphi}_{\xi}(y)}2\Big) \frac {\pa} { \pa y_2 }\Big|_{	y=y_{\xi}(x)},
	\]where $\nu_g(x)$ is the outward unit normal vector at $x$ with respect to the metric $g$.

	For $\xi \in \Si$ and $0<r \leq 2 r_{\xi}$, we define 
	\[	B_r^{\xi}:=B^{\xi} \cap\{y \in \R^2:|y|<r\}, \quad   U_r(\xi):=y_{\xi}^{-1}(B_r^{\xi}),\]
	and 
	$\cU_r(\xi):=\{x\in\Si:d_g(x,\xi)<0\},$
	where $d_g(x, \xi)$ denotes the geodesic distance between $x$ and $\xi$ with respect to the Riemannian metric $g$.

	For any $\xi\in \Si $, we define the Green's function for the problem \eqref{maineq} by following equations:
	\begin{equation}\label{Green}
		\left\{\begin{aligned}
			&-\Delta_g G^g(x,\xi)  =\delta_{\xi} -\frac{1}{|\Si |_g}, &&  x\in \intS,\\
			&\pa_{ \nu_g } G^g(x,\xi) =0, && x\in \pa \Si, \\
			&\int_{\Si } G^g(x,\xi) \,\d v_g(x) =0,&&
		\end{aligned}
		\right.
	\end{equation}
	where $\delta_{\xi}$ is the Dirac mass on $\Si$ concentrated at $\xi$.	
	We state several important properties of Green's functions (see \cite[Lemma 6]{YZ2021}):
	\begin{itemize}
		\item[a.] there exists a unique Green's function $G^g( \cdot,\xi) \in L^1(\Si )$ solves \eqref{Green} 
		in the distributional sense;
		\item [b.] for any  distinct points $x, \xi \in {\Si }$, $G^g(x, \xi)=G^g(\xi, x)$;
		\item [c.] the representation formula holds, for any $h\in C^2(\Si )$, 
		\begin{equation}
			\label{Greenrepresentation}
			h(x)- \frac 1 {|\Si |_g}\int_{\Si } h\, \d v_g= -\int_{\Si }G^g(\cdot,x)\Delta_g h \, \d v_g+ \int_{\pa \Si } G^g(\cdot,x) \pa_{\nu_g}h\, \d s_g, 
		\end{equation}
		where $\d s_g$ is the line element in $\pa \Si;$
		\item [d.]for any  $x, \xi \in {\Si}$ with $x\neq \xi$, there exists a constant $C>0$ such that
		\[
		|G^g(x, \xi)| \leq C(1+|\ln  d_g(x, \xi)|), \quad|\nabla_{g, \xi} G^g(x, \xi)| \leq Cd^{-1}_g(x, \xi),
		\]
	\end{itemize}

	Let $\chi$ be a radial cut-off function in  $C^{\infty}(\R,[0,1])$ such that
	\begin{equation}
		\label{eq:cut_off} \chi(s)=\begin{cases}
			1 &\text{ if }\,|s|\leq 1,\\
			0 & \text{ if }\, |s|\geq 2 .
		\end{cases}
	\end{equation}	
	For any fixed $\xi\in\Si$, define  
	\[
	\chi_{\xi}(x)=\chi\Big(\frac{4|y_{\xi}(x)|}{r_{\xi}}\Big)\quad \text{ and }\quad \Gamma^g(x,\xi)= \frac{4\chi_{\xi}(x)}{\kap ( \xi)} \ln \frac{1}{|y_{\xi}(x)|},\quad x\in\Si,
	\]
	where $\kap(\xi)=8\pi$ if $\xi\in\intS$ and $\kap(\xi)=4\pi$  if $\xi\in\pa\Si$.
	Decomposing  the Green's function $G^g( x,\xi)$ as 
	\[
	G^g(x, \xi)=\Gamma^g(x,\xi)+H^g(x, \xi),\quad x\in\Si.
	\]
	We deduce from \eqref{Green} that $H^g(x, \xi)$ solves the following equations:
	\begin{equation}\label{Hfunction}
		\left\{\begin{aligned}
			&-\Delta_g H^g(x, \xi)  =- \frac{4}{\kap( \xi)}(\Delta_g \chi_{\xi}) \ln  |y_{\xi}|-\frac{8}{\kap( \xi)}\langle\nabla \chi_{\xi}, \nabla \ln  |y_{\xi}|\rangle_g-\frac{1}{|\Si |_g} &&\text { in }\, \intS,\\ 
			&\pa_{\nu_g} H^g(x, \xi)=\frac{4}{\kap( \xi)}(\pa_{\nu_g} \chi_{\xi}) \ln  |y_{\xi}|+\frac{4}{\kap( \xi)} \chi_{\xi} \pa_{\nu_g} \ln  |y_{\xi}| &&\text { on }\, \pa \Si  , \\
			&\int_{\Si } H^g(x, \xi) \,\d v_g  =\frac{4}{\kap( \xi)} \int_{\Si } \chi_{\xi} \ln  |y_{\xi}| \, \d v_g.&& 
		\end{aligned}\right.
	\end{equation}
	By the regularity of the elliptic equations (see \cite{N2014,HBA20240829,A1959,W2004}, for instance), there is a unique  solution $H^g(x, \xi)$, which solves \eqref{Hfunction} in $C^{2, \al}(\Si)$.  $H^g(x, \xi)$ is the regular part of $G^g(x, \xi)$ and $R^g(\xi):=H^g(\xi, \xi)$ is so-called \emph{Robin function} on $\Si$. It is easy to check that $H^g(x, \xi)$ is independent of the choice of the cut-off function $\chi$ and the local chart.

	\subsection{A Kirchhoff-Routh type  function}
	
	In this subsection, we introduce a family of  Kirchhoff-Routh type functions, whose critical points determine the locations of the concentration points of the blowing-up solutions.\par

	For any $p,q\in \mathbb{N}_0$ with $2p+q\in\N$,   we define the thick diagonal set as
	\[
	\FF_{p,q}(\Si):=\{\xi=(\xi_1,\ldots,\xi_{p+q})\in (\intS)^p \times (\pa\Si)^{q}: \text{$\xi_i=\xi_j$ for some $ i\neq j$}\}
	\]
	and  
	\begin{equation}
		\label{def:Xi_pq} \Xi_{p,q}:=\intS^p \times (\pa\Si )^{q}\setminus \FF_{p,q}(\Si).
	\end{equation}
	In the sequel, all expressions involving $ \sum_i $, $ \max_i $, $ \min_i $, $ \bigcup_i $, etc., are understood to be taken over the index range $ i = 1, \ldots, p+q $. Similarly, $ \sum_{j \ne i} $ denotes summation over $ j $ in $ \{1, \ldots, p+q\} \setminus \{i\} $.
	For any given  $\xi:=(\xi_1,\ldots,\xi_{p+q})\in \Xi_{p,q}$,   we define  a family functions as follows: 
	\begin{equation}\label{eq:def_F}
		\cF_{\xi}^i(x):= V(x) \exp\Big\{\kap( \xi_i)H^g(x,\xi_i)+ \sum_{j\neq i} \kap( \xi_j) G^g(x,\xi_j)\Big\},
	\end{equation}
	for each $i=1,\ldots,p+q$ and the domain of $\cF_{\xi}^i$ is given by $\intS$ for $i=1,\ldots, p$ and  by $\partial\Sigma$ for $i=p+1,\ldots, p+q.$ Here,  $G^g(\cdot,\cdot)$ is the Green's function associated with problem \eqref{maineq}, and $H^g(\cdot,\cdot)$ is its regular part, as introduced in Section \ref{sec:2.1}. Furthermore, we define a Kirchhoff-Routh type function   $\ff:\Xi_{p,q}\to \R $  by  
	\begin{equation}\label{eq:reduce_fun_0}
		\ff(\xi):=\sum_{i} 2\kap( \xi_i)\ln  V(\xi_i) +\sum_{i} \kap^2( \xi_i) R^g(\xi_i) + \sum_{i}\sum_{j\neq i} \kap( \xi_i)\kap( \xi_j) G^g(\xi_i,\xi_j),
	\end{equation}
	where  $R^g(\xi_i):=H^g(\xi_i,\xi_i)$ is the Robin function; see also Section \ref{sec:2.1}.
	For $\xi\in \Xi_{p,q}$, we define
	\begin{equation}
		\label{def:l_xi}
		\cL(\xi):=\left\{\begin{aligned}
			&		\cL_1(\xi) & &\text{ if }\, q\neq 0,\\
			&\cL_2(\xi) & & \text{ if }\, q=0,
		\end{aligned}\right.
	\end{equation} 
	where
	\begin{equation}\label{eq:def_cL_1}
		\cL_1(\xi
		):=- \sum_{i=p+1}^{p+q} (\pa_{\nu_g}   \ln V +2k_g) \sqrt{\cF_{\xi}^i}\Big|_{\xi_i}	
	\end{equation} 
	and 
	\begin{equation}
		\label{eq:def_cL_2}
		\cL_2(\xi):= \sum_i 
		\kap(\xi_i)(\Delta_g\cF^i_{\xi}
		- 2K_g \cF^i_{\xi})\Big|_{\xi_i}.  
	\end{equation}
	For any $p,q\in \mathbb{N}_0$ with $2p+q\in\N$, we also define 
	\begin{equation}\label{eq:def_V-}
		\cV_{p,q}^-:=\{ \xi\in \Xi_{p,q}:  \text{$\xi$ is a critical point of $\ff$ with $\cL(\xi)<0$}\}
	\end{equation}
	and 
	\begin{equation}\label{eq:def_V+}
		\cV_{p,q}^+:=\{ \xi\in \Xi_{p,q}: \text{$\xi$ is a critical point of $\ff$ with $\cL(\xi)>0$}\}.
	\end{equation}
	
	Now we introduce the following assumptions on  $\ff$ and $\cL(\xi)$: 
	\begin{itemize}
		\item[(C1)]\label{item:C1} $\ff$ is a Morse function  with non-empty critical point set;
		\item[(C2)] \label{item:C2} $\cL(\xi)\neq 0$ for  any critical point $\xi$ of $\ff$. 
	\end{itemize}
	By perturbing the potential function $V$ or the Riemannian  metric $g$, conditions \hyperref[item:C1]{(C1)} and \hyperref[item:C2]{(C2)} can be satisfied by a transversality theorem; see \cite{Bartsch2019, HB2024, Ahmedou2023, LLWY2018} and references therein for details.
	


	\subsection{Projected bubbles}
	
	Bubbling-off solutions of $(P_{\rho})$ converge, after scaling, to specific entire solutions known as bubbles. Away from the blow-up points, these solutions converge to a linear combination of Green's functions with poles at the blow-up points. Hence, to obtain an accurate approximation of these solutions, one needs to use modified bubbles that incorporate their global behavior. \par 
	It is well known that  for any $(\lam, z)\in (0,+\infty)\times  \R^2 $, 
	\begin{equation}\label{bubble}
		\tilde{\delta}_{\lam,z}(y)=\ln  \frac{ 8\lam^2 }{(1 +\lam^2|y-z|^2)^2},\quad y\in \R^2
	\end{equation} 
	are  all the solutions to  the following Liouville equation: 
	\[
	\left\{\begin{aligned}
		&-\Delta u= e^u  &&\text{  in }\, \R^2,\\
		&	\int_{\R^2}  e^u <\infty. &&
	\end{aligned}\right.
	\]
	Moreover,  $\int_{\R^2}e^{\tilde{\delta}_{\lam,z}}\, \d y=8\pi $ for all  $(\lam, z)\in (0,+\infty)\times  \R^2 $;   see \cite{CL1991}.

	Using the isothermal coordinate $(y_{\xi}, U(\xi))$, we  pull-back $\tilde{\delta}_{\lam,0}$ to a function defined on the Riemann surface $\Si$, that is 
	\begin{equation}
		\label{def:bubble}\delta_{\lam,\xi}(x)=	\left\{\begin{aligned}
			&\tilde{\delta}_{\lam,0}(y_{\xi}(x)) &&\text{ if }\, x\in U(\xi),\\
			&	0 &&  \text{ if }\, x\in \Si\setminus U(\xi).
		\end{aligned}\right.
	\end{equation}
	From  \eqref{conformalchange} , it follows that $\delta_{\lam,\xi}$ satisfies 
	\begin{equation}\label{deltachange}
		-\Delta_g  \delta_{\lam,\xi}=e^{ \delta_{\lam,\xi}-\varphi_\xi}\quad \text{ in }\, U(\xi),
	\end{equation}
	where 	$\varphi_{\xi}=\hat{\varphi}_{\xi}\circ y_{\xi}^{-1}$. 
	For any $f\in L^1(\Si)$, we denote 
	$
	\overline{f}=\frac 1 {|\Si|_g}\int_\Si f \,\d v_g.
	$
	We  introduce a projection operator $P$  that maps $ \delta_{\lam,\xi}$ into the space $\oH$. More precisely, 
	$P  \delta_{\lam,\xi}$ is defined as the solution to the following  problem: 
	\begin{equation}\label{projectedbubble}
		\left\{\begin{aligned}
			&	-\Delta_g P  \delta_{\lam,\xi}=    \chi_{\xi} e^{\delta_{\lam,\xi}}- \overline{  \chi_{\xi} e^{\delta_{\lam,\xi}}  } &&  \text{ in }\, \intS,\\
			&	\pa_{ \nu_g } P  \delta_{\lam,\xi}=0 && \text{ on }\, \pa \Si , \\
			&\int_{\Si } P  \delta_{\lam,\xi} \,\d v_g=0. &&
		\end{aligned}\right.
	\end{equation}
	By classical elliptic  theory,  problem \eqref{projectedbubble}  admits a unique solution $ P\delta_{\lam,\xi}$ in $\oH$, and standard regularity theory implies that  $ P\delta_{\lam,\xi}\in C^{\infty}({\Si })$; see, e.g., \cite{A1959,YZ2021,N2014}. We refer to the function  $ P\delta_{\lam,\xi}$   as a \emph{projected bubble. }

	\subsection{Sharper estimates for blow-up solutions}
	Let $\{u_k\}\subset \oH$ be a sequence of blow-up solutions to  \eqref{maineq} with $\rho=\rho_k$, where   $\lim_{k\to\infty}\rho_k=  \rho_*:=4\pi m$ for some  $m\in\N$.
	The blow-up analysis in \cite[Lemma 3.1]{LSY2023} shows that $\{u_k\}$ blows up at  distinct points  $(\xi_1,\ldots,\xi_{p+q})\in \Xi_{p,q}$ for some $p,q\in \N_0$ satisfying $2p+q=m$.  Moreover, we have  
	\[
	\rho_k V e^{u_k} \, \d v_g \stackrel{*}{\rightharpoonup}  \sum_i \kap(\xi_i) \delta_{\xi_i}\quad\text{(weak-$*$ convergence)},
	\]
	and 
	$\delta_{\xi_i}$ denotes the Dirac measure  at $\xi_i\in\Si$.

	Let  
	\begin{equation}\label{def:w_k}
		w_k:=u_k-\ln \Big(\int_\Si  V u_k \, \d v_g\Big).
	\end{equation}A direct computation shows that $w_k$ solves the following equation
	\[
	\left\{
	\begin{aligned}
		&		-\Delta_g w_k = \rho_k( V e^{w_k} -1) && \text{ in }\, \intS,\\
		&		\pa _{\nu_g}  w_k=0 && \text{ on }\, \pa \Si,
	\end{aligned}
	\right.
	\]
	and satisfies   $\int_\Si V e^{w_k}\, \d v_g=1$. By applying the  blow-up analysis in \cite[Lemma 3.1]{LSY2023}, we  conclude that $w_k(x)\to -\infty$ uniformly in any compact set of $\Si\setminus\{\xi_1,\ldots,\xi_{p+q}\}$, and   the local maximum satisfies
	\begin{equation}\label{def:h_kj}
		h_{k,j}:= \sup_{ U_{r_0}(\xi_j)} w_k =w_k(\xi_{k,j})\to +\infty
	\end{equation}
	as $k\to +\infty$, where  $r_0>0$ is sufficiently small,  and $\xi_{k,j}$ denotes the local maximum point of $w_k$  satisfying $\xi_{k,j}\to  \xi_j$ as $k\to +\infty$.   
	Let  $
	h_{k}=\max_{x\in \Si} w_k(x)$ for $ k\in  \N.$
	Then, for $k\in  \N$ sufficiently large, it follows that $h_{k}=\max_{j=1,\ldots, p+q} h_{k,j}$. 
	
	Using the method of moving planes and a $\sup +  \inf$ inequality from H.  Brezis, Y.Y. Li and I. Shafrir\cite{BLS1993},  Y.Y. Li\cite{Li1999} derived a uniform estimate for the behavior of $ w_k $ near the blow-up point $ \xi_j $ on closed surfaces or domains with Dirichlet boundary conditions. In our setting, when $ \xi_j \in \intS $, the analysis coincides with the Dirichlet case. When $ \xi_j \in \pa  \Si $, similar treatment is possible by employing isothermal coordinates and the reflection method, following the approaches in \cite{LSY2023, YZ2021}. Thus, we have the following lemma as in \cite[Theorem 2.1]{CL2002}. 
	
	\begin{lem}\label{lem:w_k_1}
		Let $r_0>0$ be sufficiently small. Then there exists a constant $C>0$, independent of $j$ and  $k$,  such that 
		\[	
		\Big|w_k(x) - \ln  \frac{ e^{h_{k,j}}}{( 1+ \frac{ \rho_k V(\xi_{k,j}) e^{h_{k,j}}}{8} |y_{\xi_{j}}(x)|^2)^2} \Big|\leq C
		\]
		for all  $x\in U_{r_0}(\xi_j)$,  where $ (y_{\xi_j}, U_{r_0}(\xi_j))$ is an isothermal coordinate around $\xi_{j}$ with $y_{\xi_{j}}(\xi_{j})=0$. 
	\end{lem}
	\begin{proof}
		Here, we provide a sketch of the proof for the case  $ \xi_j \in \pa  \Si$. The argument for the interior case  $\xi_j\in\intS$ can be found in  \cite[Theorem 0.3]{Li1999}.
		
		Let  $( y_{\xi_j}, U(\xi_j))$ be  the isothermal coordinates with  $y_{\xi_j}: U(\xi_j)\to  \B^+_{2r_{\xi_j}}$. Define
		$ \tilde{w}_{k,j}(y):= w_k\circ y_{\xi_j}^{-1}(y)$.  Then, according to  \eqref{conformalchange},   $\tilde{w}_{k,j}$ satisfies
		\[
		\left\{
		\begin{aligned}
			&	-\Delta \tilde{w}_{k,j}= \rho_k (V( y_{\xi_j}^{-1}(y))   e^{\tilde{w}_{k,j}} - 1 )e^{\hat{\varphi}_{\xi_j}}  && \text{ in }\, \B_{r_0}^+,\\
			&	\frac{\pa }{\pa  y_2} \tilde{w}_{k,j}=0 && \text{ on }\, \B_{r_0}^+\cap \pa \R_+^2,
		\end{aligned}
		\right.
		\]
		where $0<r_0\leq 2 r_{\xi_j}$.  Using the reflection along the $y_2$-axis, we extend $\tilde{w}_{k,j}$ to be defined on the open disk $\B_{r_0}$. 
		Let $\tilde{\zeta}_{k,j}$ be the unique solution of the following problem: 
		\[
		\left\{
		\begin{aligned}
			&	-\Delta \tilde{\zeta}_{k,j} = \tilde{f}_{k,j} && \text{ in } \,\B_{r_0},\\
			&	\tilde{\zeta}_{k,j}=0 && \text{ on }\, \pa  \B_{r_0} ,
		\end{aligned}
		\right.
		\]
		where 
		\[	
		\tilde{f}_{k,j}(y_1, y_2):=\left\{
		\begin{aligned}
			&\rho_k e^{\hat{\varphi}_{\xi_j}(y_1, y_2)} &&  \text{ in } \,\B_{r_0}^+,\\
			&\rho_k e^{\hat{\varphi}_{\xi_j}(y_1, -y_2)}  && \text{ in }\, \B_{r_0}\cap \{ y_2<0\}.
		\end{aligned}
		\right.
		\]
		By standard elliptic estimates, there exists a constant  $C>0$ independent of $j$ and $k$, such that $\|\tilde{\zeta}_{k,j}\|_{L^{\infty}(\B_{r_0})}\leq C$. 
		Let $\tilde{u}_k:= \tilde{w}_{k,j}+\tilde{\zeta}_{k,j}$. Then, we have 
		\[	
		-\Delta \tilde{u}_k= \rho_k \tilde{V}_{k} e^{\tilde{u}_k} \quad  \text{ in }\,\B_{r_0},
		\]
		where
		\[	  
		\tilde{V}_{k}(y_1, y_2):= \left\{	\begin{aligned}
			&	V( y_{\xi_j}^{-1}(y_1, y_2)) e^{\hat{\varphi}_{\xi_j}(y_1, y_2)} e^{-\tilde{\zeta}_{k,j}(y_1, y_2)} && \text{ in }\, \B_{r_0}^+,\\
			&	V( y_{\xi_j}^{-1}(y_1, -y_2)) e^{\hat{\varphi}_{\xi_j}(y_1, -y_2)} e^{-\tilde{\zeta}_{k,j}(y_1, y_2)} && \text{ in } \B_{r_0}\cap \{ y_2<0\}.
		\end{aligned}\right.
		\]
		By \cite[Theorem 0.3]{Li1999}, we obtain
		\[ \Big| \tilde{u}_k - \ln \frac{ e^{h_{k,j}} }{ (1+ \frac{ \rho_k V(\xi_{k,j}) }{8} e^{h_{k,j}} |y|^2)^2} \Big| \leq C\quad  \text{ in }\,  \B_{r_0}. \]
		Lemma \ref{lem:w_k_1} then follows immediately.
	\end{proof}
	Due to \cite{Battaglia2019},  we also have the following estimate when $x$ is away from blow-up points: 
	\begin{lem}\label{lem:w_k_2} As $k\to +\infty$,  we have
		$
		w_k-\overline{w}_k\to  \sum_j \kap(\xi_j) G^g(\cdot,\xi_j)
		$ in  $C^2(\Si\setminus\{ \xi_1,\ldots,\xi_{p+q}\})$,  where   
		$G^g(\cdot,\cdot)$ is the Green's function defined in  \eqref{Green}. 
	\end{lem}

	\begin{cor} \label{cor:rates}For any $j=1,\ldots, p+q$ and $\delta_0>0$ sufficiently small, we have 
		\[	|h_{k,j}-h_{k}|\leq C \quad \text{ and }\quad  | w_k +h_{k}| \leq C\quad  \text{ in } \, \Si\setminus\cup_j \cU_{\delta_0}(\xi_j). \]
	\end{cor}
	\begin{proof}
		It follows from 	Lemma \ref{lem:w_k_2} that 
		$|w_k(x) -\overline{w}_k|\leq C$ for $x\in \pa  \cU_{\frac 1 2 \delta_0}(\xi_j)$
		with  $\delta_0>0$. 
		On the other hand, Lemma \ref{lem:w_k_1} yields that  for $x\in\pa  \cU_{\frac 1 2 \delta_0}(\xi_j)$, 
		$ |w_k(x)+h_{k,j}|\leq C$. This estimate, together with the triangle inequality, implies that 
		$ |h_{k,j}+\overline{w}_k |\leq C$ and $|h_{k,j}-h_{k}|\leq C$   for  $x\in \pa  \cU_{\frac 1 2 \delta_0}(\xi_j).$  Moreover, 
		for any $x\in \Si\setminus\cup_j \cU_{\delta_0}(\xi_j) $, 
		$ |w_k(x)+h_{k,j}|\leq C$. 
	\end{proof}

	For $r_0>0$ sufficiently small, we  define the local mass of 
	$w_k$ at $\xi_j$ as 
	$
	\rho_{k,j}:= \rho_k \int_{\cU_{r_0}(\xi_j)} V e^{w_k} \, \d v_g.
	$
	It follows from Corollary \ref{cor:rates} that
	\begin{equation}\label{eq:est_rhok_rho_kj}
		\rho_k -\sum_{j} \rho_{k,j}=\rho_k  \int_{\Si\setminus \cup_j \cU_{r_0}(\xi_j)} V e^{w_k} \, \d v_g = O( e^{-h_{k}}).
	\end{equation}

	We have the following estimate for $h_{k,j}$: 
	\begin{lem}\label{lem:lambda_kj}
		For  any $j=1,\ldots, p+q$,  the following estimate holds as   $k\to +\infty$:
		\[	
		\Big| h_{k,j} + \overline{w}_k + F^*_{k,j}(\xi_{k,j})+ 2 \ln\Big( \frac{\rho_k V(\xi_{k,j})}{8}\Big)\Big|=O( e^{-\frac 1 2 h_{k,j}} ),
		\]
		where
		\begin{equation}
			\label{eq:F_star} F^*_{k,j}(x):= \rho_{k,j} H^g(x, \xi_{k,j}) + \sum_{ i\neq j} \rho_{k,i} G^g(x,\xi_{k,i}),
		\end{equation}
		and $H^g(\cdot,\cdot)$ is the regular part of the Green's function $G^g(\cdot,\cdot)$ as defined in  \eqref{Hfunction}. 
	\end{lem}
	\begin{proof}
		The proof is similar to \cite[Estimate D]{CL2002}, with a minor modification involving the boundary blow-up points, which we omit here.
	\end{proof}

	For any $j=1,\ldots, p+q$ and $r_0>0$ sufficiently small, we define 
	\[
	\eta_{k,j}(x):= w_{k}(x) - ( F^*_{k,j}(x)- F^*_{k,j}(\xi_{k,j}))\quad  \text{ in }\, \cU_{r_0}(\xi_j).
	\]
	\begin{lem}\label{lem:refine_wk}
		There exists some  constant $C>0$, independent of $j$ and  $k$,  such that 
		\[ |\nabla (\eta_{k,j}- \delta_{\lam_{k,j},\xi_{k,j}})(\xi_{k,j})|_g \leq C, \]
		where $\delta_{\lam,\xi}$ is defined in  \eqref{def:bubble}.
	\end{lem}

	\begin{proof}
		We just consider the case $\xi_j\in\pa\Si$,  as the interior case 
		$\xi_j\in\intS$ can be found in \cite[Lemma 4.1]{CL2002}. 
		
		Let $\tilde{\eta}_{j}:=\eta_{k,j}\circ y_{\xi_j}^{-1}(y)$. 	For some $r_0>0$ sufficiently small, $\tilde{\eta}_{j}$ satisfies
		\[
		\Delta \tilde{\eta}_{j}+ \rho_k \tilde{V}_{j} e^{\tilde{\eta}_{j}} =\Big(\rho_k -\sum_{j} \rho_{k,j} \Big) e^{\hat{\varphi}_{\xi_j}}= O( e^{-h_k})\quad \text{ in }\, B_{r_0}^{\xi_j}
		\] 
		in view of \eqref{eq:est_rhok_rho_kj},	where 	$
		\tilde{V}_j(y)=V\circ y_{\xi_j}^{-1}(y) e^{\hat{\varphi}_{\xi_j}(y)} e^{ F^*_{k,j}\circ y_{\xi_j}^{-1}(y) - F^*_{k,j}( \xi_{k,j})}.
		$
		Denote $
		\tilde{\delta}_{k,j}(y):=\delta_{\lam_{k,j},\xi_{k,j}}\circ y_{\xi_{j}}^{-1}(y)-\ln(\rho_k V(\xi_{k,j})),
		$
		where  $\lam_{k,j}=\sqrt{ \frac{ \rho_k V(\xi_{k,j}) e^{h_{k,j}}}{8}}\to +\infty$ as $k\to +\infty$.  
		It follows that   $\tilde{\epsilon}_{k,j}(y):=(\tilde{\eta}_j-\tilde{\delta}_{k,j})( \frac{y}{\lam_{k,j}})$ solves 
		\[ \Delta \tilde{\epsilon}_{k,j}+\frac{ 8  }{(1+|y|^2)^2} \Big( \frac{\tilde{\tilde{V}}_j e^{\tilde{\epsilon}_{k,j}}}{\tilde{\tilde{V}}_j (z_{k,j})} -1\Big)=O(e^{-h_k})\quad  \text{ in }\,  \lam_{k,j} B_{r_0}^{\xi_j}, \]
		where  $\tilde{\tilde{V}}_j(y)= \tilde{V}_j(y/\lam_{k,j} ) $ and  $z_{k,j}= \lam_{k,j} y_{\xi_j}^{-1}(\xi_{k,j})$.
		Since $\xi_j\in \pa \Si$, we have $\lam_{k,j} B_{r_0}^{\xi_j}= \B_{\lam_{k,j}r_0}^+$ and $\pa _{y_2} \tilde{\epsilon}_{k,j}=0$ on $\B_{\lam_{k,j}r_0}^+\cap \pa \R_+^2$.  
		Using the reflection technique as in Lemma \ref{lem:w_k_1}, $\tilde{\epsilon}_{k,j}$ can be extended to the open disk $\B_{\lam_{k,j}r_0}$. For simplicity, we continue to denote the extended function by $\tilde{\epsilon}_{k,j}$.
		
		Using the gradient estimate in \cite[Lemma 4.1]{CL2002},  for $z_{k,j}= y_{\xi_j}(\xi_{k,j})$ we have
		\[ 
		|\nabla \tilde{\epsilon}_{k,j}(z)|\leq C\Big( \frac{e^{-\lam_{k,j}}}{ 1+|z-z_{k,j}|}+ \frac{\sup_{\frac 1 2 |z-z_{k,j}|\leq |y|\leq 2 |z-z_{k,j}|}  |\tilde{\epsilon}_{k,j}(y)|}{1+|z-z_{k,j}|} \Big)
		\]
		in a small neighborhood of $z_{k,j}$. 
		Hence,	Lemma \ref{lem:refine_wk} is concluded. 
	\end{proof}
	\begin{lem}\label{lem:est_wk2}
		For any $j=1,\ldots, p+q$, there exist  $C>0$ and $r_0>0$ such that  as $x\to  \xi_{k,j}$,
		\[ \left| w_k(x) - \delta_{\lam_{k,j},\xi_{k,j}}(x)-\ln( \rho_k  V(\xi_{k,j})) \right| \leq C d_g(x,\xi_{k,j}),\]
		in $U_{r_0}(\xi_{k,j})$.  
	\end{lem}
	\begin{proof}
		The proof follows from Lemma \ref{lem:refine_wk} immediately. 
	\end{proof}

	Let 
	\[ \psi(\alpha,\xi,\lambda):=\sum_{i} \al_i P\delta_{\xi_i,\lam_{i}},\]
	for $\al:=(\al_1,\ldots,\al_{p+q})\in (\R_+)^{p+q}$, $\xi:=(\xi_1,\ldots,\xi_{p+q})\in \Xi_{p,q}$ and $\lam:=(\lam_1,\ldots,\lam_{p+q})\in (\R_+)^{p+q}$. 
	For any $\var>0$, we define a neighborhood of potential critical points at infinity. Given $\rho>0$, $p,q\in \N_0$ with $2p+q\in\N$, and  $\eta>0$, we set
	\begin{align}
		V_\rho^{\eta}(p,q,\var): =\Big\{ u\in \oH: &\,\|\nabla J_{\rho }(u)\|< \var, \exists\, \lam=(\lam_1,\ldots,\lam_{p+q})\in (\R_+)^{p+q} \text{ with } \lam_1,\ldots,\lam_{p+q} > \var^{-1},\nonumber\\
		& \text{ and }\exists\, \xi=(\xi_1,\ldots, \xi_{p+q})\in \Xi_{p,q}\text{ with } \min _{\mytop{i,j=1,\ldots,p+q}{i\neq j} }d_g(\xi_i, \xi_j )> \eta\nonumber \\
		&
		\text{ and } \min_{i=1,\ldots, p} d_g(\xi_i,\pa \Si)>\frac \eta 2, \text{ s.t. }   \| u-\psi(e,\xi,\lam)\|<\var \Big\}, \label{def:Vm}
	\end{align}
	where $e=(1,\ldots, 1)\in \R^{p+q}$. 

	\begin{prop}\label{prop:u_k_in_V^eta}
		Let $\{u_k\}$ be a sequence of blow-up solutions to \eqref{maineq} with $\rho=\rho_k$.  Assume that  $\lim_{k\to+\infty}\rho_k=  \rho_*:=4\pi m$,  and  the blow-up set of $\{u_k\}$ is $\{ \xi^*_1,\ldots,\xi^*_{p+q}\}$ with $(\xi^*_1,\ldots,\xi^*_{p+q})\in \Xi_{p,q}$ up to a permutation, where $p,q\in\N_0$ and $m=2p+q\in \N$.  Then there exist constants  $ \var_0 > 0 $ and $ \eta_0 > 0 $ such that, for any $ \var\in (0,\varepsilon_0) $, for some $k_0\in \N$ we have  $$u_k \in V_{\rho_*}^{\eta_0}(p, q, \var),\quad  \forall   k \geq k_0.$$
	\end{prop} 
	
	\begin{proof}
		Let $w_k$ be defined as in \eqref{def:w_k}. Using the blow-up analysis in \cite{LSY2023}, there exists $r_0>0$ such that $\cU_{r_0}(\xi_i)\cap \cU_{r_0}(\xi_j)=\emptyset$ for  $i\neq j$, and $ \cU_{r_0}(\xi_i)\cap \pa \Si=\emptyset$ for  $i=1,\ldots,p$.
		Let $h_{k,j}$ and $\xi_{k,j}$ be defined as in \eqref{def:h_kj}. Then, $h_{kj}=w_k(\xi_{k,j})\to +\infty$ and $\xi_{k,j}\to  \xi_j$ as $k\to +\infty$.
		Let 
		\[ \lam_{k,j}:= \sqrt{ \frac{ \rho_k V(\xi_{k,j} )}{8}} e^{\frac 1 2 h_{k,j}}.\]
		By Corollary \ref{cor:rates}, it follows that
		\begin{equation}\label{eq:w_k_out}
			e^{w_k(x)} =O( e^{h_{k,j}})= O\Big( \frac 1 {\lam_{k,j}^2}\Big)\quad  \text{ in } \, \Si \setminus\bigcup_j \cU_{ r_0}(\xi_j).
		\end{equation} 
		Define 
		\[ \epsilon_k:= w_k - \sum_j P\delta_{\lam_{k,j},\xi_{k,j}}.\] Denote $P\delta_{k,j}:=P \delta_{\lam_{k,j},\xi_{k,j}}$, $\delta_{k,j}= \delta_{\lam_{k,j},\xi_{k,j}}$, $ \chi_j= \chi( 4|y_{\xi_{k,j}}|/ r_0) $, $\hat{\varphi}_j: = \hat{\varphi}_{\xi_{k,j}}$ and $\varphi_j :=\hat{\varphi}_{\xi_{k,j}}\circ y_{\xi_{k,j}}$. Then, $\epsilon_k$ satisfies the equation
		\[	
		-\Delta_g \epsilon_k= \rho_k V e^{w_k} -\sum_{j} \chi_j e^{\delta_{k,j}} -\Big(\rho_k- \sum_{j} \overline{ \chi_j  e^{\delta_{k,j}}}\Big)   :=f_k.
		\]
		Observe that 
		\begin{equation}\label{eq:kappa_i}
			\begin{split}
				\rho_k-\sum_j \int_\Si\chi_je^{\delta_{k,j}}\, \d v_g& =\rho_k -\sum_{j}\rho_{k,j}+ \sum_j \int_{\Sigma} (V e^{w_k} -\chi_j e^{\delta_{k,j}} ) \, \d v_g
				\\ &=O\Big(  \sum_j \int_{\cU_{ r_0}(\xi_{k,j})} e^{\delta_{k,j}} d_g(x,\xi_{k,j})\, \d v_g+ \sum_j \frac 1 {\lam^2_{k,j}} \Big) =O\Big(\sum_j \frac 1 {\lam_{k,j}}\Big),
			\end{split}
		\end{equation}
		from Lemma \ref{lem:est_wk2} and \eqref{eq:est_rhok_rho_kj}. 
		Combining \eqref{eq:w_k_out} with Lemma \ref{lem:est_wk2} and \eqref{eq:kappa_i}, we have  
		\[ f_k(x)= O\Big(\sum_j  e^{\delta_{k,j}} d_g(x,\xi_{k,j}) \mathbbm{1}_{ \cU_{ r_0}(\xi_{j})} + \sum_{j} \frac 1 {\lam_{k,j}}\Big),\]
		where $\mathbbm{1}_A$ is the characteristic function of a  set $A\subset \Si$, defined by $\mathbbm{1}_A(x)= \begin{cases}
			1 &  x\in A \\
			0 & \text{otherwise}
		\end{cases}$. 
		By straightforward calculation,  for $p\in (1,2)$, it holds
		$ 
		\|f_k\|_p^p= O( \sum_{j} \frac{1}{\lam_{k,j}^{2-p}}).
		$
		Using the regularity theory (see \cite{YZ2021,A1959}) together with Corollary \ref{cor:rates},  we obtain for $p\in (1,2)$
		\[ \|\epsilon_k-\bar{\epsilon}_k\|_{W^{2,p}(\Si)}\leq \| f_k\|_p= O\Big( \sum_{j} \frac{1}{\lam_{k,j}^{2/p-1}}\Big).\]
		In particular, choosing
		$p=\frac 87 $, the Sobolev embedding theorem $ W^{2,p}(\Sigma)\hookrightarrow H^1(\Sigma)$ yields that 
		\[ \|\epsilon_k-\bar{\epsilon}_k\|_{H^1(\Si)}= O\Big( \sum_{j} \frac{1}{\lam_{k,j}^{2/p-1}}\Big) =O\Big( \sum_j \frac 1 {\lam_{k,j}^{3/4}}\Big).\] 
		Note that by definition, we have $\epsilon_k-\overline{\epsilon}_k= u_k - \sum_{j}P\delta_{k,j}$. 
		Since $u_k$ solves \eqref{maineq} with  $\rho=\rho_k$,  we have $\|\nabla J_{\rho_k}(u_k) \|=0$, and it follows that $\|\nabla J_{\rho_*}(u_k) \| =O( |\rho_k-\rho_*|)$.  
		Therefore, for any fixed  $\var>0$. there exist constants $k_0>0$ sufficiently large  and $\eta_0>0$ sufficiently small such that $u_k \in V_{\rho_*}^{\eta_0}(p,q,\var)$ for all $k\geq k_0$. 
	\end{proof}
	\section{Finite dimensional reduction}\label{sec:3}
	In this section, we begin by parametrizing the neighborhood at infinity to perform a finite dimensional reduction. We then perform asymptotic expansions of both the gradient and the energy functional in this neighborhood. Additionally, we also derive estimates of $\rho_k-4\pi m $ for any sequence of blow-up solutions of \eqref{maineq} with parameter $\rho=\rho_k\to 4\pi m$ as $k\to +\infty$.
	
	\subsection{Parametrization of the neighborhood at infinity}\label{sec:3.1}
	For any integers $p,q\in \N_0$ with $m=2p+q\in  \N$,   $\var>0$ and a constant $\eta>0$, $V^\eta_{\rho_*}(p,q,\var)$ defined in  \eqref{def:Vm}       is  the sequel neighborhood of potential critical
	points at infinity for $\rho_*=4\pi m$ with $m\in \N$. 
	
	Given that  $p,q\in \N_0$ with $2p+q\in  \N$.   For any $\var>0$ and $\eta>0$,  we define the set $\tilde{B}^\eta_{\var}$ as the collection of  $(\al,\xi,\lam)\in(\R_+)^{p+q}\times \Xi_{p,q}\times(\R_+)^{p+q} $:
	\begin{align}
		\tilde{B}^\eta_{\var}:=\Big\{(\al,\xi,\lam):&\,\al=(\al_1,\ldots,\al_{p+q}),\,|\al_i-1|<\frac{1}{\eta},\,\forall\, i=1,\ldots, p+q; \xi=(\xi_1,\ldots,\xi_{p+q})\in\Xi_{p,q},\notag\\&\, d_g(\xi_i,\xi_j)>\eta,\,\forall \,i\neq j,\text{ and }d_{g}(\xi_i,\pa\Si)>\frac{\eta}{2},\forall \, i=1,\ldots, p;\notag\\&\,\lam=(\lam_1,\ldots,\lam_{p+q}),\lam_i>\var^{-1},\forall\,i=1,\ldots, p+q\Big\}.\label{def:ballB_var^eta}	
	\end{align}
	
	If $u$ is a function in $V_{\rho_*}^\eta(p,q,\var)$, one can find an optimal representation, following the ideas introduced in \cite{BC1, BC2}. Namely, we have:
	\begin{prop}\label{prop:unique}
		There exists $\var_0>0$	such that,  for any $\var\in (0,\var_0)$, if $u\in V_{\rho_*}^\eta(p,q,\var)$, 
		then the following minimization problem
		\begin{equation}
			\label{eq:min}
			\min_{(\al,\xi,\lam)\in \bigcup_{R>0}\tilde{B}^\eta_{4R}} \|u-\psi(\al,\xi,\lam)\|
		\end{equation}
		admits a unique solution $(\al , \xi, \lam)$ up to a permutation.  Moreover, the minimizer is achieved in $\tilde{B}^\eta_{2\var}$. In particular, $u$ admits a unique decomposition of the form as follows:
		\[
		u=\sum_i \al _i P\delta_{\lam_i,\xi_i}+w,
		\] where $w\in \oH$  satisfies the orthogonality conditions
		\begin{align}\label{orthogonality}
			\la w, P\delta_{\lam_i,\xi_i}\ra =
			\Big\la w, \frac{\pa P\delta_{\lam_i,\xi_i}}{\pa \lam_i}\Big\ra	 =	\Big\la w,\frac{\pa P\delta_{\lam_i,\xi_i}}{\pa (\xi_i)_j} \Big\ra=0
		\end{align}
		for any $i\in\{1,\ldots,p+q\}$ and $j\in\{1,\ii(\xi_i)\}$.	In addition, the variables $\{\alpha_i\}_{i=1}^{p+q}$ satisfy
		\begin{equation}\label{eq:apha_i_estimate}
			\max_{i}|\al _i-1|=o\Big(\sum_i \frac{1}{\sqrt{\ln \lam_i}}\Big)=o\Big(\frac 1 {\sqrt{|\ln\var|}}\Big)
		\end{equation}
		as $\var\to  0^+$.
	\end{prop}
	\begin{proof}
		The complete proof can be found in Appendix \ref{app:B}.
	\end{proof}
	Given  $\xi=(\xi_1, \ldots, \xi_{p+q})\in\Si^{p+q}$ and $\lam=(\lam_1, \ldots, \lam_{p+q})\in(\R_+)^{p+q}$, we further define 
	\[
	E_{\lam,\xi}^{p,q}:=\{w \in \oH:  \text {$w$ satisfies  \eqref{orthogonality} and $\|w\|\leq C \var$}\}
	\]
	for some universal constant $C>0$.  This set will be used throughout the subsequent analysis.
	
	\subsection{Expansions of the gradient  in the neighborhood at infinity}\label{sec:3.2}
	To analyze the behavior  of the functional $J_{\rho}  $ in the set $V^{\eta}_{\rho}(p,q,\var) $, one needs to  compute the  derivatives of $J_{\rho}$ with respect to the variables $(\al,\xi,\lam)$. To maintain clarity of exposition, we provide the technical details of the proof in Appendix \ref{app:C}.
	
	In this section, we   assume that  $p,q\in \N_0$ with  $m:=2p+q\in\N$, $\rho_*:=4\pi m$, $\eta>0$, $\mu\in \R$  is a  small  constant  and $\var>0$ is sufficiently small.  In addition, We use the standard asymptotic notation: 
	$ g = O(f)$ means that $ |g/f| \leq C$ for some constant $ C > 0$, while $ g = o(f)$ means that $ g/f \to 0$.
	
	Recall that 
	\[ 	\cF_{\xi}^i(x):= V(x) \exp\Big\{\kap( \xi_i)H^g(x,\xi_i)+ \sum_{j\neq i} \kap( \xi_j) G^g(x,\xi_j)\Big\},\quad \text{ for }x\in\Si. \]
	We also introduce the following auxiliary functions for given $\xi=(\xi_1,\ldots, \xi_{p+q})\in \Xi_{p,q}$ and $i=1,\ldots, p+q$:
	\begin{equation}
		\label{def:rg}
		\rg_{\xi}^i(x) := \exp\Big\{(\al_i-1) \kap( \xi_i)H^g(x,\xi_i)+ \sum_{j\neq i}(\al_j-1) \kap( \xi_j) G^g(x,\xi_j) \Big\},
	\end{equation}
	and 
	\begin{equation}
		\label{def:frak_f2} 	\mathfrak{f}_2(x):=\left\{
		\begin{array}{ll}
			0& \text{ if }\,x\in \intS,\\
			-(\pa_{\nu_g}  \ln  (\cF^i_{\xi}\rg^i_{\xi})+2k_g )\cF^i_{\xi}\rg^i_{\xi}
			\Big|_{x}& \text{  if }\,	x\in \pa\Si.
		\end{array}
		\right.
	\end{equation} 	
	
	\begin{prop}\label{prop:lamexpansion}
		Let $u=u_0+w \in V_{\rho}^\eta(p,q,\var)$ with $\rho=(1+\mu)\rho_*$, where   $u_0=\sum_{i} \al_i P\delta_{\lam_i,\xi_i}$ and  $w\in E_{\lam,\xi }^{p,q}$. Then, for each $i=1,\ldots,p+q$,  the following expansion holds as  $\var\to  0^+$:
		\begin{align}
			\Big\langle \nabla J_{\rho }(u),\lam_i\frac{\pa P\delta_{\lam_i,\xi_i}}{\pa \lam_i}\Big\rangle =&\,2\al _i \kap(\xi_i)   (\tau_i-\mu+\mu\tau_i)-\frac \pi 2   \frac{\rho\mathfrak{f}_2(\xi_i)  }{\int_{\Si} V e^u \, \d v_g} \lam_i^{4 \al _i-3}\nonumber\\
			&\,+\mu\rho_*\frac{\kap(\xi_i)}{|\Si|_g}\frac{\ln\lam_i}{\lam_i^2}- (1+\mu)(1-\tau_i)\sum_{j}\frac{\kap(\xi_i)\kap(\xi_j)}{|\Si|_g}\frac{\ln\lam_j}{\lam_j^2} \nonumber\\
			&\,+O\Big(\sum_j |\al _j-1|^2+\|w\|^2 \Big)+o\Big(\sum_j\frac{\ln\lam_j}{\lam_j^2}\Big),\label{prop:lamexpansion-1}
		\end{align}
		where 
		\begin{gather}
			\label{def:tau_i}
			\tau_i:= 1-\frac{\rho_*}{8(2 \al _i-1)} \frac{(\mathcal{F}_{\xi}^i
				\mathrm{g}_{\xi}^i )(\xi_i)}{\int_{\Si} V e^u\, \d v_g} \lam_i^{4 \al _i-2},
		\end{gather}
		and $\cF_{\xi}^i,\rg_{\xi}^i$ and $ \mathfrak{f}_2$ are given by \eqref{eq:def_F}, \eqref{def:rg} and \eqref{def:frak_f2}, respectively.
		Moreover, we have
		\begin{equation}
			\label{eq:est_tau_i}
			|\tau_i|=O\Big(\var+|\mu|+\sum_j |\al _j-1|^2\Big).
		\end{equation} 
	\end{prop}
	\begin{proof}
		The proof is provided in  Appendix \ref{app:C}.
	\end{proof}

	From \eqref{def:tau_i} and \eqref{eq:est_tau_i},  there exists a constant $C>0$ such that 
	\[	
	\frac{1}{C}<\frac{(\cF_{\xi}^i g^i_\xi)(\xi_i) \lam_i^{4\al_i-2}}{(\cF_{\xi}^j g^j_\xi)(\xi_j) \lam_j^{4\al_j-2}}<C,\quad \forall\, i,j=1,\ldots,p+q.
	\]
	In particular, under stronger assumptions, the blow-up scales become asymptotically synchronized:
	\begin{cor}\label{cor:lam_i_samespeed}
		Under the same assumptions as in Proposition \ref{prop:lamexpansion}, if, in addition, we assume that $|\mu |\leq C\var$ for some constant $C>0$ and $\sum_s|\al_s-1|\ln \lam_s=o(1)$, then  for any $i,j=1,\ldots, p+q$, the following estimate holds:
		\begin{equation}\label{eq:est_lam_i_j}
			\left|1- \frac{\cF^i_\xi(\xi_i) \lam_i^2}{\cF^j_\xi(\xi_j) \lam_j^2} \right|=o(1).
		\end{equation}
	\end{cor}
	\begin{proof}
		Using Proposition \ref{prop:lamexpansion} with the assumption  $|\mu|\leq C \var$ for some constant $C>0$, we have 
		\[ \frac{(\cF_{\xi}^i g^i_\xi)(\xi_i)}{2\al_i-1} \lam_i^{4\al_i-2} =(1+o(1)) \frac{ 8 \int_\Si V e^u \, \d v_g}{\rho_*}, \quad\text{as } \varepsilon\to 0^+. \]
		Since $g^i_\xi(\xi_i)=1+ O(\sum_j |\al _j-1|)$ and $\sum_s|\al_s-1|\ln \lam_s=o(1)$,  it follows that 
		\[ \cF^i_{\xi}(\xi_i)\lam_i^2=(1+o(1)) \cF_{\xi}^j(\xi_j)\lam_j^2,\quad \forall\, i,j=1,\ldots,p+q.\]
		This concludes the proof of \eqref{eq:est_lam_i_j}.
	\end{proof}

	\begin{prop}\label{prop:bubble_expansion}
		Under the same assumptions of Proposition \ref{prop:lamexpansion},	the following expansion  	holds as $\var\to0^+$:
		\begin{align}
			\Big\la\nabla J_{\rho} (u), \frac{P \delta_{\lam_i,\xi_i}  }{\ln \lam_i }\Big\ra=&\, 4\kap (\xi_i)( (\al _i-1)+(\tau_i-\mu+\mu\tau_i))\nonumber \\&\, - \Big(2\pi-\frac{(1+2\ln 2)\pi}{\ln\lam_i}\Big) \frac{\rho \mathfrak{f}_2(\xi_i)}{\int_\Si  V e^u  \, \d v_g} \lam_i^{4\al_i-3} -\sum_{j}  \frac{\rho \pi  H^g(\xi_j,\xi_i)\mathfrak{f}_2(\xi_j) }{2\int_\Si   V e^u \, \d v_g } \frac{\lam_{j}^{4\al _j-3}}{\ln \lam _i}\nonumber\\
			&
			\,+O \Big(\sum_{j}  \frac{|\al _j-1|}{\ln \lam _i}+\sum_{j}  \frac{ |\tau_j-\mu+\mu\tau_j|}{\ln \lam _i}+\Big(\frac{1}{\ln \lam_{i}}+ |\al _i-1|\Big)\|w\|+\|w\|^2\Big)\nonumber\\
			&\, + O\Big(\frac 1 {\ln \lam_i}\sum_j \frac{\ln \lam_j}{\lam_j^2}+ \frac{\ln \lam_i}{\lam_i^2} +\sum_j \frac 1 {\lam_j^{4\al _j-2}}\Big),\label{eq:est_PU_0_gradient}
		\end{align}
		where $\cF_{\xi}^i$,    $\rg_{\xi}^i$, $\mathfrak{f}_2$  and $\tau_i$, are  defined in  \eqref{eq:def_F},  \eqref{def:rg}, \eqref{def:frak_f2} and \eqref{def:tau_i},     respectively.  
	\end{prop}
	\begin{proof}
		The proof is provided in  Appendix \ref{app:C}.
	\end{proof}
	
	\begin{prop}\label{prop:3} Under the same assumptions as in Proposition \ref{prop:lamexpansion}, the following expansion 	holds as $\var\to0^+$:
		\begin{align}	
			\Big\la \nabla J_\rho (u), \frac{1}{\lam_i}\frac{\pa P\delta_{\lam_i , \xi_i}}{\pa(\xi_i)_j} \Big\ra =&\, -\frac{(1+\mu)}{2}\frac{1}{\lam_i}\frac{\pa \ff(\xi)}{\pa(\xi_i)_j} +O\Big( \frac{|\mu|}{\lam_i}+\sum_s\frac{|\tau_s|}{\lam_i}+ \frac{\ln \lam_i}{\lam_i^2}\Big)\nonumber\\
			&\, +O\Big(\sum_s |\al _s-1|^2+\sum_s \frac{1}{\lam_s^2}+ \|w\|^2\Big),\label{eq:gradient_pa_xi_i}
		\end{align}
		where $j\in\{1,\ii(\xi_i)\}$ and  $\ff$ is defined in \eqref{eq:reduce_fun_0}.
	\end{prop}
	\begin{proof}
		The proof is provided in  Appendix \ref{app:C}.
	\end{proof}

	\begin{cor}\label{cor:1}
		Under the same assumptions as in Proposition \ref{prop:lamexpansion}, if we further assume that $|\mu|>0$ is sufficiently small, then the following expansion 	holds as $\var\to0^+$:
		\[
		\Big\langle\nabla J_{\rho}(u), \frac{P\delta_{\lam_i,\xi_i}}{\ln \lam_i}-2\frac{\lam_i}{\al _i}  \frac{\pa  P\delta_{\lam_i,\xi_i}}{\pa  \lam_i}\Big\rangle
		=  4\kap(\xi_i)(\al _i-1)+ \bm{\epsilon}(\al ,\xi,\lam,w),
		\]
		where 
		\begin{align*}
			&\bm{\epsilon}(\al ,\xi,\lam,w)
			\\=&\left\{\begin{aligned}
				&	O\Big(\|w\|^2+\frac{\|w\|}{\ln \lam_i}+ \sum_j\Big(\frac{|\tau_j-\mu+\mu \tau_j|+|\al _j-1|}{\ln \lam_i}+|\al _j-1|^2+ \frac{1}{\lam_j}\Big) \Big)	&&\text{ if }\, q\neq 0,\\
				&	O\Big(\|w\|^2+\frac{\|w\|}{\ln \lam_i}+ \sum_j\Big(\frac{|\tau_j-\mu+\mu \tau_j|+|\al _j-1|}{\ln \lam_i}+|\al _j-1|^2+\frac{\ln \lam_j}{\lam_j^2}+\frac{1}{\lam_j^{4\al_j-2}}  \Big)\Big) 
				&&\text{ if }\, q=0.
			\end{aligned}	\right.	
		\end{align*}
	\end{cor}
	\begin{proof}
		The result follows directly from   Propositions \ref{prop:lamexpansion}  and \ref{prop:bubble_expansion}.
	\end{proof}
	\begin{cor}\label{cor:alpha_itau_i}
		Under the same assumption of Corollary \ref{cor:1}, it holds  as $\varepsilon\to 0^+$
		\begin{equation}	\label{tau_iestimate}
			|\al _i-1|=O(\var+|\mu|)	\quad  \text{ and } \quad |\tau_i|=O(\var+|\mu|).
		\end{equation}
	\end{cor}
	\begin{proof}
		From \eqref{eq:est_tau_i}, we have  
		\begin{equation} 	\label{tau_iestimate_0}
			|\tau_i|=O\Big(\varepsilon+ |\mu|+ \sum_j 
			|\alpha_j-1|^2\Big).
		\end{equation}
		Then, from Corollary \ref{cor:1}, the definition of $V_{\rho}^{\eta}(p,q,\var)$  and \eqref{eq:apha_i_estimate}, it holds  that 
		\begin{equation}
			\label{tau_iestimate_1}
			| \al _i-1|=O\Big(\var+ \sum_j \frac{|\tau_j|}{\ln \lam _i}+ \sum_j |\alpha_j-1|^2\Big),
		\end{equation}
		where we used the facts	 $\max\{\|\frac{P\delta_{\lam_i,\xi_i}}{\ln \lam_i}\|,\|\lam_i   \frac{\pa  P\delta_{\lam_i,\xi_i}}{\pa  \lam_i}\|\}=O(1)$ (see Lemmas \ref{lem:projbubble_pointwise} and \ref{lem:projbubble_lamderivative__pointwise}). 
		Summing \eqref{tau_iestimate_0} and \eqref{tau_iestimate_1} over $i\in \{1,\ldots, p+q\}$, we obtain 
		\[ \sum_i |\alpha_i-1|+\sum_i|\tau_i|=O(\varepsilon+|\mu|),\]
		in view of $|\alpha_i-1|=o(1)$ and $\frac 1 {\ln \lambda_i}=o(1)$ for $i=1,\ldots, p+q$. 
		Then, 
		\eqref{tau_iestimate} is concluded.
	\end{proof}

	Next, we expand the energy functional $J_\rho $ in $V^\eta_{\rho}(p,q,\var)$. 
	\begin{prop}\label{prop:J(u)expansion}
		Let $u=u_0+w \in V^\eta_{\rho}(p,q,\var)$ with  $\rho >0$, where $u_0=\sum_{i} \al_i P \delta_{\lam_i,\xi_i}$ and  $w\in E_{\lam,\xi}^{p,q}$. Then we have 
		\begin{equation}\label{eq:J_lam}
			J_{\rho }(u)=J_{\rho }(u_0)+f(w)+\frac{1}{2}Q(w,w)+O(\var\|w\|^2), 
		\end{equation}
		where  the linear term  $f(w)$ is given by
		\[	
		f(w):=-\rho  \frac{\int_{\Si } V e^{u_0} w \,\d v_g}{\int_{\Si } V e^{u_0}\,\d v_g}
		\] 
		and the bilinear form $Q(\cdot, \cdot)$  is defined as
		\[	 
		Q(w, w'):=\la w, w'\ra -\rho  \frac{\int_{\Si } V e^{u_0} w w' \,\d v_g}{\int_{\Si } V e^{u_0} \,\d v_g}+\rho \Big(\frac{\int_{\Si } V e^{u_0} w  \,\d v_g }{\int_{\Si } V e^{u_0} \,\d v_g }\Big)\Big(\frac{\int_{\Si } V e^{u_0} w' \,\d v_g }{\int_{\Si } V e^{u_0} \,\d v_g }\Big)
		\]
		for any $w,w'\in E_{\lam,\xi}^{p,q}$.
	\end{prop}
	\begin{proof}
		The proof is standard, and we leave the details in Appendix \ref{app:C}.
	\end{proof}
	

	
	\begin{prop}\label{prop:Q_0positive}
		Under the same assumptions as in Proposition \ref{prop:J(u)expansion},  there exists  a constant $c_0>0$ such that 
		\[ 
		Q(w, w)= Q_0(w,w) +o(\|w\|^2)\geq c_0 \|w\|^2, \quad \forall\, w \in E^{p,q}_{\lam,\xi},
		\]
		where the bilinear form $Q_0(\cdot,\cdot)$ is given by 
		\[ 
		Q_0(w,w'):= \la w, w'\ra  - \sum_{i} \int_{\Si } \chi_{\xi_i}   e^ {\delta_{\lam_i,\xi_i}} ww' \,\d v_g,\quad \forall\, w,w' \in E^{p,q}_{\lam,\xi}.
		\]
		In particular, $Q_0(\cdot,\cdot)$ is symmetric and positive definite.
	\end{prop}
	\begin{proof}
		The proof is standard, and we leave the details in Appendix \ref{app:C}.
	\end{proof}
	\begin{cor}\label{cor:est_w_w1}
		Under the same assumption of Proposition \ref{prop:Q_0positive}, we have  
		\begin{align*}
			\la \nabla J_\rho(u_0+w), w_1 \ra =&\, f(w_1)+Q_0(w, w_1)+o(\|w\|\|w_1\|)\\
			=&\,Q_0(w, w_1)+O\Big(\sum_i \Big(|\al_i-1|+\frac {  |\nabla_{\xi_i}\cF^i_\xi(\xi_i)|+|\mathfrak{f}_2(\xi_i)|} {\lam_i}+ \frac {\ln ^{\frac 32} \lam_i} {\lam_i^2}\Big)\|w_1\|\Big)\\
			&\,+o(\|w\|\|w_1\|)
		\end{align*}
		for any $w_1\in E^{p,q}_{\lam,\xi}$.
	\end{cor}
	\begin{proof}
		The result follows directly from Propositions \ref{prop:J(u)expansion}-\ref{prop:Q_0positive} and Lemma \ref{lem:A.7}. 
	\end{proof}



	As in \cite{CL2002,CL2003}, the asymptotic behavior of $\rho_k-4\pi m$ plays a crucial role in the computation of the Leray-Schauder degree. We establish the following key estimate:
	\begin{prop}\label{prop:estimates_rho_k}
		Let	 $\{u_k\}$  be a sequence of blow-up solutions of \eqref{maineq} with $\rho=\rho_k$.   Assume that  $\lim_{k\to+\infty}\rho_k=  \rho_*:=4\pi m$,  and  the corresponding blow-up point set  is $\{ \xi^*_1,\ldots,\xi^*_{p+q}\}$ with  $(\xi^*_1,\ldots,\xi^*_{p+q})\in \Xi_{p,q}$ up to a permutation, where $p,q\in\N_0$ and $m=2p+q\in \N$. 
		According to   Propositions \ref{prop:u_k_in_V^eta} and \ref{prop:unique},  for $k$ sufficiently large, $u_k$ admits a unique decomposition of the form
		\[	
		u_k =  \sum_i \al^k_i P\delta_{\lam^k_i,\xi^k_i}+w^k, 
		\]
		where $w^k\in E^{p,q}_{\lam,\xi}$. 
		For simplicity of notation, we will continue to use 
		$ u, w, \al , \xi, \lam,  \rho $ in place of $ u_k, w^k, \al ^k, \xi^k, \lam^k, \rho_k$.  
		Assume further that condition  \hyperref[item:C1]{(C1)} holds. Then, we have the following estimates: 
		\begin{equation}
			\label{eq:est_grad_J_zero}
			\max\Big\{ \|w\|, \sum_i|\al _i-1|,\sum_i  |\tau_i-\mu+\mu\tau_i|, \sum_i\frac{d_g(\xi_i,\xi_i^*)}{\ln \lam_i} \Big\}=O\Big(\sum_i \frac 1 {\lam_i}\Big)
		\end{equation}
		and 
		\begin{equation}
			\label{eq:est_mu_best} 	\mu =\frac{\rho-\rho_*} {\rho_*}=\left\{
			\begin{aligned}
				&	\frac {\pi} { 4} \frac{1}{\int_\Si V e^u \, \d v_g
				}\cL_1(\xi^*)\sqrt{\cF_{\xi}^1(\xi_1)} \lam_1 + o\Big(\frac 1 {\lam_1}\Big) && \text{ if }\, q\neq 0,\\
				&	\frac {1} { 16} \frac{1}{\int_\Si V e^u \, \d v_g
				}\cL_2(\xi^*) \ln  \lam_1  + o\Big(\frac{\ln \lam_1} {\lam^2_1}\Big)	&& \text{ if }\, q=0,	
			\end{aligned}\right. 
		\end{equation}
		where $	\cL_1$ and 	$\cL_2$ are defined in \eqref{eq:def_cL_1} and \eqref{eq:def_cL_2}, respectively. 
	\end{prop}
	\begin{proof}
		Since $u$ is a solution of \eqref{maineq},  it follows that $\nabla J_\rho(u)=0$.  Then we deduce from Proposition \ref{prop:Q_0positive}  and 	 Corollary \ref{cor:est_w_w1} with $w_1=w$ that 
		\begin{equation}
			\label{eq:est_w_best1}
			\|w\|= O\Big(\sum_i  \Big(|\al _i-1|+\frac{1}{\lam_i}\Big)\Big).
		\end{equation}
		Combining  Proposition \ref{prop:lamexpansion}, Corollary \ref{cor:1} with \eqref{eq:est_w_best1}, we derive that 
		\begin{align}
			| \tau_j-\mu+\mu\tau_j|=&\,O\Big( \sum_i |\al _j-1|^2 + \sum_i \frac 1 {\lam_i}\Big), \label{eq:est_tau_i_best1}\\
			\text{ and } \quad |\al _j-1|=& \,O\Big( \sum_i \frac{	 | \tau_i-\mu+\mu\tau_i|}{\ln \lam_i} + \sum_i \frac 1 {\lam_i}\Big), \label{tau_iestimate_best1}
		\end{align}
		for $j=1,\ldots, p+q.$
		In view of $|\al _i-1|=o(1)$ and $ | \tau_i-\mu+\mu\tau_i|=o(1)$, 
		\eqref{eq:est_tau_i_best1} and \eqref{tau_iestimate_best1}
		imply that 
		\begin{equation}
			\label{eq:est_tau_al_best}
			\sum_i 	|\al _i-1|=O\Big(\sum_i \frac 1 {\lam_i}\Big)\quad \text{ and }\quad  \sum_i | \tau_i-\mu+\mu\tau_i|=O\Big(\sum_i \frac 1 {\lam_i}\Big).
		\end{equation}
		It follows  from   \eqref{eq:est_w_best1} and  \eqref{eq:est_tau_al_best} that
		\begin{equation}
			\label{eq:est_w_best}
			\|w\|=O\Big(\sum_i \frac 1 {\lam_i}\Big). 
		\end{equation}
		Moreover, we know from \eqref{eq:est_tau_al_best} that $\sum_i|\al _i-1|\ln \lam_i=o(1)$.
		Applying Proposition \ref{prop:3} and condition \hyperref[item:C1]{(C1)}, we get  
		\begin{equation}\label{eq:est_xi*_best}	
			\sum_i d_g(\xi_i^*,\xi_i)=o(1).
		\end{equation}
		From   \eqref{eq:est_tau_al_best}, \eqref{eq:est_w_best} and \eqref{eq:est_xi*_best} immediately, we deduce 
		\[  	\max\Big\{ \|w\|, \sum_i|\al _i-1|,\sum_i  |\tau_i-\mu+\mu\tau_i| \Big\}=O\Big(\sum_i \frac 1 {\lam_i}\Big) .\]

		Next, we will utilize \eqref{eq:est_grad_J_zero}, together with Proposition \ref{prop:lamexpansion} and Lemma \ref{lem:Ve^u_integral},   to derive an estimate for $\mu=\frac{\rho-\rho_*}{\rho_*}$. By summing up \eqref{prop:lamexpansion-1} for $i=1,\ldots, p+q$,  we have 
		\begin{align}
			2\rho_* \mu=& \,2(1+\mu)\sum_i \kap(\xi_i)\tau_i-\frac{\pi}{2}\frac{\rho}{\int_{\Si} V e^u \, \d v_g}\sum_i \mathfrak{f}_2(\xi_i)\lam_i^{4 \al _i-3}  \nonumber\\
			&\,-\sum_j \frac{\rho\kap(\xi_j)}{|\Si|_g}\frac{\ln \lam_j}{\lam_j^2}+o\Big(\sum_i \frac{\ln \lam_i}{\lam_i^2}\Big),\label{eq:est_mu1}
		\end{align}
		where we used \eqref{eq:est_tau_i},  \eqref{eq:est_tau_al_best} and \eqref{eq:est_w_best}. 
		On the other hand, by applying Lemma \ref{lem:Ve^u_integral}  together with  \eqref{eq:est_tau_al_best} and \eqref{eq:est_w_best}, we obtain
		\begin{align}
			\sum_i \kap(\xi_i) \tau_i =&\,\frac{ \rho_*}{\int_{\Si} V e^u\, \d v_g}\Big(\int_{\Si} V e^u\, \d v_g- \sum_i\frac{\kap(\xi_i)(\mathcal{F}_{\xi}^i \mathrm{g}_{\xi}^i )(\xi_i)}{8(2 \al _i-1)} \lam_i^{4 \al _i-2}\Big) \nonumber \\
			=&\, \rho_*
			\sum_i\frac{\kap(\xi_i)}{2|\Si|_g}\frac{\ln \lam_i}{\lam_i^2}
			+	\frac \pi 2\frac{\rho_* } {\int_{\Si} V e^u\, \d v_g }  \sum_{i}  \mathfrak{f}_2(\xi_i) \lam_i^{4\al _i-3}\nonumber \\
			&\,+ \frac{\rho_*}{\int_\Si V e^u \, \d v_g}\sum_i  \frac{\kap(\xi_i)\mathfrak{f}_3(\xi_i)}{16} \ln \lam_i 
			+o\Big(\sum_{i} \frac{\ln \lam_i }{\lam_i^2}\Big),\label{eq:est_mu_2}
		\end{align}
		where $\mathfrak{f}_3$ is given by \eqref{def:frak_f3}
		The proof of \eqref{eq:est_mu_best} will be divided into the following two cases.

		\textbf{Case I:}  $q\neq 0$ and $\cL_1(\xi^*)\neq 0$.
		By combining  \eqref{eq:est_lam_i_j}, \eqref{eq:est_tau_al_best}, \eqref{eq:est_mu1} and \eqref{eq:est_mu_2}, we have
		\begin{align*}
			\mu=& \,\frac {\pi} { 4} \frac{1}{\int_\Si V e^u \, \d v_g
			} \sum_i \mathfrak{f}_2(\xi_i) \lam_i + o\Big(\sum_i \frac 1 {\lam_i}\Big)\\
			=&\,	\frac {\pi} { 4} \frac{1}{\int_\Si V e^u \, \d v_g
			}\cL_1(\xi^*)\sqrt{\cF_{\xi}^1(\xi_1)}  \lam_1 + o\Big(\frac 1 {\lam_1}\Big). 
		\end{align*}
		where we also used the facts $|\mu|=o(1)$,  $\sum_i d_g(\xi^*_i,\xi_i)=o(1)$ and $\sum_i|\al _i-1|\ln \lam_i=o(1)$. 
		
		\textbf{Case II:} $q=0$ and $\cL_2(\xi^*)\neq 0$.
		Similarly,  using \eqref{eq:est_lam_i_j}, \eqref{eq:est_tau_al_best}, \eqref{eq:est_mu1} and \eqref{eq:est_mu_2}, we obtain 
		\begin{align*}
			\mu=& \,\frac {1} { 16} \frac{1}{\int_\Si V e^u \, \d v_g
			}\sum_i \kap(\xi_i)\mathfrak{f}_3(\xi_i)\ln \lam_i + o\Big(\sum_i \frac{\ln\lam_i}{\lam_i}\Big)\\
			=&\,	\frac {1} { 16} \frac{1}{\int_\Si V e^u \, \d v_g
			}\cL_2(\xi^*)\ln  \lam_1  + o\Big(\frac{\ln \lam_1} {\lam^2_1}\Big), 
		\end{align*}
		where we  also used the facts $\kap(\xi_i)=\kap(\xi_i^*)$, $|\mu|=o(1)$,  $\sum_i d_g(\xi^*_i,\xi_i)=o(1)$ and $\sum_i|\al _i-1|\ln \lam_i=o(1)$.
		Thus, \eqref{eq:est_mu_best} is concluded. 
		
		Proposition~\ref{prop:3} yields that 
		\begin{align*}
			\frac1 2(1+\mu) \frac{\partial_{(\xi_i)_l}\ff(\xi)}{\lam_{i}} &=\, O\Big( \frac{|\mu|}{\lambda_i}+\sum_j\frac{|\tau_j|}{\lam_i}+ \frac{\ln \lambda_i}{\lambda_i^2}\Big) +O\Big(\sum_j |\alpha_j-1|^2+ \|w\|^2\Big),
		\end{align*}
		for any $i=1,\ldots,p+q$ and $l=1,\ldots, \ii(\xi_i).$
		Using \eqref{eq:est_mu_best}, \eqref{eq:est_tau_al_best} and  \eqref{eq:est_w_best}, it follows that
		$ |\nabla\ff(\xi)|=O(\sum_i \frac{\ln \lam_{i}}{\lam_{i}}).$
		By the smoothness of $\ff$, we take the limit as  $\rho\rightarrow\rho_*$, which implies $\nabla \ff(\xi^*)=0$, where $\xi_1^*,\ldots,\xi^*_{p+q}$ are blow-up points. From condition \hyperref[item:C2]{(C2)}, 
		$\ff$ admits only non-degenerate critical points, meaning that  $\Hess\, (\ff)$ at $\xi^*$ is non-degenerate. Consequently, we obtain
		$\sum_i d_g(\xi_i,\xi_i^*)\leq C \sum_i \frac{\ln \lam_{i}}{\lam_{i}},$
		for some constant $C>0.$
		Thus, \eqref{eq:est_grad_J_zero} is concluded.
	\end{proof}


	\section{Topological degree formula}\label{sec:4}
	In this section, we present the proof of the paper's main results. Let   $p,q\in \N_0$ such that $m:=2p+q\in\N$, and denote $\rho_*:=4\pi m$.
	Let  $\xi^*:=(\xi^*_1,\ldots,\xi^*_{p+q})\in \Xi_{p,q}$ be a critical point of $\ff$,  and set $\mu=\frac{\rho -\rho_*}{\rho_*}$  with $|\mu|$ sufficiently small. We fix a constant  $\eta>0$  such that
	\[ 
	\eta<\frac 1 2 \min \big\{ d_g(\xi^*_i,\xi^*_j): i,j=1,\ldots, p+q \text{ with }i\neq j\big\}\cup \big\{ d_g(\xi_i^*,\pa\Si): i=1,\ldots, p\big\}.
	\]
	In the case $q\neq 0$,  we define the set
	\begin{align}
		V(\mu,\xi^*):=\Big\{ u\in \oH: &\, u= \sum_i \al _i P\delta_{\lam_i,\xi_i}+w\in V^\eta_{\rho_*} (p,q, C_0|\mu|^{\frac 1 2}) \text{ with } w\in E^{p,q}_{\lam,\xi},\, \|w\|<C|\mu|,\nonumber \\
		&\, \sum_i d_g(\xi_i,\xi^*_i
		)< -C |\mu|\ln |\mu|,  \sum_i(|\al _i-1|+|\tau_i|)<C|\mu| \text{ and}\nonumber\\ & \,  \frac {|\mu|} C< \frac 1 {\lam_i}< C |\mu| \text{ for } i=1,\ldots, p+q\Big \}; 	\label{eq:def_V1_xi_rho}
	\end{align}
	while for the case $q=0$, we define
	\begin{align}
		V(\mu,\xi^*):=\Big \{ u\in \oH: &\,  u= \sum_i \al _i P\delta_{\lam_i,\xi_i}+w\in V^\eta_{\rho_*} (p,q, C_0|\mu|^{\frac 1 3}) 
		\text{ with } w\in E^{p,q}_{\lam,\xi},\, \|w\|<C|\mu|^{\frac 34 },\nonumber \\
		&\, \sum_i d_g(\xi_i,\xi^*_i
		)< C |\mu|^{ \frac 38 }, \sum_i(|\al _i-1|+|\tau_i|)<C|\mu| \text{ and} \nonumber\\ &\,  \frac {|\mu|} C< \frac{\ln \lam_i}{\lam_i^2}< C |\mu| \text{ for } i=1,\ldots, p+q\Big\}, 	\label{eq:def_V2_xi_rho}
	\end{align}
	where  $C_0, C>0$  are  some positive constants. 
	\begin{thm}\label{thm:alter_blowup_sols}
		Let $\rho_*=4\pi m$ with $m\in  \N$,  and let  $V$ be a  positive  $C^{3}$ function on $\Si$    satisfying conditions \hyperref[item:C1]{(C1)} and \hyperref[item:C2]{(C2)} for all $p,q\in \N_0$ with  $m=2p+q$. Then there exist constants $\mu_{m}> 0$,   $C_\rho  > \bar{C}_{m} > 0$, where   $C_\rho $ depends continuously on $\rho $ and satisfies $\lim_{\rho  \to \rho_*} C_\rho  =+ \infty$, such that  any solution $u$ of \eqref{maineq}   with $u\in\oH$ and   $\frac{|\rho-\rho_*|}{\rho_*} \leq  \mu_{m}$ satisfies the following:
		\begin{itemize}
			\item[(1)] $\|u\| < C_\rho $  for  $\rho  \notin  4\pi \N$;
			\item[(2)] If $\rho  =\rho_*$, then $\|u\|< \bar{C}_{m}$;
			\item[(3)] If $\rho  \neq \rho_*$,  we have either
			\begin{itemize}
				\item[i)] $\|u\| < \bar{C}_{m}$ or 
				\item[ii)] $\|u\| \geq  \bar{C}_{m}$  and  $u\in V(\frac{\rho-\rho_*}{\rho_*},\xi^*)$ for some critical point $\xi^*:=(\xi_1^*,\ldots,\xi^*_{p+q})$ of $\ff$. 
			\end{itemize}
		\end{itemize}
	\end{thm}
	\begin{proof}
		For $\rho \notin 4\pi \N$,  utilizing the compactness results in \cite{LSY2023, Battaglia2019}, one obtains the existence of a constant $C_\rho>0 $  such that $\|u\|< C_\rho $. Furthermore,   $C_\rho $ can be chosen to depend continuously on $\rho $  by defining
		\[	
		C_\rho := \sup \{\|u\|:\text{$u$ is a solution of \eqref{maineq}}\}+1.
		\]
		Due to the existence of blow-up solutions as  $\rho \to  \rho_*$ (see \cite{HBA20240831,HBA20240829}), it follows that $\lim_{\rho  \to \rho_*} C_\rho  =+ \infty$.  Thus,  (1) is proved.
		
		Suppose that  (2) fails. Then, there exists a sequence of blow-up solutions $\{u_k\} $ of \eqref{maineq} with $\rho=\rho_k= \rho_*$. Let  $\{\xi_1,\ldots,\xi_{p+q}\}$  be the corresponding blow-up point set with $\xi=(\xi_1^*,\ldots, \xi_{p+q}^*)\in \Xi_{p,q}$ for some $p,q\in \N_0$ satisfying $2p+q=m$.  
		According to Proposition \ref{prop:estimates_rho_k},  we have $\cL_1(\xi)=0$ if   $q\neq 0 $, and $\cL_2(\xi)=0$ if $q=0$, which contradict to the condition  \hyperref[item:C2]{(C2)}. Hence, (2) is valid.

		Suppose that the statement (3) fails. Then, there exists   a sequence of solutions  $\{u_k\}$ of \eqref{maineq} with $\rho=\rho _k\to  \rho_*$, such that  $\|u_k\|\geq  \bar{C}_{m}$, but  
		\begin{equation}
			\label{criticalset}
			u_{k}\notin \bigcup_{\xi^* \in \mathcal{C}(\ff)} V\Big(\frac{\rho_k-\rho_*}{\rho_*},\xi^*\Big),
		\end{equation}
		where $\mathcal{C}(\ff)$ denotes the critical point set of $\ff$.
		
		If $\{u_k\} $ does not blow-up, the compactness property implies that $\{u_k\} $ is uniformly bounded and converges to some $u^*\in C^{2,\al}(\Sigma)$, which solves \eqref{maineq} with $\rho=\rho_*$. In this case, $\|u^*\|\geq \bar{C}_m$, which contradicts to (2). Otherwise,  if $\{u_k\} $  does blow up with  blow-up points $\{\xi^*_1,\ldots, \xi^*_{p+q}\}$ satisfying that  $\xi^*:=(\xi^*_1,\ldots, \xi^*_{p+q})\in \Xi_{p,q}$, then Propositions \ref{prop:u_k_in_V^eta} and   \ref{prop:estimates_rho_k} imply that 
		$u_k\in V(\frac{\rho_k-\rho_*}{\rho_*},\xi^*)$ for sufficiently large $k$, and $\xi^*$ is a critical point of $\ff$. This leads to a contradiction with \eqref{criticalset}. Therefore,  (3) holds as well.
	\end{proof}

	For any $\rho>0$, the Leray-Schauder degree associated with the functional $J_\rho$ is defined as
	\[ 
	d_\rho: = \deg\,( \nabla J_{\rho}, \oH, 0).
	\]
	When $\rho\notin 4\pi  \N$ (i.e., $\rho$ is non-resonant),  the solution set of $\nabla J_\rho$ is uniformly bounded from above (see Theorem \ref{thm:alter_blowup_sols}).  Moreover, since  $\nabla J_\rho=I-\rho\cK$, where
	\[ 
	\cK(u)=(-\Delta_g)^{-1}\Big(\frac{V e^u}{\int_\Si V e^u \, \d v_g}-1\Big) 
	\] is a compact operator, the degree $d_\rho$ is well-defined for all non-resonant values of $\rho$.
	
	Thanks to the compactness of the solution set for non-resonant values of $\rho$,  the Leray-Schauder degree $d_\rho$ is a constant in each interval $(4\pi (m-1), 4\pi m)$  for   $m\in  \N$. 
	It is also well-known that for $\rho \in (0, 4\pi)$, the functional $J_\rho $ is coercive and admits a unique minimizer. Herein, $d_\rho =1$ for $\rho \in (0,4\pi)$. 
	Let $\rho_*=4\pi m$ with $m\in \N$, and define 
	\[
	d^-_{m}:=\lim_{\rho \to \rho_*, \,\rho <\rho_*} d_\rho \quad \text{ and }\quad d^+_{m}:=\lim_{\rho \to \rho_*,\,\rho >\rho_*} d_\rho .
	\]
	By Lemma \ref{lem:projbubble_integral_interaction}, for any $u \in V(\frac{\rho-\rho_*}{\rho_*}, \xi^*)$,  there exists a constant  $c_0 > 0$ such that
	$$
	\|u\| \geq c_0 \ln \Big(\frac{1}{|\rho - \rho_*|}\Big) , 
	$$
	which implies that $\|u\| \geq \bar{C}_m$ when $|\rho-\rho_*|$ is sufficiently small, where $\bar{C}_m$ is the constant given in Theorem \ref{thm:alter_blowup_sols}.
	This allows us to establish the following degree-counting formula for $d_\rho$ in a small neighborhood of the critical value  $\rho_*$. \par  
	Recall that  
	$\mathcal{B}_{C}:=\{u\in\oH:\|u\|<C\},$ for any $C>0$.  Let $\mu=\frac{\rho-\rho_*}{\rho_*}$, where  $\rho_*= 4\pi m$ with  $m\in \N$. 
	When $\rho >\rho_*$ is sufficiently close to $\rho_*$, Theorem \ref{thm:alter_blowup_sols}  yields that 
	\begin{align}
		d_{m}^+=&\, d_\rho =\deg\,( \nabla J_\rho , \mathcal{B}_{C_\rho }, 0)\nonumber\\
		=&\deg\,( \nabla J_\rho  , \mathcal{B}_{\bar{C}_m}, 0)+ \sum_{ 2p+q=m}\frac 1 {p! q!} \sum_{\xi^*\in \cV_{p,q}^+}\deg\,( \nabla J_\rho , V(\mu,\xi^*), 0)\nonumber\\
		=& \, d_{\rho_*} + \sum_{ 2p+q=m}\frac 1 {p! q!} \sum_{\xi^*\in \cV_{p,q}^+} \deg\,( \nabla J_\rho , V(\mu,\xi^*), 0), \label{eq:def_d_plus}
	\end{align}
	where $\cV_{p,q}^+$ is defined in \eqref{eq:def_V+}, and  the factor $ p!q!$ is included because any permutation of $(\xi_1, \ldots, \xi_p) $  and $ (\xi_{p+1}, \ldots, \xi_{p+q})$  is considered as the same solution.
	Similarly, for $\rho <\rho_* = 4\pi m $ sufficiently close to $\rho_*$, 
	\begin{equation}
		\label{eq:def_d_mius} d_{m}^-= d_{\rho_*}+\sum_{ 2p+q=m}\frac 1 {p! q!} \sum_{\xi^*\in \cV_{p,q}^-}\deg\,( \nabla J_\rho , V(\mu,\xi^*), 0),
	\end{equation}
	where $\cV_{p,q}^-$ is defined in \eqref{eq:def_V-}. It remains to compute the degree  
	\begin{equation}\label{remain_degreee}
		\deg\,( \nabla J_\rho , V(\mu,\xi^*), 0),
	\end{equation}
	which is determined by the blow-up solutions that concentrate at the points $\{\xi^*_1,\ldots, \xi^*_{p+q}\}$. Since the topological degree is invariant under homotopy, we will compute  \eqref{remain_degreee}  by deforming $\nabla J_\rho$ into a simpler operator.

	Let $\mu=\frac{\rho-\rho_*}{\rho_*}$, where  $\rho_*= 4\pi m$ with $m\in \N$. We define 
	$
	T_0 = \nabla J_\rho: V(\mu, \xi^*) \to \oH.
	$
	For any $u\in V(\mu, \xi^*)$,  Proposition \ref{prop:unique} implies  that 
	$u$ admits a unique decomposition of the form $u=\sum_i \al _i  P\delta_{\lam_i , \xi_i}+w$, where $w \in E^{p,q}_{\lam,\xi}$.  Next, for each $t\in[0,1]$, we define a homotopy operator $T_t:V(\mu, \xi^*) \to \oH$  via the following  inner products:   
	\begin{align}   
		\langle T_t(u),  w_1\rangle=&\, (1-t)\left\langle\nabla J_\rho(u),  w_1\right\rangle+tQ_0( w,  w_1)\quad  \text{ for } \, w_1 \in E^{p,q}_{\lam,\xi}, \label{eq:deform_w}\\
		\Big\langle T_t(u),   \frac{\pa P\delta_{\lam_i,\xi_i}}{\pa (\xi_i)_j}\Big\rangle=&\, (1-t)\Big\langle\nabla J_\rho(u),  \frac{\pa P\delta_{\lam_i,\xi_i}}{\pa (\xi_i)_j}\Big\rangle-t\frac {1+\mu}{2}\frac{\pa \ff(\xi)}{\pa  (\xi_i)_j} , \label{eq:deform_xi} \\
		\Big\langle T_t(u),\lam_i\frac{\pa P\delta_{\lam_i,\xi_i}}{\pa \lam_i}\Big\rangle =& \,(1-t)	\Big\langle \nabla J_\rho(u),\lam_i\frac{\pa P\delta_{\lam_i,\xi_i}}{\pa \lam_i}\Big\rangle+ 2t\al_i \kap(\xi_i)   (\tau_i-\mu+\mu\tau_i)\nonumber\\
		&\,-\frac{t\pi}{2}\frac{\rho \mathfrak{f}_2(\xi_i) }{\int_{\Si} V e^u \,\d v_g} \lam_i^{4 \al _i-3}-t\sum_j\frac{\kap(\xi_i)\kap(\xi_j)}{|\Si|_g}\frac{\ln\lam_j}{\lam_j^2},\label{eq:deform_pa_lam}
	\end{align}  
	and 
	\begin{equation}\label{eq:deform_ddelta_lam_xi}
		\Big\langle T_t (u),  \frac{P\delta_{\lam_i,\xi_i}}{\ln \lam_i}-2\frac{\lam_i}{\al _i}  \frac{\pa  P\delta_{\lam_i,\xi_i}}{\pa  \lam_i}\Big\rangle 
		=(1-t) \epsilon(\al ,\xi,\lam,w) + 4\kap(\xi_i)(\al _i-1)
	\end{equation}
	for all $ i=1,\ldots,p+q$ and  $j\in\{1,\ii(\xi_i)\}$. It is clear that $T_1$ is simpler than $T_0$.
	
	The following lemma excludes the existence of solutions to $T_{t}$ on the boundary of  $V(\mu, \xi^*)$ for all $t\in [0,1]$. 
	\begin{lem}\label{lem:deformT}
		Assume that $\mu \cL(\xi^*)>0$, where $\mu:=\frac{\rho-\rho_*}{\rho_*}$ and $\xi^*$ is a non-degenerate critical point of $\ff$. 
		Then, there exists a constant $\mu_1>0$ such that for all  $|\mu|<\mu_1$, it holds that $T_t(u) \neq 0$ for all $u \in \pa  V(\mu,\xi^*)$ and $0 \leq t \leq 1$.
	\end{lem}
	\begin{proof}
		Let $u=\sum_i\al_iP\delta_{\lam_i,\xi_i}+w \in \overline{V(\mu, \xi^*)}$, where $w\in E^{p,q}_{\lam,\xi}$ and $\overline{V(\mu, \xi^*)}$ is the closure of $V(\mu, \xi^*)$. Suppose that $T_t(u)=0$ for some $t\in [0,1]$.  
		We will proceed with our proof by considering the following two cases. 
		
		\textbf{Case I.} $q\neq 0$. Since $u\in V^\eta_{\rho_*}(p,q, C_0|\mu|^{\frac 1 2})$ in this case, it follows from \eqref{tau_iestimate} that
		$$
		\max\Big\{ 	\sum_i\frac{1}{\lam_i}, \|w\|,  \sum_i|\al _i-1|, \sum_i |\tau_i|  \Big\}\leq C_0|\mu|^{\frac 1 2}.
		$$
		Using
		Proposition \ref{prop:lamexpansion}, Corollary \ref{cor:1}, Proposition \ref{prop:3}, and Corollary \ref{cor:est_w_w1}, together with   $T_t(u)=0$,  we have the following estimates for all $i=1,\ldots,p+q$ and $j\in\{1,,\ii(\xi_i)\}$:
		\begin{align} 
			2 \kap(\xi_i)\al_i   (\tau_i-\mu+\mu\tau_i)=&\,\frac \pi 2   \frac{\rho \mathfrak{f}_2(\xi_i) }{\int_{\Si} V e^u \,\d v_g} \lam_i^{4 \al _i-3}+\sum_j\frac{\kap(\xi_i)\kap(\xi_j)}{|\Si|_g}\frac{\ln\lam_j}{\lam_j^2}+O(|\al _i-1|^2+\|w\|^2)\nonumber\\&+o\Big (\frac{\ln \lam_i}{\lam_i^2}\Big), \label{eq:deform1}\\
			4\kap(\xi_i)(\al _i-1)=&\,	O\Big (\|w\|^2+\frac{\|w\|}{\ln \lam_i}+ \sum_s\frac{|\tau_s-\mu+\mu \tau_s|+|\al _s-1|}{\ln \lam_i}+\frac{1}{\lam_i}+\sum_s |\al_s-1|^2\Big ),\label{eq:deform2}
			\\
			\frac{1+\mu}{2}\frac{\pa \ff(\xi)}{\pa (\xi_i)_j} =&\,O\Big ( |\mu|+\sum_s |\tau_s|+\sum_s\frac{ \ln \lam_s}{\lam_s}\Big )+O\Big (\sum_s \lam_s |\al _s-1|^2+ \lam_s\|w\|^2 \Big ),\label{eq:deform3}\\
			\|w\|=& \,O\Big(|\al _i-1|+ \sum_s \frac{|\nabla_{\xi_s}\cF_{\xi}^s(\xi_s)|+|\mathfrak{f}_2(\xi_s)|}{\lam_s}+\frac{\ln ^{\frac 3 2}\lam_i}{\lam_i^2}\Big).\label{eq:deform4}
		\end{align}  
		By combining \eqref{eq:deform1}, \eqref{eq:deform2} and \eqref{eq:deform4}, we obtain
		\begin{equation}\label{eq:est_all_lam_minus1}
			\max\Big\{ 
			\sum_i |\al _i-1|, \|w\|, \sum_i 	|\tau_i|\Big\}= O\Big(\sum_i \frac 1 \lam_i \Big),
		\end{equation}
		where we used the facts that $\max\{\sum_i |\al _i-1|, \sum_i \frac 1 \lam_i, \|w\|\}=o(1)$. 
		Using the same argument as in the proof of Proposition \ref{prop:estimates_rho_k}, we deduce from \eqref{eq:est_all_lam_minus1} that
		\[\mu = 	\frac {\pi} { 4} \frac{1}{\int_\Si V e^u \, \d v_g
		}\cL_1(\xi^*)\sqrt{\cF_{\xi}^1(\xi_1)}  \lam_1 + o\Big(\frac 1 {\lam_1}\Big).
		\]
		Due to the condition \hyperref[item:C2]{(C2)},  we have $|\mu| \lam_i +\frac 1 {|\mu|\lam_i}=O(1)$, and thus there exists a large constant $C>0$ such that    $\max\{\|w\|, \sum_ i |\al _i-1|, \sum_i |\tau_i|\}< C |\mu|$. By \eqref{eq:deform3}, we derive that
		$ 
		|\nabla \ff(\xi)|=O(|\mu| +\sum_i \frac{\ln \lam_i}{\lam_i}).
		$
		Observe that,  by condition \hyperref[item:C1]{(C1)}, we have  for any $i=1,\ldots, p+q$
		\[ 
		d_g(\xi_i,\xi^*_i)=O( |\nabla_{\xi_i}\cF_{\xi}^i(\xi_i)- \nabla_{\xi_i}\cF_{\xi}^i(\xi^*_i)|)=O\Big(|\mu| +\sum_i \frac{\ln \lam_i}{\lam_i}\Big).
		\]
		For $\mu>0$ sufficiently small and $C>0$ sufficiently large, it follows that $\sum_id_g(\xi_i,\xi^*_i)<- C|\mu| \ln  |\mu|$.
		Using \eqref{eq:deform1} and \eqref{eq:est_all_lam_minus1}, we can deduce  \eqref{eq:est_mu_best} in this case.
		Moreover, it is evident that for sufficiently small $\mu>0$, $u\in V^\eta_{\rho_*}(p,q,C_0|\mu|^{\frac 12})$. Therefore, we conclude that $u\in V(\mu,\xi^*)$.
		
		\textbf{Case II.} $q=0$.  Since $u\in V^\eta_{\rho_*}(p,q, C_0|\mu|^{\frac 1 3})$ in this case, it follows from \eqref{tau_iestimate} that
		\[ 
		\max\Big\{  \sum_i \frac{1}{\lam_i}, \|w\|, \sum_i  |\al _i-1|, \sum_i  | \tau_i|\Big\}	\leq C_0|\mu|^{\frac 1 3}.
		\]
		As demonstrated in Case I, by utilizing the estimate of $\nabla J_\rho(u),$  we obtain similar bounds for  \eqref{eq:deform1}, \eqref{eq:deform3} and \eqref{eq:deform4}. However, since $\mathfrak{f}_2(\xi_i)=0$ in this case,  \eqref{eq:deform2}  requires a different estimate, which is given by
		\begin{equation}\label{eq:deform_2_prime}
			4\kap(\xi_i)(\al _i-1)=	O\Big(\|w\|^2+\frac{\|w\|}{\ln \lam_i}+ \sum_s\frac{|\tau_s-\mu+\mu \tau_s|+|\al _s-1|}{\ln \lam_i}+\frac{\ln \lam_i}{\lam^2_i}\Big).
		\end{equation}
		By combining \eqref{eq:deform1}-\eqref{eq:deform_2_prime}, we obtain 
		\begin{equation}\label{eq:est_all_lam_minus2}
			\max\Big\{\sum_i 	|\tau_i|,
			\sum_i |\al _i-1|, \|w\|  \Big\}= O\Big(\sum_i \frac 1 \lam_i \Big),
		\end{equation}
		where we used the facts that $\max\{|\al _i-1|, \frac 1 \lam_i, \|w\| \}=o(1)$. As in the proof of Proposition \ref{prop:estimates_rho_k}, we can apply \eqref{eq:est_all_lam_minus1} to deduce
		\[	\mu=\frac {1} { 16} \frac{1}{\int_\Si V e^u \, \d v_g
		}\cL_2(\xi^*) \ln  \lam_1  + o\Big(\frac{\ln \lam_1} {\lam^2_1}\Big).\]
		It then follows that
		$
		\frac{|\mu|\lam_i^2}{\ln \lam_i}+ \frac{\lam_i^2}{|\mu| \ln\lam_i }=O(1)
		$  by condition \hyperref[item:C2]{(C2)}. 
		
		Using the condition \hyperref[item:C1]{(C1)} and \eqref{eq:deform3},  we get
		\begin{equation}
			\label{eq:est_naba_FF} 
			|\nabla\ff(\xi)|=O\Big(|\mu|+\sum_i \frac{\ln \lam_i}{\lam_i}\Big)=O\Big(\sum_i \frac{\ln \lam_i}{\lam_i}\Big)
		\end{equation}
		and $d_g(\xi_i,\xi_i^*)=O(\sum_i |\nabla\ff(\xi)-\nabla\ff(\xi^*)|)=O(\sum_i \frac{\ln \lam_i}{\lam_i}), $
		where we used $\mu=O( \frac{\ln \lam_i}{\lam_i})$.
		
		Combining \eqref{eq:deform1}, \eqref{eq:deform2}, \eqref{eq:deform_2_prime}, \eqref{eq:est_naba_FF}, and  \eqref{eq:est_all_lam_minus2},  we derive the following refined estimates:
		\begin{equation} \label{eq:est_all_lam_minus2_refined}
			\|w\|=O\Big(\sum_i\frac{\ln^{\frac 3 2}\lam_i}{\lam_i^2}\Big) \quad \text{ and } \quad \max\Big\{ \sum_i |\al _i-1|, \sum_i |\tau_i|\Big\} =O\Big(\sum_i \frac{\ln\lam_i}{\lam_i^2}\Big).
		\end{equation}
		By utilizing  \eqref{eq:deform1} and \eqref{eq:est_all_lam_minus2_refined}, we can deduce \eqref{eq:est_mu_best} in this case. 
		Using $\frac{|\mu| \lam_i^2}{\ln \lam_i}+ \frac{\ln \lam_i}{|\mu|\lam_i^2}=O(1)$, there exists a sufficiently large constant $C>0$ such that 
		\[ \|w\|< C|\mu|^{\frac 3 4},\quad  \sum_i d_g(\xi_i,\xi_i^*)< C |\mu|^{\frac 3 8},\quad  \sum_i|\al _i-1|<C|\mu|,\quad  \text{ and }\quad  \sum_i |\tau_i|<C|\mu|,\]
		which implies that 	 $u\in V^\eta_{\rho_*}(p,q,C_0|\mu|^{\frac 13})$ for sufficiently small $|\mu|>0$. Consequently, we conclude that $u\in V(\mu,\xi^*)$.
	\end{proof}
	As a consequence of Lemma \ref{lem:deformT}, it holds  
	\[\deg\,( \nabla J_\rho , V(\mu,\xi^*), 0)=\deg\,( T_1(u) , V(\mu,\xi^*), 0)\]
	provided  that $\mu \cL(\xi^*)>0$ and $|\mu|$ is sufficiently small. We now proceed to compute 
	\[
	\deg\,( T_1(u) , V(\mu,\xi^*), 0).
	\]
	
	Define the set 
	\[
	V^*(\mu,\xi^*):=\Big\{(\xi,\lam,\al ) \in  \Xi_{p,q}\times (\R_+)^{p+q}\times (\R_+)^{p+q}:\,u=\sum_i \al _iP\delta_{\lam_i,\xi_i}+w\in V(\mu,\xi^*),\, w\in E^{p,q}_{\lam,\xi} \Big\}.
	\]
	We introduce the map $	\Psi_{\mu,\xi^*}:  V^*(\mu,\xi^*)\to  \R^m\times\R^{p+q}\times\R^{p+q},$
	\[
	\Psi_{\mu,\xi^*}(\xi,\lam,\al ):=\left(\begin{matrix}
		\Psi_1\\
		\Psi_2\\
		\Psi_3
	\end{matrix}\right)(\xi,\lam,\al )=\left(\begin{matrix}
		(\Psi^{(i,j)}_1)_{\mytop{i=1,\ldots,p+q}{j=1,\ldots,\ii(\xi_i)}}\\
		(\Psi^{(i)}_2)_{i=1,\ldots, p+q}\\
		(\Psi_3^{(i)})_{i=1,\ldots,p+q}
	\end{matrix}\right)(\xi,\lam,\al ),\]
	where the components are defined as follows for each 
	$i=1,\ldots, p+q$ and $j\in\{1,\ii(\xi_i)\}$:
	\begin{gather*}
		\Psi_1^{(i,j)}(\xi,\lam,\al ):=
		\Big\langle T_1(u), \frac{\pa P\delta_{\lam_i,\xi_i}}{\pa (\xi_i)_j}\Big\rangle\Big|_{w=0}=
		-
		\frac {(1+\mu)}{2} \frac{1}{\lam_i}\frac{\pa \ff(\xi)}{\pa  (\xi_i)_j},\\\Psi_2^{(i)}(\xi,\lam,\al ):=
		\Big\langle T_1(u), \lam_i\frac{\pa P\delta_{\lam_i,\xi_i}}{\pa \lam_i}\Big\rangle\Big|_{w=0},
	\end{gather*}
	and 
	$\Psi_3^{(i)}(\xi,\lam,\al ):=
	\langle T_1 (u),  \frac{P\delta_{\lam_i,\xi_i}}{\ln \lam_i}-2\frac{\lam_i}{\al _i}  \frac{\pa  P\delta_{\lam_i,\xi_i}}{\pa  \lam_i}\rangle|_{w=0}= 	4\kap(\xi_i)(\al _i-1).$
	Under assumptions  \hyperref[item:C1]{(C1)} and \hyperref[item:C2]{(C2)}, and using the fact that $Q_0$ is  a positive definite inner product (see Proposition  \ref{prop:Q_0positive}), we have 
	\[ \deg \,( T_1(u), V(\mu, \xi^*),0)=\deg\,( \Psi_{\mu,\xi^*},  V^*(\mu, \xi^*), 0). \]
	The reduced finite-dimensional degree can be computed as follows:
	\begin{lem}
		\label{lem:degee_reduced}
		Assume that conditions \hyperref[item:C1]{(C1)} and \hyperref[item:C2]{(C2)} hold. Then, we have 
		\begin{equation}
			\label{eq:degree_finite_dim}
			\deg\,( \Psi_{\mu,\xi^*},  V^*(\mu, \xi^*), 0)= 
			\sgn\, (\rho-\rho_*) (-1)^{p-\morse(\ff,\xi^*)}	.
		\end{equation}
	\end{lem}
	\begin{proof}
		Due to the construction of $\Psi_{\mu,\xi^*}$, we have
		\[\frac{\pa \Psi_1}{\pa\lam_i}\equiv0,\quad \frac{\pa \Psi_1}{\pa \al _i}\equiv0, \quad \frac{\pa \Psi_3}{\pa (\xi_i)_j}\equiv 0\quad  \text{ and }\quad \frac{\pa \Psi_3}{\pa \lam_i}\equiv 0\]
		for all $i=1,\ldots,p+q$ and $j\in\{1,\ii(\xi_i)\}$.  Therefore, to compute the degree of $\Psi_{\mu,\xi^*}$, it suffices to focus on the component $\Psi_2$.

		From \eqref{eq:deform_pa_lam}, we deduce that
		\begin{align*}
			\Big\la 	T_1(u),\lam_i\frac{\pa P\delta_{\lam_i,\xi_i}}{\pa \lam_i}\Big\ra
			=&\, 2\al_i\kap(\xi_i)(\tau_i-\mu +\mu\tau_i )-\frac \pi 2   \frac{\rho\mathfrak{f}_2(\xi_i)  }{\int_{\Si} V e^u \, \d v_g} \lam_i^{4 \al _i-3}
			-\sum_j \frac{\kap(\xi_i)\kap(\xi_j)}{|\Si|_g}\frac{\ln \lam_j}{\lam_j^2}.
		\end{align*}
		
		Under conditions \hyperref[item:C1]{(C1)} and \hyperref[item:C2]{(C2)}, for  sufficiently small $|\mu|>0$, the system $\Psi_{\mu,\xi^*}=0$ admits a solution   $(\xi,\lam,\al)\in V^*(\mu, \xi^*)$ if and only if 
		\[ \al _i=1,\quad \xi_i =\xi^*_i\quad \text{ for all }\, i=1,\ldots, p+q;\]
		and each $\lam_i:=\lam_i(\mu)>0$  satisfies 
		$\la 	T_1(u),\lam_i\frac{\pa P\delta_{\lam_i,\xi_i}}{\pa \lam_i}\ra|_{w=0}=0.$
		Using \eqref{lem:A.8-2}, we have the following estimate for $\lam_i(\mu)$:
		$$
		\frac{\cF^i_{\xi^*}(\xi^*_i)\lam_i^2(\mu)}{\cF^j_{\xi^*}(\xi^*_j)\lam_j^2(\mu)}-1=O\Big(\frac 1 {\lam_1(\mu)}\Big)=o(1) \text{ as }|\mu|\to  0.
		$$
		By direct calculation, 
		\begin{equation}\label{degree1}
			\deg\,( \Psi_{\mu,\xi^*},  V^*(\mu, \xi^*), 0)=(-1)^{m-\morse(\ff,\xi^*)}\,\sgn\det\Big( \frac{\pa \Psi_2 }{\pa \lam}\Big)\Big|_{(\xi^*,\lam(\mu), e)},
		\end{equation}
		where $\lam(\mu):=(\lam_1(\mu),\ldots,\lam_{p+q}(\mu))$ and $e:=(1,\ldots, 1)\in\R^{p+q}$. At the point $(\xi^*,\lam(\mu), e)$, the Jacobian matrix of $\Psi_2$ with respect to $\lam:=(\lam_1,\ldots,\lam_{p+q})$ satisfies
		\[  \det\Big(\frac{\pa \Psi_2}{\pa \lam}\Big)=\det\left(\begin{matrix}
			\frac{\pa \Psi^{(1)}_2}{\pa \lam}\\
			\vdots\\
			\frac{\pa \Psi^{(p+q-1)}_2}{\pa \lam}\\
			\frac{\pa \Psi_2^{(p+q)}}{\pa \lam}
		\end{matrix}\right)= \det\left(\begin{matrix}
			\frac{\pa \Psi^{(1)}_2}{\pa \lam}\\
			\vdots\\
			\frac{\pa \Psi^{(p+q-1)}_2}{\pa \lam}\\
			\sum_i \frac{\pa \Psi_2^{(i)}}{\pa \lam}
		\end{matrix}\right).\]
		As in proof of Proposition \ref{prop:estimates_rho_k},  we can derive that 
		\begin{align*}
			\sum_i \Psi_2^{(i)}=& -2(\rho-\rho_*) +	\left\{\begin{aligned}
				&\,	\frac {\pi} { 2} \frac{\rho_*}{\int_\Si V e^u \, \d v_g
				}\cL_1(\xi^*)\sqrt{\cF_{\xi}^{p+q}(\xi_{p+q})} \lam_{p+q}+ o\Big(\frac 1 {\lam_1}\Big)&& \text{ if }\, q\neq 0,\\
				&\,		\frac {1} { 8} \frac{\rho_*}{\int_\Si V e^u \, \d v_g
				}\cL_2(\xi^*) \ln  \lam_{p+q}  + o\Big(\frac{\ln \lam_{1}} {\lam^2_{1}}\Big)	&& \text{ if }\, q=0,
			\end{aligned}	\right.\\
			=&\,	-2(\rho-\rho_*) +\left\{\begin{aligned}
				&\,	\frac{4\pi}{ \sqrt{\cF^{p+q}_{\xi^*}(\xi^*_{p+q})}\lam_{p+q}
				}\cL_1(\xi^*) + o\Big(\frac 1 {\lam_1}\Big)&& \text{ if }\, q\neq 0,\\
				&\,	 \frac{\ln  \lam_{p+q} }{  \cF^{p+q}_{\xi^*}(\xi^*_{p+q})\lam_{p+q}^2
				}\cL_2(\xi^*)  + o\Big(\frac{\ln \lam_{1}} {\lam^2_{1}}\Big)	&& \text{ if } \, q=0.
			\end{aligned}	\right.
		\end{align*}
		Then,  
		\[
		\frac{\pa }{\pa  \lam_{p+q}} \sum_i \Psi_2^{(i)}=\left\{\begin{aligned}&
			-\frac{4\pi}{\sqrt{\cF^{p+q}_{\xi^*}(\xi^*_{p+q})}
			}\cL_1(\xi^*)\frac 1 {\lam_{p+q}^2}+o\Big(\frac 1 {\lam_1^2}\Big) && \text{ if }\, q\neq 0,\\
			&	- \frac{2 }{ \cF^{p+q}_{\xi^*}(\xi^*_{p+q})
			}\cL_2(\xi^*)\frac {\ln\lam_{p+q}}{\lam_{p+q}^3}+	o\Big(\frac{\ln \lam_1} {\lam^3_1}\Big)	&& \text{ if }\, q=0,
		\end{aligned}	\right. 
		\]
		and for any $j\leq p+q-1$,
		\[
		\frac{\pa }{\pa\lam_j} \sum_i \Psi_2^{(i)}=\left\{\begin{aligned}
			&o\Big(\frac 1 {\lam_1^2}\Big)&& \text{ if }\,q\neq 0,\\
			&o\Big(\frac{\ln \lam_1} {\lam^3_1}\Big)	&& \text{ if }\, q=0.
		\end{aligned}\right.
		\]
		For $1\leq i,j\leq p+q-1$, using  \eqref{lem:A.8-2}  we obtain 
		\begin{align*}
			\frac{\pa }{\pa\lam_i} \Psi_2^{(i)}= &\,2\kap(\xi_i) \frac{\rho \lam_1^2 \cF_{\xi^*}^1(\xi_1^*)}{8\int_\Si V e^u \, \d v_g}\frac{- \sum_{j\neq i}  \frac{\kap(\xi_j)}{4} \lam_1^2 \cF_{\xi^*}^1(\xi_1^*)}{\lam_i \int_\Si V e^u \, \d v_g }+o\Big(\frac 1 {\lam_1}\Big);\\
			\frac{\pa }{\pa\lam_j} \Psi_2^{(i)}=& \,2\kap(\xi_i) \frac{\rho \lam_1^2 \cF_{\xi^*}^1(\xi_1^*)}{8\int_\Si V e^u \, \d v_g}\frac{ \frac{\kap(\xi_j)}{4} \lam_1^2 \cF_{\xi^*}^1(\xi_1^*)}{\lam_j \int_\Si V e^u \, \d v_g }+o\Big(\frac 1 {\lam_1}\Big),\quad i\neq j.
		\end{align*}
		Consequently, at point the $(\xi^*,\lam(\mu), e)$, we have
		\begin{align*}
			\sgn\det\Big(\frac{\pa \Psi_2}{\pa  \lam} \Big)= &\,	\sgn\det\left( \left(\begin{array}{ccccc}
				-( b_1-1)& 1& \ldots &1 & 1\\
				1& -(b_2-1)&  \ldots & 1& 1\\ \vdots&\vdots&     & \vdots&\vdots\\
				1& 1&\ldots & -(b_{p+q-1}-1)&1\\
				a_1 & a_2& \ldots & a_{p+q-1}& a_{p+q}
			\end{array}\right)+ A\right),
		\end{align*}
		in view of $\kap(\xi_i)>0$, $\cF^i_{\xi^*}(\xi_i^*)>0$ and $\lam_i>0$, where  	$b_i= \frac{\rho_*}{\kap(\xi_i)}>0$, $\forall\,i=1,\ldots,p+q$,
		\begin{gather*}
			a_i=\left\{\begin{aligned}
				&o(1) &&\text{ if }\,q\neq 0,\\
				&o\Big(\frac{\ln \lam_1}{\lam_1}\Big)& &\text{ if }\, q=0,
			\end{aligned}\right.\quad \forall\,i=1,\ldots,p+q-1,\quad  \quad a_{p+q}=\left\{\begin{aligned}
				&	-\cL_1(\xi^*)&& \text{ if }\, q\neq 0,\\
				&	-\frac{\cL_2(\xi^*)}{\lam_{p+q} } 	&&\text{ if }\, q=0,
			\end{aligned}\right.
		\end{gather*}
		and $A$ is a matrix satisfying that 
		$
		A=\left\{\begin{aligned}
			&	o(1) &&\text{ if }\, q\neq 0\\
			&	o\Big(\frac{\ln \lam_1}{\lam_1}\Big) &&\text{ if }\,  q=0
		\end{aligned}\right..
		$
		By direct calculation, we have 
		\[ \det \left( \begin{array}{cccc}
			-( b_1-1)& 1& \ldots &1 \\
			1& -(b_2-1)&  \ldots & 1\\ \vdots&\vdots&     & \vdots\\
			1& 1&\ldots & -(b_{p+q-1}-1)
		\end{array}\right)= (-1)^{p+q-1}(\Pi_{i=1}^{p+q-1}b_i)\Big(1-\sum_{i=1}^{p+q-1} \frac{1}{b_i}\Big).
		\]
		It follows that
		\begin{equation}\label{degree2}
			\sgn\det\Big(\frac{\pa \Psi_2}{\pa  \lam} \Big)\Big|_{(\xi^*,\lam(\mu), e)}=
			(-1)^{p+q}\,\sgn\,(\cL(\xi^*)).
		\end{equation}
		Finally, in view of  \eqref{eq:est_mu_best}, we obtain  $(\rho-\rho_*)\cL(\xi^*)>0$.  Consequently,  \eqref{eq:degree_finite_dim} follows directly from \eqref{degree1} and \eqref{degree2}.
	\end{proof}
	Before deriving the degree-counting formula, we first  prove  a Poincar\'e-Hopf index  type theorem:
	\begin{lem}\label{lem:deg_count_0}
		Let $\ff$ be  defined in \eqref{eq:reduce_fun_0},  and let $\Xi_{p,q}$  be given by \eqref{def:Xi_pq} for   $p,q\in \N_0$ with $2p+q\in\N$. Under the assumption of \hyperref[item:C1]{(C1)},
		the topological degree of the vector fields $\nabla\ff$  is given by
		\[
		\deg\,(\nabla \ff, \Xi_{p,q}, 0)=\left\{\begin{aligned}
			&	\prod_{i=0}^{p-1}	(\chi(\Si)-i) && \text{ if }\,q=0,\\
			&0	  &&  \text{ if }\, q\neq 0,
		\end{aligned}\right. 
		\]
		where $\chi(\Si)$ denotes the Euler characteristic of $\Si$. 
	\end{lem}
	\begin{proof}
		We apply the Poincar\'e-Hopf index theorem (see \cite[p.35]{milnor1965topology}) to compute the degree of the vector field $\nabla \ff$. 
		Since $\Xi_{p,q}$ is a Riemann surface with boundary, it necessitates a homotopic perturbation of the vector field to ensure it points outward at the boundary. 
		According to \cite[Lemma 4.1]{HB2024}, we have 
		$|\nabla \ff(\xi)|\to +\infty$ as  $\xi\to \pa (\Xi_{p,q})$. 
		A compact submanifold $ M^{\var}\subset  \Xi_{p,q}$ with smooth boundary $\pa  M^{\var}$ can be selected
		for some $\var>0$ sufficiently small. 	
		By   \cite[Proposition 3.42]{Hatcher2002}, there exists a collar neighborhood around the boundary $\pa  M^{\var}$, characterized by a homeomorphism $l_0(x)=(p_x, s_x)$,  which maps a neighborhood of $\pa  M^{\var}$ onto $\pa  M^{\var}\times [0, 1)$.
		
		Next, 	we  introduce a bump function $\eta: M^{\var}\to  [0,1]$, defined by $\tilde{\eta}(x)=\chi (2s_x)$, where $\chi $ is the same cut-off function as in \eqref{eq:cut_off}. Since $\pa  M^{\var}$ is smooth, we can define an outward unit normal vector field $N(x)$ for each point $x\in \pa  M^{\var}$. We then define the vector field
		\[
		v_{k,m}(x):= N(p_x) \tilde{\eta}(x)+(1-\tilde{\eta}(x))\nabla \ff(x),
		\] which is a smooth vector field on $M^{\var}$ that points outward at $\pa  M^{\var}$.
		
		We now construct a homotopy $H: M^{\var}\times [0,1] \to  TM^{\var}$, 
		$
		(x,t)\mapsto (1-t) \nabla\ff(x)+tv_{k,m}(x),
		$
		which leads to 
		$
		\deg\,(\nabla\ff, \Xi_{p,q}, 0)=\deg\,(v_{k,m},M^{\var}, 0).
		$
		Applying the Poincar\'e-Hopf index theorem,  we find that
		$\deg\,(v_{k,m},M^{\var}, 0)= \chi( M^{\var}).$  Since $\Xi_{p,q}$ can be continuously deformed into $M^{\var}$, we have $\chi(\Xi_{p,q}) = \chi(M^{\var})$, and thus 	\[\deg\,(\nabla\ff, \Xi_{p,q}, 0) = \chi(\Xi_{p,q}).\]
		
		Finally, it remains to compute the Euler characteristic of the configuration set $\Xi_{p,q}$. 
		Recall that $\Xi_{p,q}= X^p \bigcup \tilde{X}^q$,
		where 
		\[
		X^p:= \intS^p \setminus \{ (x_1,\ldots,x_p)\in\intS^p:  \text{$x_i=x_j$ for some $i\neq j$}  \}	\] and 
		\[\tilde{X}^q:= (\pa \Si)^q\setminus \{ (x_1,\ldots,x_q)\in(\pa \Si)^q:  \text{$x_i=x_j$ for some $i\neq j$}  \}. 	\] 
		We introduce a fibration $\tilde{\mu}(x_1, \ldots, x_p) = (x_1, \ldots, x_{p-1})$, where the fiber $\tilde{\mu}^{-1}(x_1, \ldots, x_{p-1})$ is homeomorphic to $\intS \setminus \{x_1, \ldots, x_{p-1}\}$. 
		By \cite[Theorem 9.3.1]{Spanier1966},  we have
		\begin{align*}
			\chi(X^{p})=&\,\chi(\intS\setminus\{x_1,\ldots,x_{p-1}\})
			\chi(X^{p-1})= (\chi(\intS)-(p-1))	\chi(X^{p-1})\\
			=& \,\ldots\ldots=  (\chi(\intS)-(p-1))\ldots  (\chi(\intS)-1)\chi(\intS)\\
			=&\, (\chi(\Si)-(p-1))\ldots  (\chi(\Si)-1)\chi(\Si),
		\end{align*}
		where we  used the fact  that $\intS$ is a retract by deformation of $\Si$. 	Similarly, by induction on $q$, we derive that 
		$ \chi(\tilde{X}^q)= (\chi(\pa \Si)-(q-1))(\chi(\pa \Si)-1)\chi(\pa \Si)=0,$
		since $\chi(\pa \Si)=0$. 
	\end{proof}
	
	\begin{lem}\label{lem:deg_jump}
		Let $m\in\N$. 	Under the assumptions of \hyperref[item:C1]{(C1)} and \hyperref[item:C2]{(C2)}, we have
		\begin{align*}
			d_m^+-d_m^-=& \,\sum_{\mytop{2p+q=m}{p,q\in\N_0}}\sum_{\xi^*\in \cV_{p,q}^+\cup \cV^-_{p,q}} \frac{1}{p!q!}(-1)^{p+\morse(\ff,\xi^*)}\\
			=&\, \sum_{2p=m,\, p\in \N_0} \frac{1}{p!} (-\chi(\Si)+(p-1))\ldots (-\chi(\Si)+1) (-\chi(\Si)),
		\end{align*}
		where $\cV_{p,q}^-$ and $\cV_{p,q}^+$  are defined in \eqref{eq:def_V-} and \eqref{eq:def_V+}, respectively, and $\chi(\Si)$ is the Euler characteristic of the surface $\Si$.  
	\end{lem}
	\begin{proof}
		By applying Lemma~\ref{lem:degee_reduced} together with \eqref{eq:def_d_plus} and~\eqref{eq:def_d_mius}, we obtain
		\begin{align*}
			d^+_m-d^-_m=&\, \sum_{2p+q=m} \frac{1}{p!q!}\Big( \sum_{\xi^*\in \cV_{p,q}^+} \deg\,(\nabla J_\rho, V(\mu,\xi^*),0) - \sum_{\xi^*\in \cV_{p,q}^-} \deg\,(\nabla J_\rho, V(\mu,\xi^*),0)\Big)\\
			=&\,\sum_{2p+q=m} \frac{1}{p!q!}\Big( \sum_{\xi^*\in \cV_{p,q}^+} \sgn\,(\rho-\rho_*)\,(-1)^{p+ \morse(\ff,\xi^*)}\\&\, - \sum_{\xi^*\in \cV_{p,q}^-} \sgn\,(\rho-\rho_*)\,(-1)^{p+ \morse(\ff,\xi^*)} \Big)\\
			=& \,\sum_{2p+q=m}  \sum_{\xi^*\in \cV_{p,q}^+\cup \cV_{p,q}^-} \frac{1}{p!q!} (-1)^{p+\morse(\ff,\xi^*)}. 
		\end{align*}
		Then, Lemma \ref{lem:deg_count_0} yields that
		\begin{align*}
			d_m^+-d_m^-=&\,\sum_{2p+q=m}\frac{(-1)^p}{p!q!}  \deg\,(\nabla \ff, \Xi_{p,q}, 0)\\
			=& \,\sum_{2p=m,\,p\in \N_0} \frac{(-1)^p}{p!} (\chi(\Si)-(p-1))\ldots  (\chi(\Si)-1)\chi(\Si) .
		\end{align*}
		This completes the proof.
	\end{proof}
	Having derived the explicit formula for the degree jump at the critical values  $4\pi m$,  we are now in a position to complete the proof of Theorem \ref{thm:blowup}.  
	\begin{proof}[Proof of Theorem \ref{thm:blowup}]
		According to Lemma \ref{lem:deg_jump}, for $m$ is an odd number we have  $d^+_m-d^-_m=0$, i.e., the blow-up solutions at $4\pi m$ contribute zero degree to the whole system. On the other hand, when $m=2\ell$ (an even number), for $\mu:=(\rho-4\pi m)/4\pi m$ close to 0, Lemma \ref{lem:deg_jump} gives that 
		\begin{align*}
			d_m^+-d_m^-=&\,\sum_{2p+q=m} \frac{1}{p!q!}\Big( \sum_{\xi^*\in \cV_{p,q}^+} \deg\,(\nabla J_\rho, V(\mu,\xi^*),0) - \sum_{\xi^*\in \cV_{p,q}^-} \deg\,(\nabla J_\rho, V(\mu,\xi^*),0)\Big)\\=& \, \frac{1}{\ell!}(-\chi(\Si)+(\ell-1))\ldots (-\chi(\Si)+1) (-\chi(\Si)),	
		\end{align*}
		which coincides with the result on closed Riemann surfaces in \cite[Lemma 5.2]{CL2003}.
		Under the assumption $\chi(\Si)<0$, we have $ d_m^+ -d_m^->0$.  Let $\mu=\frac{1}{k}$,  and for sufficiently large $k\in \N$, there exist integers $p^k, q^k$,  a configuration $(\xi^*)^k$, and a corresponding solution  $u_k\in V(\frac 1 k, (\xi^*)^k)$ of \eqref{maineq} with $\rho=\rho_k \to 4\pi m$, such that  $2p^k+q^k=m$ and  $(\xi^*)^k\in \cV_{p^k,q^k}^-\cup\cV_{p^k,q^k}^+$. Passing to a subsequence, $\{u_k\}$ blows up at $\xi^*:=(\xi_1^*, \ldots, \xi^*_{p+q})$ with the limiting configuration  $\xi^*\in \cV_{p,q}^-\cup\cV_{p,q}^+$ as $k\to +\infty$, where $2p+q=m$.  This verifies the existence of blow-up solutions to the problem \eqref{maineq}, thereby completing the proof.
	\end{proof}

	We now proceed to establish the degree-counting formula in the non-resonant case.
	\begin{proof}[Proof of Theorem \ref{thm:main}]
		It is well known that for $\rho \in (0, 4\pi)$, the functional $J_\rho $ is coercive and admits a unique minimizer. Herein, $d_\rho =1$ for all $\rho \in (0,4\pi)$. For $\rho\in (4\pi, 8\pi)$, we have $d_{\rho}=d_{1}^+=d_{1}^-+(d_{1}^+-d_{1}^-)=1$ in view of Lemma \ref{lem:deg_jump}.  Now, consider  $\rho\in (4\pi (m-1), 4\pi m)$ with $m \geq 3$ and denote $\ell:=[\frac{m-1}{2}]$.  By applying Lemma \ref{lem:deg_jump} once again, we obtain 
		\[	d_{\rho}=d_{m-1}^+=d^-_1+ \sum_{i=1}^{m-1}( d_i^+- d^-_i) 
		=
		1+\sum_{j=1}^{\ell} \binom{ j-1-\chi(\Si)}{j}    =
		\binom{ \ell-\chi(\Si)}{\ell}, \]
		where we have used 
		$d^-_{i+1}=d_i^+$ for $i=1,\ldots, m-1$. This completes the proof. 	
	\end{proof}



	\appendix
	\section{Technical estimates}\label{app:0}
	We provide some basic estimates from \cite{HBA20240829, EF2014} and omit some proofs of several intermediate results. 
	
	Let  $\xi,\zeta\in\Si$ and $\xi\neq \zeta$.
	For the isothermal coordinates $( U_{}(\xi),y_{\xi})$  and  $(U(\zeta),y_{\zeta})$, we  assume that 
	$U(\xi)\cap U(\zeta)=\emptyset$.  There exists a uniform radius
	$r_0>0$  such that $U_{2r_0}(\xi)\subset U(\xi)$ and $U_{2r_0}(\zeta)\subset U(\zeta)$.  And  we define the  cut-off  function
	$
	\chi_{\xi}(x):=\left\{\begin{aligned}
		&\chi\Big(\frac{|y_{\xi}|}{r_0}\Big) && \text{ if $x\in U(\xi)$}\\
		&0&& \text{ otherwise}
	\end{aligned}\right.,
	$
	where $\chi$ is given by \eqref{eq:cut_off}.\par 
	Let 
	\[
	f_{\xi}:=\frac{\Delta_g\chi_{\xi}}{|y_{\xi}|^2}+2 \langle\nabla \chi_{\xi},\nabla |y_{\xi}|^{-2}\rangle_g+\frac 2 {|\Si |_g} \int_{\Si }\frac 1 {r_0} \chi'\Big(\frac{|y_{\xi}|}{r_0}\Big)\frac{e^{-\varphi_{\xi}} }{ |y_{\xi}|^3} \,\d v_g\quad \text{ in }\, \Si,
	\]
	and
	$
	\Psi_{\lam,\xi}:=\chi_{\xi} ( \delta_{\lam,\xi}+\ln  \frac{\lam^2}{8})+\kap( \xi) H^g(\cdot, \xi).
	$
	By the definition of $f_{\xi}$,  it follows that as  $\lam\to+\infty$, we have
	$
	-\Delta_g \Psi_{\lam, \xi}= \frac{2}{\lam^2} f_{\xi}+O(\frac{1}{\lam^4})\text{ in } \Si,
	$
	so that 
	\[
	\int_{\Si} f_{\xi} \,\d v_g=-\frac{\lam^2}{2} \int_{\Si} \Delta_g \Psi_{\lam, \xi} \, \d v_g+O\Big(\frac{1}{\lam^2}\Big)=-\frac{\lam^2}{2} \int_{\pa \Si} \pa_{\nu_g} \Psi_{\lam, \xi} \, \d s_g+O\Big(\frac{1}{\lam^2}\Big)=O\Big(\frac{1}{\lam^2}\Big),
	\]
	which implies that $\int_{\Si} f_{\xi} \,\d v_g=0$.
	Therefore,  $F_{\xi}$ is  well-defined  as the unique solution of
	\begin{equation}\label{Fxi}
		\left\{\begin{aligned}
			&-\Delta_g F_{\xi}= f_{\xi}&& \text{  in }\,\Si, \\
			&\pa_{\nu_g}F_{\xi}=0 && \text{ on }\,\pa\Si, \\
			&\int_{\Si } F_{\xi} \,\d v_g=0,&&
		\end{aligned}\right.
	\end{equation}
	see, for instance, \cite{N2014,EF2014}. 
	
	Next, we present several lemmas concerning the asymptotic expansion of $ P \delta_{\lambda, \xi} $ and its derivatives as $ \lambda \to +\infty$, as proved in \cite{HBA20240829}.
	
	\begin{lem}[see Lemma 7.7 of \cite{HBA20240829}]\label{lem:projbubble_pointwise}
		As $\lam  \to + \infty$, the following expansion  holds in $C(\Si )$: 
		\begin{align*}
			P  \delta_{\lam,\xi}=&\,\chi_{\xi} \Big( \delta_{\lam,\xi}+\ln  \frac{\lam^2}{8} \Big)+\kap( \xi) H^g(\cdot, \xi)+h_{\lam,\xi}-\frac{2F_{\xi}}{\lam^2}  +O\Big(\frac{\ln  \lam }{\lam ^3}\Big)\nonumber\\=&\,2\chi_{\xi}\ln \Big( \frac{\lam^2}{1+\lam^2|y_{\xi}|^2}\Big)+ \kap( \xi) H^g(\cdot, \xi) +\frac{\kap( \xi)}{2|\Si |_g}\frac{\ln  \lam }{\lam^2}+O\Big(\frac{1}{\lam^2}\Big) ,
		\end{align*}
		where $F_{\xi}$ is given by \eqref{Fxi} and
		\[ h_{\lam,\xi}:=	\frac{\kap( \xi)}{2|\Si |_g}\frac{\ln  \lam }{\lam^2}+ \frac{2}{|\Si |_g}\frac{1}{\lam^2} \Big(\frac{\kap( \xi)}{8}+\int_{\Si } \chi_{\xi} \frac{1- e^{-\varphi_{\xi}}}{|y_{\xi}|^2} \,\d v_g -\int_{\Si } \frac{1}{r_0} \chi'\Big(\frac{|y_{\xi}|} {r_0}\Big) \frac{ \ln |y_{\xi}|}{|y_{\xi}|}\,\d v_g \Big). \] 
		In particular,  the following  expansion holds in $C_{loc}(\Si \setminus\{\xi\})$ as $\lam  \to + \infty$:
		\begin{align*}
			P  \delta_{\lam,\xi}=&\,\kap( \xi)G^g(\cdot, \xi)+h_{\lam,\xi}- \frac {2\chi_{\xi}} {|y_{\xi}|^2}\frac {1} {\lam^2}-\frac{2F_{\xi}}{\lam^2}+O\Big(\frac{\ln  \lam}{\lam ^3} \Big)\nonumber\\=&\,\kap( \xi)G^g(\cdot, \xi)+\frac{\kap( \xi)}{2|\Si |_g}\frac{\ln  \lam }{\lam^2}+O\Big(\frac{1}{\lam^2}\Big).
		\end{align*}
	\end{lem}

	\begin{lem}[see Lemma 7.4 of \cite{HBA20240829}] \label{lem:projbubble_xi_pointwise}  As $\lam  \to + \infty$, the following expansion  holds in $C(\Si )$:
		\begin{align*}
			\frac{\pa 	P\delta_{\lam,\xi}}{\pa\xi_j}=&\, \chi_{\xi} \frac{\pa  \delta_{\lam,\xi}}{\pa \xi_j}   +\kap( \xi) \frac{\pa H^g(\cdot,\xi)}{\pa \xi_j} -4\ln  (|y_{\xi}|)\frac{\pa \chi_{\xi}}{\pa \xi_j } +O \Big(\frac{\ln \lam}{\lam^2}\Big)\\=&\,-\frac{4 \chi_{\xi}\lam^2 (y_{\xi})_j}{1+\lam^2|y_{\xi}|^2}    +\kap( \xi) \frac{\pa H^g(\cdot,\xi)}{\pa \xi_j} -4\ln  (|y_{\xi}|)\frac{\pa \chi_{\xi}}{\pa \xi_j } +O \Big(\frac{\ln \lam}{\lam^2}\Big),
		\end{align*}
		where $j\in\{1,\ii(\xi)\}$.  In particular, we have
		$
		\frac{\pa 	P\delta_{\lam,\xi}}{\pa\xi_j}=\kap( \xi) \frac{\pa G^g(\cdot,\xi)}{\pa \xi_j}  +O(\frac{\ln \lam}{\lam^2}), 
		$  in $C_{loc}(\Si \setminus\{\xi\})$ as $\lam  \to + \infty$,
		where $j\in\{1,\ii(\xi)\}$.
	\end{lem}

	\begin{lem}[see Lemma 7.5 of \cite{HBA20240829}] \label{lem:projbubble_lamderivative__pointwise} As $\lam  \to + \infty$, the following expansion  holds in $C(\Si )$: 
		\begin{align*}
			\lam  \frac{\pa P\delta_{\lam,\xi}}{\pa \lam} =&\,\chi_{\xi}\Big(\lam  \frac{\pa  \delta_{\lam,\xi}}{\pa \lam} +2\Big)+h^1_{\lam,\xi} +\frac{4F_{\xi}}{\lam^{2}}+o\Big(\frac{1}{\lam^{2}}\Big)\nonumber\\=&\,\frac{4\chi_{\xi}}{1+\lam^2|y_{\xi}|^2}-\frac{\kap( \xi)}{|\Si |_g}\frac{\ln \lam}{\lam^2}+O\Big(\frac{1}{\lam^2}\Big),
		\end{align*}
		where 
		\[h^1_{\lam,\xi}:=-\frac{\kap( \xi)}{|\Si |_g}\frac{\ln \lam}{\lam^2}-\frac{4}{|\Si |_g}\frac{1}{\lam^2}\Big( \int_{\Si } \chi_{\xi}\frac{1-e^{-\varphi_{\xi}}}{|y_{\xi}|^2}\,\d v_g-\int_{\Si } \frac 1 {r_0}\chi'\Big(\frac {|y_{\xi}|} {r_0}\Big)\frac{\ln  |y_{\xi}|}{|y_\xi|} \,\d v_g\Big).\]
	\end{lem}

	\begin{lem}\label{lem:projbubble_derivative_integral}
		Assume that $\xi, \zeta \in \Si$ with $\xi \neq\zeta$, the following expansions 	hold as $\lam,\mu \to +\infty$: 
		\begin{align*}
			\Big\langle P\delta_{\lam,\xi}, \frac{1}{\lam}\frac{\pa P\delta_{\lam,\xi}}{\pa \xi_l}   \Big\rangle=& \,\frac{\kap^2( \xi)}{2\lam} \frac{\pa  R^g(\xi) }{\pa \xi_l}+ O\Big(\frac{\ln \lam }{\lam ^{3} }\Big), \\
			\Big\langle P\delta_{\lam,\xi},\frac{1}{\mu} \frac{\pa P\delta_{\mu,\zeta}}{\pa \zeta_l}   \Big\rangle =&\, \frac{\kap( \xi)\kap( \zeta)}{\mu}	\frac{\pa G^g(\xi,\zeta)}{\pa \zeta_l}	+O\Big( \frac{\ln \mu }{\mu ^{3} }\Big), \\
			\Big\langle P\delta_{\lam,\xi}, \lam  \frac{\pa P\delta_{\lam,\xi}}{\pa\lam}   \Big\rangle=&\,2\kap( \xi)-\frac{\kap^2( \xi)}{|\Si |_g} \frac{\ln  \lam }{\lam^2} +  O\Big(\frac{1}{\lam^2}\Big),  \\
			\Big\langle P\delta_{\lam,\xi}, \mu\frac{\pa P\delta_{\mu,\zeta}}{\pa \mu }   \Big\rangle=& \,-\frac{\kap( \xi)\kap( \zeta)}{|\Si |_g} \frac{\ln  \mu}{\mu^2}+  O\Big(\frac{1}{\mu^2}\Big),
		\end{align*}
		where $j\in\{1,\ii(\xi)\}$ and $l\in\{1,\ii(\zeta)\}$.
	\end{lem}
	\begin{proof}
		From   \eqref{deltachange} and  \eqref{projectedbubble}, we have  that
		$ 
		\langle P\delta_{\lam,\xi}, \frac{1}{\lam}\frac{\pa P\delta_{\lam,\xi}}{\pa \xi_j}  \rangle =\int_{\Si } \frac{ \chi_{\xi}}{\lam} \frac{\pa P\delta_{\lam, \xi}}{\pa \xi_j} e^{\delta_{\lam,\xi}-\varphi_{\xi}} \,\d v_g.
		$
		Now we deduce from Lemma 	 \ref{lem:projbubble_xi_pointwise}  that, for any $j\in\{1,\ii(\xi)\}$,
		\begin{align*}
			\Big\langle P\delta_{\lam,\xi}, \frac{1}{\lam}\frac{\pa P\delta_{\lam,\xi}}{\pa \xi_j}  \Big\rangle 	=&\,
			\int_{\Si }\frac{\chi_{\xi}}{\lam }  \Big(-\frac{4 \chi_{\xi}\lam^2 (y_{\xi})_j}{1+\lam^2|y_{\xi}|^2}   + \kap( \xi) \frac{\pa H^g(\cdot,\xi) }{\pa \xi_j} -4\ln | y_{\xi}| \frac{\pa \chi_{\xi}}{\pa \xi_j}+ O\Big(\frac{\ln \lam }{\lam ^{2} }\Big)\Big)	 e^{ \delta_{\lam,\xi}-\varphi_{\xi}}\,\d v_g
			\\=&\,	\int_{\Si }\frac{\chi_{\xi}}{\lam }  \Big(\kap( \xi) \frac{\pa H^g(\cdot,\xi) }{\pa \xi_j} + O\Big(\frac{\ln \lam }{\lam ^{2} }\Big)\Big)	 e^{ \delta_{\lam,\xi}-\varphi_{\xi}}\,\d v_g
			\\	=&\,	\frac{\kap^2( \xi) }{2\lam} \frac{\pa  R^g(\xi)}{\pa \xi_j}+ O\Big(\frac{\ln \lam }{\lam ^{3} }\Big), 
		\end{align*}
		where we  used  the Taylor expansion  around $\xi$:
		\[  
		\frac{\pa H^g(y_{\xi}^{-1}(y),\xi) }{\pa \xi_j} =	\frac{\pa H^g(y_{\xi}^{-1}(y),\xi) }{\pa \xi_j}\Big|_{y=0}+\nabla_y \frac{\pa H^g(y_{\xi}^{-1}(y),\xi) }{\pa \xi_j}\Big|_{y=0}\cdot  y+O(|y|^2),
		\] 
		\[  
		\frac{\pa^2 H^g(y_{\xi}^{-1}(y),\xi)}{ \pa y_2\pa \xi_j}\Big|_{y=0}= \pa _{x_2}\pa _{\xi_1} H^g(x,\xi)\Big|_{x=\xi}=\pa _{\xi_1} \pa _{x_2} H^g(x,\xi)\Big|_{x=\xi}=0, 
		\]
		and 
		\[  
		\int_{\B^+_{r_0}}\frac{1}{\lam}
		\frac{8\lam^2y_2}{(1+\lam^2|y|^2)^2}\,\d y=\frac{16}{\lam^2}\Big(\frac{1}{2}\ln (1+r_0^2\lam^2)+\frac{1}{1+r_0^2\lam^2}-\frac{1}{2}\Big), 
		\] 
		in the last equality. 
		Observe that
		$
		\frac{\pa^2 G^g(y_{\xi}^{-1}(y),\xi)}{\pa y_2\pa \xi_j}\Big|_{y=0}=\pa _{x_2}\pa _{\xi_1} G^g(x,\xi)|_{x=\xi}=0.
		$
		Similarly, we have 
		\begin{align*}
			\Big\langle P\delta_{\lam,\xi}, \frac{1}{\mu}\frac{\pa P\delta_{\mu,\zeta}}{\pa \zeta_l}  \Big\rangle 	=&\,
			\int_{\Si }\frac{\chi_{\xi}}{\mu }  \frac{\pa P\delta_{\mu,\zeta}}{\pa \zeta_l}	 e^{ \delta_{\lam,\xi}-\varphi_{\xi}}\,\d v_g
			\\=&\,\int_{\Si }  \frac{\chi_{\xi}}{\mu}\Big(\kap( \zeta ) \frac{\pa G^g(\cdot,\zeta )}{\pa \zeta_l} +O\Big(\frac{\ln \mu}{\mu ^{2} }\Big) \Big) e^{ \delta_{\lam,\xi}-\varphi_{\xi}} \,\d v_g
			\\	=&\,	\frac{\kap( \xi)\kap( \zeta) }{\mu} \frac{\pa  G^g(\xi,\zeta)}{\pa \zeta_l} + O\Big(\frac{\ln \mu}{\mu ^{3} }\Big).
		\end{align*}
		It remains to establish the last two estimates in Lemma \ref{lem:projbubble_derivative_integral}. 
		By Lemma \ref{lem:projbubble_lamderivative__pointwise}, we obtain
		\begin{align*}
			\Big\langle P\delta_{\lam,\xi}, \lam\frac{\pa P\delta_{\lam,\xi}}{\pa \lam}  \Big\rangle =&\,\int_{\Si } \chi_{\xi} \lam  \frac{\pa P\delta_{\lam, \xi} }{\pa \lam}e^{ \delta_{\lam,\xi}-\varphi_{\xi}}\,\d v_g\\
			=	&\,	\int_{\Si }	\chi_{\xi} \Big(\frac{4\chi_{\xi}}{1+\lam^2|y_{\xi}|^2}-\frac{\kap( \xi)}{|\Si |_g}\frac{\ln \lam}{\lam^2}+  O\Big(\frac{1}{\lam^2}\Big) \Big)e^{ \delta_{\lam,\xi}-\varphi_{\xi}}	\,\d v_g\\
			=& \,	2\kap( \xi)-\frac{\kap^2( \xi)}{|\Si |_g} \frac{\ln \lam}{\lam^2}+ O\Big(\frac{1}{\lam^2}\Big) 
		\end{align*}
		and 
		\begin{align*}
			\Big\langle P\delta_{\lam,\xi}, \mu   \frac{\pa P\delta_{\mu   ,\zeta }}{\pa \mu   }  \Big\rangle =&\int_{\Si } \chi_{\xi} \mu     \frac{\pa P\delta_{\mu   , \zeta } }{\pa \mu   }e^{ \delta_{\lam,\xi}-\varphi_{\xi}}\,\d v_g\\
			=	&\,\int_{\Si }\chi_{\xi}  \Big(  -\frac{\kap( \zeta )}{|\Si |_g}\frac{\ln  \mu}{\mu^2}+O\Big(\frac{1}{\mu ^2}\Big)  \Big)e^{ \delta_{\lam,\xi}-\varphi_{\xi}} \,\d v_g
			\\=&\,-\frac{\kap( \zeta )\kap( \xi)}{|\Si |_g}\frac{\ln \mu  }{\mu ^2}  + 	O\Big(\frac{1}{\mu ^2}\Big).   
		\end{align*}
		Lemma \ref{lem:projbubble_derivative_integral} is established.
	\end{proof}
	\begin{lem}\label{lem:projbubble_integral_interaction}
		Under the assumptions of Lemma \ref{lem:projbubble_derivative_integral},  the following expansions  hold as $\lam,\mu  \to +\infty$: 
		\begin{equation}
			\label{lem:projbubble_integral_interaction-1}
			\langle P\delta_{\lam,\xi}, P\delta_{\lam,\xi}\rangle = 4\kap( \xi) \ln  \lam  +\kap^2( \xi) R^g(\xi)-2\kap( \xi) +\frac{\kap^2( \xi)}{|\Si |_g}\frac{\ln  \lam }{\lam^2} + O\Big(\frac{1}{\lam^2}\Big) 
		\end{equation}
		and		
		\begin{equation}\label{lem:projbubble_integral_interaction-2}
			\langle P\delta_{\lam,\xi}, P\delta_{\mu,\zeta}\rangle =\kap( \xi)\kap( \zeta) G^g(\xi,\zeta)  +\frac{\kap( \xi)\kap( \zeta)}{2|\Si |_g}\frac{\ln  \lam }{\lam^2}+\frac{\kap( \xi)\kap( \zeta)}{2|\Si |_g}\frac{\ln  \mu}{\mu^2} + O\Big(\frac{1}{\lam^2}+\frac{1}{\mu^2}\Big).
		\end{equation}
	\end{lem}
	\begin{proof}  It follows from \eqref{deltachange} and \eqref{projectedbubble} that
		$
		\langle P\delta_{\lam,\xi}, P\delta_{\lam,\xi}\rangle	=\int_{\Si }\chi_{\xi}  P\delta_{\lam , \xi}  e^{ \delta_{\lam,\xi}-\varphi_{\xi}} \,\d v_g.
		$
		Note that 
		\[
		\int_{\Si}\chi_{\xi}e^{ \delta_{\lam,\xi}-\varphi_{\xi}}\, \d v_g=\int_{B_{2r_0}^{\xi}}\chi\Big(\frac{|y|}{r_0}\Big)e^{ \delta_{\lam,\xi}}\, \d y=\kap( \xi)+O\Big(\frac{1}{\lam^2}\Big).
		\]
		Then we deduce from		Lemma \ref{lem:projbubble_pointwise} that 
		\begin{align*}
			\langle P\delta_{\lam,\xi}, P\delta_{\lam,\xi}\rangle	
			=& \,\int_{\Si } \chi_{\xi}\Big[2\chi_{\xi}\ln\Big( \frac{\lam^2}{1+\lam^2|y_{\xi}|^2}\Big)+\kap( \xi) H^g(\cdot, \xi)+\frac{\kap( \xi)}{2|\Si |_g}\frac{\ln  \lam }{\lam^2}+O\Big(\frac{1}{\lam^2}\Big) \Big]  e^{\delta_{\lam,\xi}-\varphi_\xi} \,\d v_g \\
			=&\,\int_{\Si } \chi_{\xi}    \Big[2\chi_{\xi}\ln \Big( \frac { \lam^2|y_{\xi}|^2}{1 +\lam^2|y_{\xi}|^2 }\Big) +\kap( \xi) G^g(\cdot,\xi) \Big]  e^{\delta_{\lam,\xi}-\varphi_\xi}\,\d v_g +\frac{\kap^2( \xi)}{2|\Si |_g}\frac{\ln  \lam }{\lam^2} + O\Big(\frac{1}{\lam^2}\Big)  \\
			=&\,
			2\int_{\Si}\chi_{\xi}^2\ln \Big( \frac { \lam^2|y_{\xi}|^2}{1 +\lam^2|y_{\xi}|^2 }\Big) e^{\delta_{\lam,\xi}-\varphi_\xi}\,\d v_g \\&\,+ \kap( \xi)\int_{\Si}\chi_{\xi}G^g(\cdot,\xi)e^{\delta_{\lam,\xi}-\varphi_\xi}\,\d v_g+\frac{\kap^2( \xi) }{2|\Si |_g}\frac{\ln  \lam}{\lam^2} + O\Big(\frac{1}{\lam^2}\Big) \\
			:	=& \, I_1+I_2+\frac{\kap^2( \xi) }{2|\Si |_g}\frac{\ln  \lam}{\lam^2} + O\Big(\frac{1}{\lam^2}\Big).
		\end{align*}
		A direct calculation shows that 
		\[
		I_1=2\int_{B^{\xi}_{2r_0}}\chi^2\Big(\frac{|y_{\xi}|}{r_0}\Big)\ln \Big( \frac { \lam^2|y_{\xi}|^2}{1 +\lam^2|y_{\xi}|^2 }\Big) \frac{8\lam^2}{ (1+\lam^2|y_{\xi}|^2)^2 } \,\d y=-2\kap( \xi)+O\Big(\frac{1}{\lam^2}\Big).
		\]
		Furthermore, 	using \eqref{Green}, \eqref{Greenrepresentation},  \eqref{projectedbubble}  and Lemma \ref{lem:projbubble_pointwise}, we have
		\[	
		I_2=\kap( \xi)P\delta_{\lam,\xi}(\xi)=4\kap( \xi)\ln \lam +\kap^2( \xi)R^g(\xi)+\frac{\kap^2( \xi) }{2|\Si |_g}\frac{\ln  \lam}{\lam^2} + O\Big(\frac{1}{\lam^2}\Big).
		\]
		Thus, the validity of \eqref{lem:projbubble_integral_interaction-1} follows immediately from the above estimates. The proof of  \eqref{lem:projbubble_integral_interaction-2} can be completed through analogous calculations, and for brevity, we omit the details here.
	\end{proof}

	Let $m=2p+q\in \N$, where $p,q\in \N_0$. In what follows, we assume that  $\xi:=(\xi_1,\ldots,\xi_{p+q})$ lies in a compact subset of $\Xi_{p,q}$ as defined in  \eqref{def:Xi_pq}. 
	For each $i=1,\ldots,p+q$, let $(y_{\xi_i}, U(\xi_i))$ be an isothermal coordinate  around $\xi_i$. We may assume that  
	$U(\xi_i)\cap U(\xi_j)=\emptyset$ for $i\neq j $, and $U(\xi_i)\cap \pa\Si =\emptyset$  for $i=1,\ldots,p$. Moreover,  there exists a uniform constant $r_0>0$  such that $U_{2r_0}(\xi_i)\subset U(\xi_i)$  for all $i$. Define the  cut-off  function
	$
	\chi_{\xi_i}(x):=\left\{\begin{aligned}
		&\chi\Big(\frac{|y_{\xi_i}|}{r_0}\Big) && \text{ if $x\in U(\xi_i)$}\\
		&0&& \text{ otherwise}
	\end{aligned}\right.,
	$
	for $i=1,\ldots,p+q$.

	The following two technical lemmas are given in \cite{ABL2017}. For the reader's convenience, we include them here without proofs.
	
	\begin{lem}[see Lemma A.4 of  \cite{ABL2017}]\label{lem:w_error_integral_estimates}
		There exists a constant $C=C(\Si)>0$ such that for  $i=1,\ldots,p+q$, we have
		\begin{gather*}
			\Big(\int_{\Si }  |w| e^{\delta_{\lam_i,\xi_i}}\, \d v_g\Big)^{\frac{1}{\beta}} \leq C\sqrt{\beta}\|w\|, \quad \forall\, w \in E_{\lam,\xi }^{p,q},\quad \forall\, \beta\geq 1,\\
			\int_{\Si} h|e^w-1-w|e^{\delta_{\lam_i,\xi_i}}\, \d v_g  \leq C\Big(\int_{\Si} h^2 e^{\delta_{\lam_i,\xi_i}}\, \d v_g\Big)^{1 / 2}\|w\|^2, \quad \forall\, w \in E_{\lam,\xi }^{p,q},\quad \forall\,h \in L^{\infty}(\Si), 
		\end{gather*}
		and
		\[
		\int_{\Si} \Big|e^w-1-w-\frac{w^2}{2}\Big| e^{\delta_{\lam_i,\xi_i}}\, \d v_g   \leq C\|w\|^3, \quad \forall\, w \in E_{\lam,\xi }^{p,q} .
		\]
	\end{lem}
	
	\begin{lem}[see Lemma A.5 of\cite{ABL2017}] \label{lem:A.5} 
		For any fixed constant  $\beta\in [0,2]$, there exists a constant $C=C(\beta)>0$ such that, for any  $i=1,\ldots,p+q$ and for all $w \in E_{\lam,\xi }^{p,q}$,   the following estimate holds:
		\[	
		\int_{\Si}\chi_{\xi_i}  |y_{\xi_i}|^\beta|w| e^{\delta_{\lam_i,\xi_i}}\, \d v_g \leq\left\{\begin{aligned}
			&C\frac{\|w\|}{\lam_i ^{\beta}} && \text { if }\, \beta \in[0,2), \\
			&C\frac{\ln^{\frac{3}{2}}\lam_i }{\lam_i^2}   \|w\| && \text { if }\, \beta = 2.
		\end{aligned} \right.
		\]
	\end{lem}
	
	\begin{lem}\label{lem:Ve^{u_0}_pointwise}
		Let $u=  u_0+w\in V_{\rho}^\eta(p,q,\var)$ with $\rho>0$, where $u_0= \sum_{i}\al_i P\delta_{\lam_i,\xi_i}$ and
		$w\in E_{\lam,\xi }^{p,q}$. Then, there exists some small constant $r_0>0$ such that 
		\[
		V e^{u_0}=
		\frac{\lam_i^{4\al_i}\cF^i_{\xi}\rg^i_{\xi}}{ ( 1+\lam_i^2|y_{\xi_i}|^2)^{2\al_i}}\Big(1  +\sum_{j} \frac{\al_j\kap( \xi_j)}{2|\Si |_g} \frac{\ln  \lam_j}{\lam_j^2}+ O\Big( \sum_j\frac{\al_j}{\lam_j^2} \Big)\Big) \quad \text{ in }\, U_{r_0}(\xi_i)
		\]
		for any  $i=1,\ldots,p+q$, and 
		\[
		V e^{u_0}=V\Big(1+\sum_{j}  \frac{\al_j\kap( \xi_j)}{2|\Si |_g} \frac{\ln  \lam_j}{\lam_j^2}+O\Big(\sum_{j}\frac{\al_j}{\lam_j^2} \Big)\Big)\exp\Big\{  \sum_{j} \al_j \kap( \xi_j)G^g(\cdot,\xi_j)\Big\} \quad \text{  in }\,\Si \setminus \cup_{j} U_{r_0}(\xi_j) ,
		\]
		where  $\cF^i_{\xi}$ and 	$\rg_{\xi}^i$ are defined in \eqref{eq:def_F} and \eqref{def:rg}, respectively.
	\end{lem}
	\begin{proof}
		Let us 	fix $i\in \{1,\ldots,p+q\}$. 	It follows from  Lemma \ref{lem:projbubble_pointwise} that   for $x\in U_{r_0}(\xi_i)$,
		\begin{align*}
			V(x) e^{u_0(x)}=& \,V(x) \exp\Big\{ \al_i\Big( 2\ln \Big(\frac{\lam_i^2}{1+\lam_i^2|y_{\xi_i}|^2}\Big)+ \kap( \xi_i)H^g(x,\xi_i) +  \frac{\kap( \xi_i)}{2|\Si |_g}\frac{\ln  \lam_i}{\lam_i^2}+O\Big(\frac{1}{\lam_i^2}\Big)\Big)\\&\,+\sum_{j\neq i}\al_j\Big(\kap( \xi_j)G^g(x,\xi_j)+\frac{\kap( \xi_j)}{2|\Si |_g}\frac{\ln  \lam_j}{\lam_j^2}+ O\Big(\frac{1}{\lam_j^2}\Big)\Big) \Big\}\\
			=&\,\frac{\lam_i^{4\al_i}(\cF^i_{\xi}\rg^i_{\xi})(x)}{ ( 1+\lam_i^2|y_{\xi_i}|^2)^{2\al_i}}\Big(1  +\sum_{j} \frac{\al_i\kap( \xi_j)}{2|\Si |_g} \frac{\ln  \lam_j}{\lam_j^2}+O\Big(\sum_j\frac{\al_j}{\lam_j^2}\Big) \Big),
		\end{align*}
		and for $x\in \Si \setminus\cup_{j} U_{r_0}(\xi_j)$, 
		\begin{align*}
			V(x) e^{u_0(x)}=&\, V(x) \exp\Big\{  \sum_{j}\al_j \Big[  \kap( \xi_j)G^g(x,\xi_j)  +  \frac{\kap( \xi_j)}{2|\Si |_g} \frac{\ln  \lam_j}{\lam_j^2}+ O\Big(\frac{1}{\lam_j^2}\Big) \Big]\Big\}\\
			=&\,V(x)  \Big(1+\sum_{j} \frac{\al_j\kap( \xi_j)}{2|\Si |_g} \frac{\ln  \lam_j}{\lam_j^2} + O\Big(\sum_{j} \frac{\al_j}{\lam_j^2}\Big)\Big) \exp\Big\{  \sum_{j}  \al_j\kap( \xi_j) G^g(x,\xi_j)\Big\}. 
		\end{align*}
		Lemma \ref{lem:Ve^{u_0}_pointwise}  is thus established.
	\end{proof}

	For $i=1,\ldots, p+q$, we define the auxiliary function as follows:
	\begin{equation}\label{def:Phi}
		\Phi _i(x):=\frac {1} {( 1+\lam_i^2|y_{\xi_i}(x)|^2)^{2\al_i-2}},\quad x \in  U(\xi_i).
	\end{equation}
	If $|\al_i-1|\ln \lam_i=o(1)$,  then  $\Phi _i(x)=1+o(1)$ uniformly on $U(\xi_i)$. In general, if $|\al_i-1|\leq \frac 1 8$, there exists a constant $C>0$ such that for all $x\in U(\xi_i)$,
	\begin{equation}\label{Phibound}
		|\Phi_i (x)-1|=	|\Phi _i (x)-\Phi _i (\xi_i)|= \Big|\int_0^1 \frac{4(1-\al_i)\lam_i^{2}|y_{\xi_i}|^2t }{( 1+\lam_i^2
			|y_{\xi_i}|^2t^2)^{2\al_i-1} } \, \d t\Big|\leq  C   |\al_i-1|\sqrt{\lam_i|y_{\xi_i}|}.
	\end{equation}
	
	\begin{lem}\label{lem:A.7}
		For    $u_0= \sum_{i}\al_i P\delta_{\lam_i,\xi_i}$ and
		$w\in E_{\lam,\xi }^{p,q}$, we have 
		\begin{equation} \label{lem:A.7-1}
			\int_{U_{r_0}(\xi_i)} \Phi _i  w   e^{\delta_{\lam_i,\xi_i}-\varphi_{\xi_i}} \,\d v_g=O\Big(\Big( |\al_i-1|+ \frac{1}{\lam_i^2}\Big)\|w\|\Big),\quad \forall\, 1\leq i\leq p+q,
		\end{equation}
		and 	
		\begin{equation}\label{lem:A.7-2}
			\int_{\Si } V e^{u_0} w \,\d v_g = O\Big(\sum_{i} \lam_i^{4\al_i-2}\Big(|\al_i-1|+\frac{	|\nabla_{\xi_i}\cF_{\xi}^i(\xi_i)|+|\mathfrak{f}_2(\xi_i)| }{\lam_i} +\frac{\ln^{\frac{3}{2}} \lam_i}{\lam_i^2}\Big) \|w\|\Big)+O(\|w\|),
		\end{equation}
		where $\cF_{\xi}^i$ and $\mathfrak{f}_2$ are defined in  \eqref{eq:def_F}  and \eqref{def:frak_f2}, respectively.
	\end{lem}
	\begin{proof}
		We first derive that from \eqref{Phibound} that  
		\[
		\int_{U_{r_0}(\xi_i)}  \Phi _i  w  e^{\delta_{\lam_i,\xi_i}}\,\d v_g=\int_{U_{r_0}(\xi_i)}  w e^{\delta_{\lam_i,\xi_i}} \,\d v_g+O\Big(|\al_i-1|\int_{U_{r_0}(\xi_i)}  \sqrt{\lam_i|y_{\xi_i}|} |w|e^{\delta_{\lam_i,\xi_i}} \,\d v_g  \Big).
		\]
		Next, by applying \eqref{orthogonality} and  the  Sobolev embedding, we get
		\begin{equation} \label{lem:A.7-5}
			\int_{U_{r_0}(\xi_i)}  w e^{\delta_{\lam_i,\xi_i}-\varphi_{\xi_i}} \,\d v_g=-\int_{\Si\backslash U_{r_0}(\xi_i)}  \chi_i  w e^{\delta_{\lam_i,\xi_i}-\varphi_{\xi_i}} \,\d v_g=O\Big(\frac{\|w\|}{\lam_i^2}\Big)	.
		\end{equation} 
		Now, by combining \eqref{lem:A.7-5} with Lemma \ref{lem:A.5}, the proof of \eqref{lem:A.7-1} is complete. It remains to establish \eqref{lem:A.7-2}. 
		
		Indeed, it follows from Lemma \ref{lem:Ve^{u_0}_pointwise} and the  Sobolev embedding that
		\begin{align}
			\int_{\Si } V e^{u_0} w\,\d v_g=& \,
			\sum_{i} \int_{U_{r_0}(\xi_i)}Ve^{u_0} w \,\d v_g + \int_{\Si \setminus \cup_{j} U_{r_0}(\xi_j)} Ve^{u_0} w \,\d v_g \nonumber\\
			=&\, \sum_{i} \frac{\lam_i^{4\al_i-2}}{8}\int_{U_{r_0}(\xi_i)} e^{\delta_{\lam_i,\xi_i}}\Phi _i  \cF^i_{\xi} \rg^i_{\xi} \Big(1+ O\Big( \frac{\ln \lam_i}{\lam_i^2}\Big)\Big) w\,\d v_g + O(\|w\|).\label{lem:A.7-3}
		\end{align}
		Then, 	using the Taylor expansion, along with \eqref{Phibound} and  \eqref{lem:A.7-5},    we obtain
		\begin{align}
			&\int_{U_{r_0}(\xi_i)} e^{\delta_{\lam_i,\xi_i}}\Phi _i  \cF^i_{\xi} \rg^i_{\xi} w\,\d v_g\nonumber\\
			=& \,(\cF^i_{\xi} \rg^i_{\xi})(\xi_i)\int_{B^{\xi_i}_{r_0}}  \Phi _i  w e^{\tilde\delta_{\lam_i,\xi_i}} \,\d y+ O\Big(|\nabla_g(e^{\varphi_{\xi_i}}\cF^i_{\xi} \rg^i_{\xi})(\xi_i)|\int_{B^{\xi_i}_{r_0}}|\Phi _i ||y| |w| e^{\tilde\delta_{\lam_i,\xi_i}} \,\d y \Big)\nonumber\\
			&\,+ O\Big(\int_{B^{\xi_i}_{r_0}}|\Phi _i ||y|^2 |w|  e^{\tilde\delta_{\lam_i,\xi_i}} \,\d y \Big)\nonumber\\=&\, (\cF^i_{\xi} \rg^i_{\xi})(\xi_i)\int_{B^{\xi_i}_{r_0}}  w  e^{\tilde\delta_{\lam_i,\xi_i}}\,\d y+ O((\cF^i_{\xi} \rg^i_{\xi})(\xi_i)|\al_i-1|\|w\|)\nonumber\\&\,+ O\Big(\frac{|\nabla_g(e^{\varphi_{\xi_i}}\cF^i_{\xi} \rg^i_{\xi})(\xi_i)|}{\lam_i}\|w\| \Big)+ O\Big(\frac{\ln^{\frac{3}{2}} \lam_i}{\lam_i^2} \|w\|  \Big)\nonumber\\=&\,O\Big(\Big(|\al_i-1|+ \frac{|\nabla(e^{\varphi_{\xi_i}}\cF^i_{\xi} \rg^i_{\xi})(\xi_i)|}{\lam_i}+ \frac{\ln^{\frac{3}{2} } \lam_i}{\lam_i^2}\Big)\|w\|\Big). \label{lem:A.7-4}
		\end{align} 
		Considering \eqref{varphixi} and $\rg_{\xi}^i(\xi_i)=1+O(|\al _i-1|)$, we have 
		$	\nabla(e^{\varphi_{\xi_i}}\cF^i_{\xi} \rg^i_{\xi})(\xi_i)= 	\mathfrak{f}_2(\xi_i) 	 + O(|\al _i-1|).$
		Therefore,  \eqref{lem:A.7-2} follows immediately from  \eqref{lem:A.7-3} and \eqref{lem:A.7-4}.
	\end{proof}
	\begin{lem}\label{lem:Ve^u_integral}
		Let $u=  u_0+w\in V_{\rho}^\eta(p,q,\var)$ with $\rho>0$, where  
		$u_0= \sum_{i}\al_i P\delta_{\lam_i,\xi_i}$ and $w\in E_{\lam,\xi }^{p,q}$. 
		Then, we have 
		\begin{align}
			\int_{\Si } V e^u \,\d v_g =&\,\sum_{i}  \frac{\kap( \xi_i)(\cF^i_{\xi} \rg^i_{\xi})(\xi_i)}{8(2\al_i-1)}\lam_i^{4\al_i-2} + \frac{\pi}{2}\sum_{i}    \mathfrak{f}_2(\xi_i)\lam_i^{4\al_i-3} \nonumber\\
			&\,+
			O\Big(\sum_{i} \lam_i^{4\al_i-2}\Big(
			\frac{\ln \lam_i}{\lam_i^{2}}+ |\al_i-1|^2 +\|w\|^2\Big)+1\Big),\label{lem:A.8-1}
		\end{align}
		where $\mathfrak{f}_2$ is defined in  \eqref{def:frak_f2}.	Furthermore, if $\sum_{i}|\al_i-1|\ln \lam_i=o(1)$,  then \eqref{lem:A.8-1} can be refined as follows:
		\begin{align}
			\int_{\Si } V e^u \,\d v_g =&\,\sum_{i}  \frac{\kap( \xi_i)(\cF^i_{\xi} \rg^i_{\xi})(\xi_i)}{8(2\al_i-1)}\lam_i^{4\al_i-2} \Big(1+\sum_j \frac{\kap(\xi_j)}{2|\Si|_g}\frac{\ln \lam_j}{\lam_j^2}\Big)\nonumber \\
			&\,+\frac{\pi}{2} \sum_{i}  \mathfrak{f}_2(\xi_i)\lam_i^{4\al_i-3}+\sum_{i}
			\frac{\kap( \xi_i)}{{16}} \mathfrak{f}_3(\xi_i)\ln  \lam_i\nonumber \\
			&\,+O\Big(\sum_{i} \lam_i^{2}(
			|\al_i-1|^2 +\|w\|^2)+1\Big),\label{lem:A.8-2}
		\end{align}
		where  
		\begin{align}
			\mathfrak{f}_3(x):=&\,
			\Delta_g (\cF^i_{\xi}\rg^i_{\xi}e^{{\varphi}_{\xi_i}})(x)\nonumber\\=&\,
			\left\{	\begin{aligned}
				&			(	\Delta_g(\cF^i_{\xi}\rg^i_{\xi})
				- 2K_g \cF^i_{\xi}\rg^i_{\xi})\Big|_{x} && \text{ if }\,x\in \intS,\\
				&	(\Delta_g(\cF^i_{\xi}\rg^i_{\xi})
				- 2K_g \cF^i_{\xi}\rg^i_{\xi} + 4 k_g  \cF^i_{\xi}\rg^i_{\xi}(\pa _{\nu_g} \ln V+k_g))\Big|_{x} && \text{ if }\,x \in \pa \Si.
			\end{aligned}	\right.\label{def:frak_f3}
		\end{align}
	\end{lem}
	\begin{proof}
		We  decompose the integral	$\int_{\Si } V e^u \,\d v_g$ into the following parts:
		\[	\int_{\Si } V e^u \,\d v_g=	\int_{\Si } V e^{u_0}\,\d v_g+ \int_{\Si } V e^{u_0} w\,\d v_g+ \int_{\Si } V e^{u_0}(e^w -1-w)\,\d v_g
		:=I_1+I_2+I_3. \]
		We proceed to estimate $I_1$, $I_2$ and $I_3$ in the following steps.

		\textbf{Step 1: Estimate of  $I_1$.}  	By Lemma \ref{lem:Ve^{u_0}_pointwise}, we have 
		\begin{align}
			I_1=& \,\sum_{i} \int_{U_{r_0}(\xi_i)} V e^{u_0} \,\d v_g + \int_{\Si \setminus \cup_{j} U_{r_0}(\xi_j)}V e^{u_0} \,\d v_g \nonumber\\
			=&\,\sum_{i} \int_{U_{r_0}(\xi_i)} \frac{\lam_i^{4\al_i}\cF^i_{\xi} \rg^i_{\xi}}{ ( 1+\lam_i^2|y_{\xi_i}|^2)^{2\al_i}}\Big( 1+O\Big(\frac{\ln \lam_i}{\lam_i^2}\Big)\Big) \,\d v_g  + O(1).\label{lem:A.8-3}
		\end{align}
		Using the Taylor expansion and \eqref{varphixi}, we know that
		\begin{align}
			\int_{U_{r_0}(\xi_i)} \frac{\lam_i^{4\al_i}\cF^i_{\xi} \rg^i_{\xi}}{ ( 1+\lam_i^2|y|^2)^{2\al_i}} \,\d v_g=&\, (\cF^i_{\xi} \rg^i_{\xi})(\xi_i) \int_{B^{\xi_i}_{r_0}} \frac {\lam_i^{4\al_i}} { ( 1+\lam_i^2|y|^2)^{2\al_i}}\,\d y\nonumber \\&\,+\int_{B^{\xi_i}_{r_0}} \frac {\lam_i^{4\al_i}} { ( 1+\lam_i^2|y_{\xi_i}|^2)^{2\al_i}}\nabla_y ((\cF^i_{\xi}\rg^i_{\xi})(y_{\xi_i}^{-1}(y))e^{\hat\varphi_{\xi_i}})\Big|_{y=0}\cdot y \,\d y  \nonumber\\&\,+O\Big( \int_{B^{\xi_i}_{r_0}} \frac {\lam_i^{4\al_i}|y|^2} { ( 1+\lam_i^2|y|^2)^{2\al_i}}\,\d y\Big)\nonumber\\
			=& \,\frac{\kap( \xi_i)(\cF^i_{\xi} \rg^i_{\xi})(\xi_i)}{8(2\al_i-1)}\lam_i^{4\al_i-2}+B\Big(\frac{3}{2}, 2 \al_i-\frac{3}{2}\Big) \mathfrak{f}_2(\xi_i) \lam_i^{4 \al_i-3}\nonumber\\&\,+O\Big(\lam_i^{4 \al_i-2}\Big(\frac{\ln \lam_i}{\lam_i^2}+\frac{|\al_i-1|}{\lam_i^{3 / 2}}\Big)+1\Big),\label{eq:est_Fg_0rd}
		\end{align}
		where we used the estimates
		\begin{align*}
			\int_{B^{\xi_i}_{r_0}} \frac{\lam_i^{4\al_i} }{(1+\lam_{i}^2|y|^2)^{2\al_j}}  \, \d y	=&\,\frac{\kap(\xi_i)\lam_i^{4\al_i-2}}{8(2\al_i-1)}+O(1),\\
			\int_{B^{\xi_i}_{r_0}}\frac{ \lam_i^{4\al_i} y_2  }{(1+\lam_{i}^2|y_{\xi_i}|^2)^{2\al_i}}  \, \d y=&\,\lam_i^{4\al_i-3}\mathfrak{f}_2(\xi_i) \Big(B\Big(\frac{3}{2},2\al_i-\frac{3}{2}\Big)+O\Big(\frac{1}{\lam_i^{4\al_i-3}}\Big)\Big),
		\end{align*}
		and by using \eqref{Phibound},
		\[	\int_{B^{\xi_i}_{r_0}} \frac {\lam_i^{4\al_i}|y|^2} { ( 1+\lam_i^2|y|^2)^{2}}\,\d y+\int_{B^{\xi_i}_{r_0}} \frac {\lam_i^{4\al_i}|y|^2} { ( 1+\lam_i^2|y|^2)^{2}}(\Phi_i-1)\,\d y=O\Big(\lam_i^{4\al_i-2}\Big(\frac{\ln \lam_i}{\lam_i^2}+\frac{|\al_i-1|}{\lam_i^{3/2}}\Big)\Big).
		\]	
		By combining these estimates, we obtain 
		\begin{align}
			I_1= &\,\sum_{i} \frac{\kap( \xi_i)(\cF^i_{\xi} \rg^i_{\xi})(\xi_i)}{8(2\al_i-1)}\lam_i^{4\al_i-2} +\sum_{i} \lam_i^{4\al_i-3}B\Big(\frac{3}{2},2\al_i-\frac{3}{2}\Big)\mathfrak{f}_2(\xi_i) \nonumber\\&\, + O\Big(\sum_i \lam_i^{4\al_i-2}\Big(\frac{\ln \lam_i}{\lam_i^2}+|\al _i-1|
			\lam_i^{-\frac 3 2}\Big)\Big)+O(1).	\label{lem:A.8-I1}
		\end{align}
		
		\textbf{Step 2: Estimate of  $I_2$.} By Lemma \ref{lem:A.7} and the Cauchy-Schwarz inequality, we have
		\begin{equation}\label{lem:A.8-I2}
			I_2 = O\Big(\sum_{i} \lam_i^{4\al_i-2} \Big(|\al_i-1|^2 + \frac{1}{\lam_i^2} +\|w\|^2\Big)\Big)  + O(\|w\|).
		\end{equation}
		
		\textbf{Step 3: Estimate of  $I_3$.}
		By Lemma   \ref{lem:Ve^{u_0}_pointwise}, \eqref{def:Phi} and  the fact $|e^s - s - 1| = O(s^2)$ for $s \in \R$, we obtain
		\begin{align*}
			I_3 =&\,O\Big(\sum_{j} \lam_j^{4\al_j-2}\int_{U_{r_0}(\xi_j)}  e^{\delta_{j} }\Phi_j  |e^w - w- 1| \,\d v_g\Big) +O\Big(\int_{\Si\setminus\cup_{j} U_{r_0}(\xi_j)  }|e^w - w- 1| \,\d v_g\Big)\\=&\,O\Big(\sum_{j} \lam_j^{4\al_j-2}\int_{U_{r_0}(\xi_j)}  e^{\delta_{j} }\Phi_j  |e^w - w- 1| \,\d v_g\Big) +O(\|w\|^2),
		\end{align*}
		where we used the Sobolev embedding theorem in the last equality.
		In addition, Lemma \ref{lem:w_error_integral_estimates} yields that 
		\[
		\int_{U_{r_0}(\xi_j)}  e^{\delta_{j} }\Phi_j  |e^w - w- 1| \,\d v_g\leq C \Big(\int_{U_{r_0}(\xi_j)}e^{\delta_{j} }\Phi_j^2\, \d v_g\Big)^{1/2}\|w\|^2=O(\|w\|^2).
		\]
		By	combining the above estimates in this step,  we   obtain 
		\begin{equation}\label{lem:A.8-I3}
			I_3  = O\Big(\sum_{i} \lam_i^{4\al_i-2} \|w\|^2\Big).
		\end{equation}

		Observe that for $i=1,\ldots,p+q$, it holds $
		B(\frac{3}{2},2\al_i-\frac{3}{2})=B(\frac{3}{2},\frac{1}{2})+O(|\al_i-1|)=\frac{\pi}{2}+O(|\al_i-1|).
		$
		By combining the estimates \eqref{lem:A.8-I1}-\eqref{lem:A.8-I3}, we prove the validity of  \eqref{lem:A.8-1}. To establish \eqref{lem:A.8-2}, we need a refined estimate of \eqref{lem:A.8-3}.
		
		By Taylor expansion and   \eqref{varphixi}, we have
		\begin{align*}
			\int_{U_{r_0}(\xi_i)} \frac{\lam_i^{4\al_i}\cF^i_{\xi} \rg^i_{\xi}}{ ( 1+\lam_i^2|y_{\xi_i}|^2)^{2\al_i}} \,\d v_g
			=&\, (\cF^i_{\xi} \rg^i_{\xi})(\xi_i) \int_{B^{\xi_i}_{r_0}} \frac {\lam_i^{4\al_i}} { ( 1+\lam_i^2|y_{\xi_i}|^2)^{2\al_i}}\,\d y\\&\, +\int_{B^{\xi_i}_{r_0}} \frac {\lam_i^{4\al_i}} { ( 1+\lam_i^2|y|^2)^{2\al_i}} \nabla_y ((\cF^i_{\xi}\rg^i_{\xi})(y_{\xi_i}^{-1}(y))e^{\hat{\varphi}_{\xi_i}})\Big|_{y=0} \cdot y \,\d y \\
			&\, +\int_{B^{\xi_i}_{r_0}}\frac{\lam_i^{4\al_i}}{4}\frac {\Delta( (\cF^i_{\xi}\rg^i_{\xi})(y_{\xi_i}^{-1}(y)) e^{\hat{\varphi}_{\xi_i}})\Big|_{y=0}|y|^2} { ( 1+\lam_i^2|y|^2)^{2\al_i}} \,\d y\\&\,+O\Big( \int_{B^{\xi_i}_{r_0}} \frac {\lam_i^{4\al_i}|y|^3} { ( 1+\lam_i^2|y|^2)^{2\al_i}}\,\d y\Big)\\
			=&\,\frac{\kap( \xi_i)(\cF^i_{\xi}\rg^i_{\xi})(\xi_i) }{8(2\al_i-1)}\lam_i^{4\al_i-2} +   B\Big(\frac{3}{2}, 2\al_i- \frac{3}{2} \Big)   \mathfrak{f}_2(\xi_i)\lam_i^{4\al_i-3} \\&\, + \frac{\kap(\xi_i)}{16}\mathfrak{f}_3(\xi_i)\ln \lam_i +O(1),
		\end{align*}
		where we used the estimates
		\[	
		\int_{B^{\xi_i}_{r_0}} \frac { \lam_i^{4\al_i} |y|^2} { ( 1+\lam_i^2|y|^2)^{2\al_i}}\,\d y=
		\left\{\begin{aligned}
			&	\frac{\kap ( \xi_i)}{4} \ln  \lam_i+ O(1) && \text{ if }\, \al _i=1,\\
			&	O(1) && \text{ if }\,\al _i\neq 1,
		\end{aligned}\right.
		\]
		and $
		\int_{B^{\xi_i}_{r_0}} \frac {\lam_i^{4\al_i}|y|^3} { ( 1+\lam_i^2|y|^2)^{2\al_i}}\,\d y=O(1).
		$ Therefore, we get
		\begin{align}
			I_1=&   \,\sum_{i} \frac{\kap( \xi_i)(\cF^i_{\xi} \rg^i_{\xi})(\xi_i)}{8(2\al_i-1)}\lam_i^{4\al_i-2}\Big( 1+\sum_j \frac{\kap(\xi_j)}{2|\Si|_g}\frac{\ln \lam_j} {\lam_j^2}\Big)\nonumber\\&\, +\sum_{i} B\Big(\frac{3}{2},2\al_i-\frac{3}{2}\Big)\mathfrak{f}_2(\xi_i)\lam_i^{4\al_i-3}
			+\sum_{i} \frac{\kap( \xi_i)}{{16}}\mathfrak{f}_3(\xi_i)\ln \lam_i + O(1).\label{eq:B(3/2,1/2)}
		\end{align}
		By combining the above estimates for $I_1, I_2$ and $I_3$, we complete the proof of   \eqref{lem:A.8-2}.
	\end{proof}

	
	\section{Minimization problems}\label{app:B}
	In this appendix, we provide the proof of Proposition  \ref{prop:unique}, following the arguments presented in \cite{BC1,BC2}.
	
	Recall that $V_{\rho_*}^\eta(p,q,\var)$ is defined in \eqref{def:Vm} with $\rho=\rho_*$, and $\tilde{B}^\eta_{\var}$ is given by \eqref{def:ballB_var^eta}. In the following, we   assume that  $p,q\in \N_0$ with  $m:=2p+q\in\N$, $\rho_*:=4\pi m$, and $\eta>0$ is a fixed constant,  while $\var>0$ is  sufficiently small.
	%

		\begin{lem}\label{lem:pre_unique}
			Let $\{\var_k\}$ be a sequence satisfying  $\var_k>0$ and $\lim_{k\to+\infty}\var_k=0$. Assume that  $(\al^k, \xi^k, \lam^k), (\tilde{\al}^k, \tilde{\xi}^k, \tilde{\lam}^k) \in \tilde{B}^\eta_{\var_k}$, such that	
			\begin{equation}
				\label{eq:asss_0}
				\lim_{k\to+\infty}	 \|\psi(\al^k,\xi^k, \lam^k) -\psi(\tilde{\al}^k,\tilde{\xi
				}^k, \tilde{\lam}^k)\|=  0 .
			\end{equation}
			Then, up to a permutation of  $(\tilde{\al}^k,\tilde{\xi
			}^k, \tilde{\lam} ^k)$, we have 
			\begin{gather}
				\lim_{k\to+\infty}	\Big| \frac{\lam^k_i}{\tilde{\lam}   ^k_i}-1\Big|=0, \quad\forall i=1,\ldots,p+q,\label{eq:delta_eqs_1}\\
				\lim_{k\to+\infty}	\lam^k_i\tilde{\lam}^k_i d_g^2(\xi^k_i, \tilde{\xi}^k_i)  =0,\quad\forall i=1,\ldots,p+q, \label{eq:delta_xi_eqs_0}\\
				\lim_{k\to+\infty}	|\al^k_i-\tilde{\al}^k_i|^2\ln \lam_i^k=0,\quad\forall i=1,\ldots,p+q.\label{eq:alpha_eqs_0}
			\end{gather}
		\end{lem}
		\begin{proof}
			From Lemma \ref{lem:projbubble_integral_interaction}, we have
			\begin{equation}\label{eq:1} 
				\langle P\delta_{\lam_i,\xi_i}, P\delta_{\lam_i,\xi_i}\rangle 	
				=  4\kap ( \xi_i)\ln \lam_i+\kap^2(\xi_i)R^g(\xi_i)-2\kap ( \xi_i) +  o(1),
			\end{equation}
			as $\lam_i\to  +\infty$ and for any $\xi_i,\xi_j$ satisfy $d_g(\xi_i,\xi_j)\geq 2\eta$, we have 
			\begin{equation}
				\label{eq:2} 
				\langle P\delta_{\lam_i,\xi_i}, P\delta_{\lam_j,\xi_j}\rangle 	=\kap ( \xi_i)\kap ( \xi_j) G^g(\xi_j,\xi_i)+o(1),
			\end{equation}
			as $\lam_i,\lam_j\to  +\infty$. 
			Since by \eqref{Green} and \eqref{Greenrepresentation},
			\begin{equation}\label{Green2}
				P\delta_{\lam_i,\xi_i}(\xi_j)=\int_{\Si}(-\Delta_g P\delta_{\lam_i,\xi_i})G^g(x,\xi_j)\,\d v_g=\int_{\Si}\chi_{\xi_i}e^{\delta_{\lam_i,\xi_i}-\varphi_{\xi_i}}G^g(x,\xi_j)\,\d v_g
			\end{equation}
			for all $i,j=1,\ldots,p+q$,	by Lemma \ref{lem:projbubble_pointwise},  \eqref{Greenrepresentation} and \eqref{Green2}, we have for any $d_g(\xi_i,\xi_j)\to  0$,
			\begin{align}
				\langle P\delta_{\lam_i,\xi_i}, P\delta_{\lam_j,\xi_j}\rangle =&\,	\int_{\Si }\chi_{\xi_i}  P\delta_{\lam_j,\xi_j} e^{\delta_{\lam_i,\xi_i}-\varphi_{\xi_i}}  \,\d v_g \nonumber\\
				=&\,\int_{\Si } \chi_{\xi_i} e^{-\varphi_{\xi_i}} \frac{8\lam_i^2}{ (1+\lam_i^2|y_{\xi_i}|^2)^2 }  \Big(2\chi_{\xi_j}\ln  \Big(\frac { \lam_j^2}{ 1+\lam_j^2|y_{\xi_j}|^2 }\Big) +\kap ( \xi_j) H^g(x,\xi_j) +o(1) \Big) \,\d v_g\nonumber\\
				=&\,\int_{\Si } \chi_{\xi_i} e^{-\varphi_{\xi_i}} \frac{8\lam_i^2}{ (1+\lam_i^2|y_{\xi_i}|^2)^2 }  \Big(2\chi_{\xi_j}\ln  \Big(\frac { \lam_j^2|y_{\xi_j}|^2}{ 1+\lam_j^2|y_{\xi_j}|^2 }\Big) +\kap ( \xi_j) G^g(x,\xi_j) +o(1) \Big) \,\d v_g\nonumber\\
				=&\, 
				-2\int_{B^{\xi_i}_{2r_0}}\chi^2\Big(\frac{|y|}{r_0}\Big) \frac{8\lam_i^2\ln\Big( 1+\frac{1}{\lam_j^2|y|^2}\Big)}{(1+\lam_i^2|y|^2)^2} \, \d y + \kap ( \xi_j)P\delta_{\lam_i,\xi_i}(\xi_j)+o(1)\Big)\nonumber\\
				=&\,-16 \int_{\R_{\xi_i}} \frac{\ln \Big(1+ \frac{\lam^2_i}{\lam_j^2|y|^2}\Big)}{(1+|y|^2)^2} \, \d y  -2\kap ( \xi_j)\ln  \Big( \frac{1}{\lam_i^{2}}+ |y_{\xi_i}(\xi_j)|^2\Big)\nonumber\\&\,+
				\kap ( \xi_i)\kap ( \xi_j) H^g(\xi_i,\xi_j)+ o(1),\label{eq:4}
			\end{align}
			as $\lam_i,\lam_j\to  +\infty$,
			where $\R_{\xi_i}= \R^2$ if $\xi_i\in\intS$, and $\R_{\xi_i}=\R^2_+$ if $\xi_i\in \pa\Si $. 
			
			\textbf{Case I.}  $p+q=1$. It is evident that $d_g(\xi^k_1, \tilde{\xi}^k_1)\to  0$ as $k\to+\infty$; otherwise, by \eqref{eq:1} and \eqref{eq:2}, we have 
			\[\|\psi(\al^k,\xi^k, \lam^k) -\psi(\tilde{\al}^k,\tilde{\xi
			}^k,\tilde{\lam}^k)\|= O(\ln \lam^k_1+ \ln\tilde{\lam}^k_1)\to  +\infty, \] which contradicts to the assumption \eqref{eq:asss_0}. 
			
			Without loss of generality, we assume that  $\frac{\lam    ^k_1}{\tilde{\lam}^k_1}:=c_k\leq  1$.  One obtains from  \eqref{eq:1} and \eqref{eq:4} that 
			\begin{align}
				&\frac 1{\kap ( \xi^k_1)}	\| \al^k_1 P\delta_{\lam^k_1,\xi^k_1}-\tilde{\al}^k_1 P\delta_{\tilde{\lam}   ^k_1,\tilde{\xi}^k_1}  \|^2\notag\\=&\,  -2(\al^k_1)^2  -2(\tilde{\al}^k_1)^2 +\kap(\xi_1^k)((\al_1^k)^2R^g(\xi_1^k) +(\tilde{\al}^k_1)^2 R^g(\tilde{\xi}^k_1)-2 \al_1^k \tilde{\al}^k_1 H^g(\xi^k_1,\tilde{\xi}^k_1)) \notag\\
				&\,+  \frac {4\tilde{\al}^k_1 \al^k_1} {\pi}\int_{\R^2} \frac{ \ln  ( 1+\frac { c_k^2 } {|y|^2})}{(1+|y|^2)^2} \, \d y+4\tilde{\al}^k_1 \al^k_1\ln  (1+({\lam }^k_1)^2 |y_{{\xi}^k_1}(\tilde{\xi}^k_1)|^2)+ 4(\tilde{\al}^k_1 -\al^k_1)^2 \ln  \lam^k_1\notag\\
				&\, +4(\tilde{\al}^k_1)^2 \ln \frac 1{ c_k}+o(1), \label{eq:3}
			\end{align} 
			as $ k\to  +\infty$.
			Using the assumption \eqref{eq:asss_0}, we obtain from \eqref{eq:3} that 
			\begin{gather*}
				|\tilde{\al}^k_1 -\al^k_1|^2= O\Big(\frac 1 {\ln  \lam^k_1}\Big)=o(1), \\
				\frac{\tilde{\lam}  ^k_1}{\lam^k_1}+ \frac{\lam^k_1}{\tilde{\lam}  ^k_1}+ {\tilde{\lam}  ^k_1}{\lam^k_1}d^2_g( \tilde{\xi}^k_1,\xi^k_1)  =O(1),
			\end{gather*}
			as $k\to +\infty$.
			We observe that 
			\begin{align*}
				\frac {4\tilde{\al}^k_1 \al^k_1} {\pi}\int_{\R^2} \frac{ \ln  ( 1+\frac { c_k^2 } {|y|^2})}{(1+|y|^2)^2} \, \d y =& \,\frac {4\tilde{\al}^k_1 \al^k_1} {\pi}\int_{\R^2} \Big(\frac{ \ln  ( 1+\frac { 1 } {|y|^2})}{(1+|y|^2)^2}+\frac{ \ln  ( 1+\frac { c_k-1 } {1+|y|^2})}{(1+|y|^2)^2} \Big)\, \d y \\
				=&\, 4\tilde{\al}^k_1 \al^k_1+ \frac {4\tilde{\al}^k_1 \al^k_1} {\pi}\int_{\R^2} \frac{ \ln  ( 1+\frac { c_k-1 } {1+|y|^2})}{(1+|y|^2)^2} \, \d y\\
				=&\, 4\tilde{\al}^k_1 \al^k_1+  4\tilde{\al}^k_1 \al^k_1\ln  c_k + \frac {4\tilde{\al}^k_1 \al^k_1} {\pi}\int_{\R^2} \frac{ \ln  ( \frac{1}{c_k}(1-\frac 1 {1+|y|^2})+\frac { 1 } {1+|y|^2})}{(1+|y|^2)^2} \, \d y.
			\end{align*}
			Passing  the limit $k\to +\infty$ in \eqref{eq:3}, and assuming, up to a subsequence, that $c:=\lim_{k\to +\infty} c_k$,  we obtain
			\begin{align}
				0=& \,\liminf_{k\to +\infty}\frac {4\tilde{\al}^k_1 \al^k_1} {\pi}\int_{\R^2} \frac{ \ln  ( \frac 1 {c}(1-\frac 1 {1+|y|^2})+\frac { 1 } {1+|y|^2})}{(1+|y|^2)^2} \, \d y+\liminf_{k\to +\infty} 4\tilde{\al}^k_1 \al^k_1\ln  (1+({\lam }^k_1)^2 |y_{{\xi}^k_1}(\tilde{\xi}^k_1)|^2)\notag\\
				&\,+\liminf_{k\to +\infty} 4(\tilde{\al}^k_1 -\al^k_1)^2 \ln  \lam^k_1.\label{eq:lim_+}
			\end{align}  
			Since each term on the right-hand side of \eqref{eq:lim_+} is non-negative,the following estimates hold as $k\to+\infty$:
			\[
			\frac{\tilde{\lam}^k_1}{\lam^k_1}-1=o(1), \quad	|\tilde{\al}^k_1 -\al^k_1|^2\ln  \lam^k_1= o(1),\quad
			\text{and} \quad {\tilde{\lam} ^k_1}{\lam^k_1}  d^2_g( \tilde{\xi}^k_1,\xi^k_i)=o(1). \] 
			This establishes Lemma \ref{lem:pre_unique} in the case 
			$p+q=1$.
			
			\textbf{Case II.}	$p+q\geq2$. Up to a subsequence and a permutation, we may assume that
			$d_g(\xi^k_l,\tilde{\xi}^k_l)=\min_{1\leq i\leq p+q} d_g(\xi^k_l,\tilde{\xi}^k_i)$ for all $l=1,\ldots,p+q$.   For any $i\neq j$,
			\[	
			\max\{\langle P\delta_{\lam^k_i,\xi^k_i}, P\delta_{\lam^k_j,\xi^k_j} \rangle ,\langle P\delta_{\tilde{\lam} ^k_i,\tilde{\xi}^k_i}, P\delta_{\tilde{\lam}^k_j,\tilde{\xi}^k_j}\rangle, \langle P\delta_{\lam^k_i,\xi^k_i}, P\delta_{\tilde{\lam} ^k_j,\tilde{\xi}^k_j}\rangle \}=O(1),
			\]
			in view of $\max\{d_g(\xi^k_i,\xi^k_j),d_g(\tilde{\xi}^k_i,\tilde{\xi}^k_j)\}\geq 2 \eta $. Then, as $k\to +\infty$,
			\[	
			o(1)=	\|\psi(\al^k,\xi^k, \lam^k) -\psi(\tilde{\al}^k,\tilde{\xi
			}^k,\tilde{\lam}^k)\|^2
			= \sum_{i}  \|{\al}^k_iP\delta_{\lam^k_i, {\xi}^k_i}- \tilde{\al}^k_iP\delta_{\tilde{\lam} ^k_i, \tilde{\xi}^k_i} \|^2+O(1).
			\]
			Applying the same reasoning as in the case $p+q=1$, we conclude that 
			\begin{gather*}
				|\tilde{\al}^k_i -\al^k_i|^2\ln  \lam^k_i= o( 1), \quad i=1,\ldots,p+q, \\
				\frac{\tilde{\lam}  ^k_i}{\lam^k_i}+ \frac{\lam^k_i}{\tilde{\lam}  ^k_i}+ {\tilde{\lam}  ^k_i}{\lam^k_i}d^2_g( \tilde{\xi}^k_i,\xi^k_i)  =O(1), \quad i=1,\ldots,p+q.
			\end{gather*}
			Moreover, by  \eqref{eq:2}, for $i\neq j$, we have
			\begin{align*}
				& \, \al^k_i \al^k_j \langle P\delta_{\lam^k_i,\xi^k_i}, P\delta_{\lam^k_j,\xi^k_j}\rangle +\tilde{\al}^k_i \tilde{\al}^k_j \langle P\delta_{\tilde{\lam}  ^k_i,\tilde{\xi}^k_i}, P\delta_{\tilde{\lam}   ^k_j,\tilde{\xi}^k_j}\rangle   \\
				&\,  - \tilde{\al}^k_i \al^k_j \langle P\delta_{\tilde{\lam}  ^k_i,\tilde{\xi}^k_i}, P\delta_{\lam^k_j,\xi^k_j}\rangle- \tilde{\al}^k_j \al^k_i \langle P\delta_{\lam^k_i,\xi^k_i},  P\delta_{\tilde{\lam}   ^k_j,\tilde{\xi}^k_j}\rangle  \\
				=&\, \al^k_i \al^k_j \kap ( \xi_i^k)\kap ( \xi^k_j)G^g(\xi^k_i,\xi^k_j)+ \tilde{\al}^k_i \tilde{\al}^k_j \kap ( \tilde{\xi}_i^k)\kap ( \tilde{\xi}^k_j) G^g(\tilde{\xi}^k_i,\tilde{\xi}^k_j)-  \tilde{\al}^k_i \al^k_j\kap ( \tilde{\xi}_i^k)\kap ( \xi^k_j)G^g(\tilde{\xi}^k_i,{\xi}^k_j)\\
				&\,-\tilde{\al}^k_j \al^k_i  \kap ( \xi_i^k)\kap ( \tilde{\xi}^k_j) G^g({\xi}^k_i,\tilde{\xi}^k_j) +o(1)\\
				\,\to & 0,
			\end{align*}
			which implies that $\sum_{i}  \|  {\al}^k_iP\delta_{\lam^k_i, {\xi}^k_i}-\tilde{\al}^k_iP\delta_{\tilde{\lam} ^k_i, \tilde{\xi}^k_i}\|^2=o(1)$. Applying the conclusion for $p+q=1$, we deduce Lemma \ref{lem:pre_unique} holds for $p+q\geq 2$. 
		\end{proof}
		\begin{lem}\label{lem:small_2vare}
			Let $\eta'>\eta$ be a fixed constant. Then there exists $\var_0 > 0$ such that, for any   $u\in  V^{\eta'}_{\rho_*}(p,q,\var)$ with $\var\in(0,\var_0]$, and  for any $R\geq \var$,   the infimum
			\begin{equation}
				\label{eq:problem_inf} 	\inf_{(\al, \xi, \lam ) \in \tilde{B}^{\eta} _{4R}} \|u - \psi(\al, \xi, \lam )\|
			\end{equation}
			is achieved in $\tilde{B}^{\eta} _{2R}$ and is not achieved in $\tilde{B}^{\eta} _{4R} \setminus \tilde{B}^{\eta} _{2R}$.
		\end{lem}
		\begin{proof}
			We proceed by contradiction. Suppose  that  Lemma \ref{lem:small_2vare} does not hold.
			Then there exists sequences $\var_k\to 0^+$,  $R_k\geq \var_k$, and $u^*_k\in V^{\eta'}_{\rho_*}(p,q,\var_k) $ such that the infimum \eqref{eq:problem_inf} is achieved  at some point $(\al_k,\xi_k,\lam_k)\in\tilde{B}^{\eta} _{4R_k}\setminus\tilde{B}^{\eta} _{2R_k}$.  To simplify notation, we omit the subscript $k$ and continue to write
			$\var, \xi, \lam , \al, u^*, R$ in place of  $\var_k, \xi_k, \lam _k, \al_k, u^*_k, R_k$, respectively. 
			
			Let $e:=(1,\ldots, 1)\in\R^{p+q}$.  Then, there exists $ (\tilde{\xi} ,\tilde{\lam})$ such that $(e,\tilde{\xi} ,\tilde{\lam} )\in \tilde{B}^{\eta} _{\var}$ and  $\|u^*-\psi(e,\tilde{\xi} ,\tilde{\lam})\|< \var$. Since $(\al,\xi, \lam )$ is the minimizer to the problem \eqref{eq:problem_inf} with $u=u^*$, we have 
			$\|u^*- \psi(\al,\xi, \lam  )\|\leq \|u^*-\psi(e,\tilde{\xi} ,\tilde{\lam})\|<\var$. It follows that 
			\[ \|\psi(\al,\xi, \lam )-\psi(e,\tilde{\xi} ,\tilde{\lam} )\|=O(\var)=o(1),\]
			as $\var\to  0$. From Lemma \ref{lem:pre_unique}, we know that, for each $i=1,\ldots,p+q$,
			\[	 \frac{\lam_i}{\tilde{\lam} _i}-1=o(1), \quad \lam_i\tilde{\lam} _id_g^2(\xi_i,\tilde{\xi}_i)=o(1), \text{ and }
			\quad  |\al_i -1|=o(1).\]
			Hence, by the definition of $\tilde{B}^{\eta} _{\var}$, we conclude that for sufficiently small $\var>0$, it  holds 
			$(\al, \xi,\lam )\in  \tilde{B}^{\eta} _{ 2\var}$, which is a contradiction. 	
		\end{proof}
		Now we are ready to prove Proposition \ref{prop:unique}.
		\begin{proof}[Proof of Proposition \ref{prop:unique}]
			We proceed by contradiction. Suppose Proposition \ref{prop:unique} fails. Then, by Lemma \ref{lem:pre_unique},  there exist sequences $\var_k\to 0^+$,  $u^*_k \in V_{\rho_*}^\eta(p,q,\var_k)$, and $ (\al^k, \xi^k, \lam^k)$, $(\tilde{\al}^k, \tilde{\xi}^k,\tilde{\lam}^k)\in \tilde{B}^{\eta} _{2 \var_k}$ with $(\al^k, \xi^k, \lam^k)\neq (\tilde{\al}^k, \tilde{\xi}^k,\tilde{\lam}^k) $, however both $\psi(\al^k, \xi^k, \lam^k)$ and $\psi(\tilde{\al}^k, \tilde{\xi}^k,\tilde{\lam}^k)$ are minimizers of  the  problem \eqref{eq:min}. 
			To simplify notation, we omit the index  $k$ and continue to write 
			$\var, \al,\xi,\lam , \tilde{\al}, \tilde{\xi},\tilde{\lam}, u^*$ in place of $\var_k, \al^k,\xi^k,\lam^k, \tilde{\al}^k, \tilde{\xi}^k,\tilde{\lam}  ^k, u^*_k$, respectively. 
			Denote that  $w = u^* - \psi (\al, \xi, \lam 
			)$, $\tilde{w} = u^* - \psi(\tilde{\al},\tilde{\xi} ,\tilde{\lam})$. 
			Since $ (\al^k, \xi^k, \lam^k)$ and $(\tilde{\al}^k, \tilde{\xi}^k,\tilde{\lam}^k)$ solve the minimization problem \eqref{eq:min}, we have 
			\begin{align}
				0 =& \langle w, P\delta_{\lam_i,\xi_i}\rangle  = \langle \tilde{w}, P\delta_{\tilde{\lam} _i,\tilde{\xi} _i} \rangle , \label{eq:or_0}\\
				0 =& \Big\langle w, \frac{\pa P\delta_{\lam_i,\xi_i}}{\pa(\xi_i)_j}\Big\rangle= \Big\langle \tilde{w}, \frac{\pa P\delta_{\tilde{\lam}   _i,\tilde{\xi} _i} }{\pa(\tilde{\xi} _i)_l}\Big \rangle,\label{eq:or_1}\\
				0 =& \Big\langle w, \frac{\pa P\delta_{\lam_i,\xi_i}}{\pa\lam_i}\Big\rangle=  \Big\langle \tilde{w}, \frac{\pa P\delta_{\tilde{\lam}  _i,\tilde{\xi} _i}}{\pa \tilde{\lam}   _i}   \Big\rangle.\label{eq:or_2}
			\end{align}
			for all $i=1,\ldots,p+q$, $j\in\{1,\ii(\xi_i)\}$ and $l\in\{1,\ii(\tilde{\xi}_i)\}$.
			Denote $\varphi_i=\varphi_{\xi_i}$,  $\tilde \varphi_i=\tilde\varphi_{\xi_i}$,  $\chi_i=\chi_{\xi_i}$,  $\tilde \chi_i=\tilde\chi_{\xi_i}$, $\delta_i=\delta_{\lam_i,\xi_i}$, $P\delta_i= P\delta_{\lam_i,\xi_i}$ $\tilde{\delta}_i=\delta_{\tilde{\lam} _i,\tilde{\xi} _i}$ and $P\tilde{\delta}_i=P\delta_{\tilde{\lam}   _i,\tilde{\xi} _i}$. 
			The orthogonal property \eqref{eq:or_0} implies that 
			\begin{align*}
				\sum_{j} \langle \al_j P\delta_j-\tilde{\al}_j P\tilde{\delta}_j, P\delta_i\rangle =&\,
				\langle \psi(\al,\xi,\lam)- \psi(\tilde{\al},\tilde{\xi} ,\tilde{\lam}   ), P\delta_i\rangle \\
				=&\, \langle \tilde{w}, P\delta_i\rangle - \langle w, P\delta_i \rangle 
				= \langle \tilde{w}, P\delta_i- P\tilde{\delta}_i\rangle. 
			\end{align*}
			As in the proof of Lemma \ref{lem:small_2vare},  
			there exist $(e, \xi,\lam ^*) \in \tilde{B}^{\eta} _{\var}$ such that 
			$\|u^* - \psi(e, \xi, \lam ^*)\|<\var$. 
			Lemma \ref{lem:pre_unique}  implies that for each $i=1,\ldots,p+q$,
			\begin{equation}\label{eq:est_alpha_1}
				\frac{\lam ^*_i}{\lam_i}-1=o(1)\quad\text{ and }\quad
				|\al_i-1|^2 = o\Big( \frac1 {\ln  \lam ^*_i}\Big)=o\Big(\frac{1}{|\ln  \var|}\Big). 
			\end{equation}
			Moreover,  as $\var\to  0$, we have
			\[	
			\|\psi(\al,\xi,\lam)- \psi(\tilde{\al},\tilde{\xi} ,\tilde{\lam}   )\|=o(1).
			\] Let 
			$a_i= \lam_i d_g(\xi_i,\tilde{\xi} _i)$, $b_i = \frac{\lam   _i}{\tilde{\lam}_i}-1$, and $\iota_i= \al_i-\tilde{\al}_i$. Then, by Lemma \ref{lem:pre_unique}, we immediately obtain
			\[ |a_i|=o(1), \quad |b_i|=o(1)\quad  \text{ and } \quad |\iota_i|=o(1). \]
			On the one hand, we have
			\[	
			\langle \al_j P\delta_j-\tilde{\al}_j P\tilde{\delta}_j, P\delta_i\rangle =(\al_j-\tilde{\al}_j) \langle P\delta_j, P\delta_{i}\rangle + \tilde{\al}_j \langle P\delta_j -P\tilde{\delta}_j, P\delta_i\rangle.
			\]
			As a result, it follows that
			\begin{equation}
				\label{eq:app_B1} 	\langle \tilde{w}, P\delta_i-P\tilde{\delta}_i\rangle =
				\sum_{j}(	(\al_j-\tilde{\al}_j) \langle P\delta_j, P\delta_i\rangle + \tilde{\al}_j \langle P\delta_j -P\tilde{\delta}_j, P\delta_i\rangle).
			\end{equation}

			For any $j=1,\ldots, p+q$ and $\var>0$ sufficiently small, we choose the isothermal coordinates $( U(\tilde{\xi} _j), y_{\tilde\xi_j})$ around $\tilde\xi_j$ such that 
			\[ y_{\tilde{\xi}_j }(x)= y_{\xi_j}(x)-y_{\xi_i}(\tilde{\xi} _j),\quad U(\tilde{\xi}_j )= y_{\xi_j}^{-1}( B^{\xi_j}\cap \B_{2r_0}(y_{\xi_j}(\tilde{\xi} _j)))\]
			with $\tilde\varphi_{j}(y)= \varphi_{j}(y)- \varphi_{j}( y_{\xi_j}(\tilde\xi_j))- \langle \nabla\varphi_{j}( y_{\xi_j}(\tilde\xi_j)), y\rangle_{\R^2} $. 
			Then, by the mean value theorem, for any $x\in  U_{2r_0}(\xi_j)\cap U_{2r_0}(\tilde{\xi} _j)$, we obtain
			\begin{equation}\label{eq:b02}
				|\chi_{j}  e^{-\varphi_{j}} - \tilde\chi_{j} e^{-\tilde\varphi_{j}}| \leq C \sup_{x\in \Si } |\nabla (\chi_{j}  e^{-\varphi_{j}})| d_g(\xi_j, \tilde\xi_j)=O(d_g(\xi_j,\tilde\xi_j)),
			\end{equation}
			where $C>0$ is a constant.
			The estimate \eqref{eq:b02}  implies that
			\begin{align}
				\chi_{j}    e^{ \delta_j-\varphi_{j}}- \tilde\chi_{j}  e^{\tilde{\delta}_j-\tilde\varphi_{j}}\notag=&\, \chi_{j}  e^{-\varphi_{j}} ( e^{ \delta_j}-e^{\tilde{\delta}_j})+  (\chi_{j}  e^{-\varphi_{j}} - \tilde\chi_{j} e^{-\tilde\varphi_{j}} ) e^{\tilde{\delta}_j}\\
				\,
				=& \,-2b_j(1+o(1)) \chi_{j}   e^{\delta_j-\varphi_{j}}+ 2(1+o(1))\Big(  \lam_j^2|y_{\xi_j}(\tilde{\xi} _j)|\notag\\
				&\, -2\lam_j^2\langle y_{\xi_j}(x),y_{\tilde{\xi} _j}(\xi_j) \rangle_{\R^2}\Big) \frac{\chi_{j}   e^{ \delta_j-\varphi_{j}}} {1+\lam_j^2 |y_{\xi_j}|^2 }+
				O\Big( \frac{|a_j|}{\lam_j} \mathbbm{1} _{U_{4r_0}(\tilde{\xi}_j)}  e^{ \delta_j}+\frac{1}{\lam_2}\Big).\label{eq:b03}
			\end{align}
			Applying Lemmas \ref{lem:projbubble_xi_pointwise}  and  \ref{lem:A.5}, we deduce that 
			\begin{align}
				|\langle \tilde{w}, P\delta_i-P\tilde{\delta}_i\rangle |
				=&\, \Big| \int_{\Si } -\Delta_g(P\delta_i-P\tilde{\delta}_i) \tilde{w}\,\d v_g \Big|\leq  \int_{\Si }|\chi_{j}    e^{ \delta_j-\varphi_{j}}- \tilde\chi_{j} e^{\tilde{\delta}_j-\tilde\varphi_{j}}| |\tilde{w}| \,\d v_g \nonumber\\
				=&\, O((|a_i|+|b_i|)\|\tilde{w}\|)=o( |a_i|+|b_i|),\label{eq:app_B2}
			\end{align}
			since $ \|\tilde w\|=o(1)$. 
			
			From estimate \eqref{eq:b04} with Lemma \ref{lem:projbubble_integral_interaction}, we obtain
			\begin{align}
				&	\langle P\delta_j-P\tilde{\delta}_j,  P\delta_i\rangle = \int_{\Si } (-\Delta_g( P\delta_j- P\tilde{\delta}_j)) P\delta_i \,\d v_g \notag\\
				=& \,O\Big(  \int_{\Si } \Big( (|a_j|+|b_j|)\chi_{j}  e^{ \delta_j-\varphi_{j}} + \frac{|a_j|}{\lam_j}\mathbbm{1}_{U_{4r_0}(\xi_j)} e^{ \delta_j} \Big)\Big(\chi_{i}  \Big(\delta_i+\ln \Big(\frac{\lam_i^2}{8}\Big)\Big) + O(1)\Big) \, \d v_g \Big)\notag\\
				=&\,\left\{\begin{aligned}	
					&O( ( |a_i|+|b_i|)\ln \lam_i)&& \text{ if }\, i=j, \\
					&	O\Big(\frac{|a_j|}{\lam_j} \Big) && \text{ if }\, i\neq j.
				\end{aligned}\right.\label{eq:b12}
			\end{align}	
			Using \eqref{eq:app_B1}  together with  \eqref{eq:1}, \eqref{eq:2}, \eqref{eq:app_B2} and \eqref{eq:b12}, we deduce 
			as $\var\to  0$ that 
			\begin{equation}\label{eq:b0_final}
				\iota_i= o\Big(  \frac{|a_i|+|b_i|}{ \ln \lam_i}\Big)
				+ O\Big( |a_i|+|b_i|+ \sum_{j\neq i} o( |a_j|+|\iota_j|)\Big),
			\end{equation}
			in view of $\|\tilde{w}\|=o(1)$. 
			Similarly, we have
			\begin{equation}
				\label{eq:de_or_1} 
				\Big\langle	 \tilde{w}, \frac{\pa P\delta_i}{\pa \zeta_i} -\frac{\pa P\tilde{\delta}_i}{\pa\tilde{\zeta} _i} \Big\rangle =\sum_{j}\Big( (\al_j-\tilde{\al}_j) \Big\langle P\delta_j, \frac{\pa P\delta_i}{\pa\zeta_i}  \Big\rangle+ \tilde{\al}_j 		\Big\langle  P\delta_j- P\tilde{\delta}_j , \frac{\pa P\delta_i}{\pa \zeta_i} \Big\rangle\Big).
			\end{equation}
			Here and in the following, we use $\frac{\pa }{\pa \zeta_i}$ to denote  $\frac{\pa }{\pa (\xi_i)_j}$ for some fixed $j\in\{1,\ii(\xi_i)\}$;   similarly,  $\frac{\pa }{\pa \tilde{\zeta}_i}$ denotes  $\frac{\pa }{\pa \tilde{(\xi_i)}_l}$ for some fixed $l\in\{1,\ii(\tilde\xi_i)\}$.
			By straightforward calculation, we find
			\begin{align}
				&	\Big|\frac{\pa }{\pa \zeta_j}(\chi_{j}    e^{ \delta_j-\varphi_{j}})- \frac{\pa }{\pa \tilde{\zeta} _j}(\tilde\chi_{j}  e^{\tilde{\delta}_j-\tilde\varphi_{j}} )\Big|\notag \\
				=&\, \Big|\chi_{j}  e^{-\varphi_{j}} \Big( e^{ \delta_j}\frac{\pa \delta_j}{\pa\zeta_j}   - e^{\tilde{\delta}_j}\frac{\pa \tilde{\delta}_j}{\pa\tilde{\zeta} _j} \Big) \Big| + O\Big(\Big(\frac{|a_j|}{\lam_j} + |a_j|+|b_j|\Big)\mathbbm{1}_{U_{4r_0}(\xi_j)}  e^{ \delta_j}\Big)\notag\\
				=& \,  O\Big(\chi_{j}   e^{ \delta_j-\varphi_{j}}\Big( ( |a_j|+|b_j|) \Big(\Big| \frac{\pa \delta_j}{\pa\zeta_j } \Big|+1\Big)+ \lam_j|b_j|\Big) + \frac{|a_j|}{\lam_j} \mathbbm{1}_{U_{4r_0}(\xi_j)} e^{ \delta_j} \Big).
				\label{eq:b04}
			\end{align}
			Applying Lemmas \ref{lem:projbubble_xi_pointwise}  and \ref{lem:A.5}, we obtain
			\begin{align*}
				\Big|\Big\langle \tilde{w},\frac{\pa P\delta_i}{\pa\zeta_i}- \frac{\pa P\tilde{\delta}_i}{\pa\tilde{\zeta} _i} \Big\rangle	\Big| 
				=& \,\Big| \int_{\Si }\Big(\frac{\pa }{\pa\zeta_j} (\chi_{j}   e^{ \delta_j-\varphi_{j}})- \frac{\pa }{\pa\tilde{\zeta} _j}(\tilde\chi_{j}  e^{\tilde{\delta}_j-\tilde\varphi_{j}})\Big) \tilde{w} \,\d v_g  \Big|\\
				=&\,O(  (|a_i|+|b_i|) \|\tilde{w}\|)=
				o( |a_i|+|b_i|  ). 
			\end{align*}
			By Lemma \ref{lem:projbubble_xi_pointwise}  and \eqref{eq:b03}, we have  
			\begin{align*}
				&	\Big\langle P\delta_j-P\tilde{\delta}_j, \frac{\pa P\delta_i}{\pa\zeta_i} \Big\rangle= \int_{\Si }(-\Delta_g)(P\delta_j-P\tilde{\delta}_j ) \frac{\pa P\delta_{i}}{\pa\zeta_i}   \,\d v_g \\
				=& \, \int_{\Si } (\chi_{j}    e^{ \delta_j-\varphi_{j}}- \tilde\chi_{j} e^{\tilde{\delta}_j-\tilde\varphi_{j}})\\
				&\, \times  \Big( \chi_{i} \frac{\pa \delta_i}{\pa\zeta_i}   + \kap ( \xi_i)\frac{\pa  H^g(\cdot,\xi_i)}{\pa\zeta_i}  -4\ln  |y_{\xi_i}| \frac{\pa \chi_{i}}{\pa\zeta_i}  + O\Big(\frac{\ln \lam_i}{\lam_i^2}\Big) \Big) \,\d v_g \\
				=&\,\left\{\begin{aligned}			
					&	-64\lam_i \int_{\lam_i B^{\xi_i}_{2r_0}} | \lam_iy_{\xi_i}(\tilde{\xi} _i)_l| \frac{y_l^2}{(1+|y|^2)^4} \,\d y + o( \lam_i(|a_i|+|b_i|))
					&& \text{ if }\, i=j ,\\
					&	O(|a_j|+|b_j|) && \text{ if }\, i\neq j.
				\end{aligned}
				\right.
			\end{align*}
			Moreover, Lemma \ref{lem:projbubble_xi_pointwise}   gives that 
			\begin{equation}\label{eq:13}
				\Big\langle P\delta_j, \frac{\pa  P\delta_i}{\pa\zeta_i}\Big\rangle=O(1).
			\end{equation}
			Collecting the estimates \eqref{eq:de_or_1}-\eqref{eq:13} and using  $\|\tilde{w}\|=o(1)$, we arrive at
			\begin{equation}\label{eq:b1_final}
				\lam_i  |a_i|=o(\lam_i(|a_i|+|b_i|) )+ O\Big(\sum_{j} (|a_j|+|b_j|+|\iota_j|)\Big).
			\end{equation}
			By the orthogonality condition \eqref{eq:or_2},  we find
			\begin{equation}
				\label{eq:de_or_2} 
				\Big\langle \tilde{w}, \frac{\pa P\delta_i}{\pa \lam_i} -\frac{\pa P\tilde{\delta}_i}{\pa \tilde{\lam}  _i} \Big\rangle=\sum_{j}\Big( (\al_j-\tilde{\al}_j) 	\Big\langle P\delta_j, \frac{\pa  P\delta_i}{\pa\lam_i} 	\Big\rangle  + \tilde{\al}_j \Big\langle P\delta_j- P\tilde{\delta}_j, \frac{\pa P\delta_j}{\pa\lam_i} \Big\rangle\Big).
			\end{equation}
			By Lemma \ref{lem:projbubble_lamderivative__pointwise}, we derive that 
			\begin{align}
				&\frac{\pa }{\pa\lam_j}(\chi_{j}  e^{ \delta_j-\varphi_{j}})- \frac{\pa }{\pa\tilde{\lam}   _j}(\tilde\chi_{j}  e^{\tilde{\delta}_j-\tilde\varphi_{j}} )\notag = \chi_{j}  e^{-\varphi_{j}} \Big( e^{ \delta_j}\frac{\pa \delta_j}{\pa\lam_j}   - e^{\tilde{\delta}_j}\frac{\pa \tilde{\delta}_j}{\pa\tilde{\lam}  _j} \Big)\notag \\
				=&\,-(1+o(1)) b_j \chi_{j}    e^{ \delta_j-\varphi_{j}}\frac{\pa \delta_{j}}{\pa\lam_j} + (1+o(1)) \chi_{j}  e^{ \delta_j-\varphi_{j}} \Big(2b_j + \lam_j^2|y_{\xi_j}(\tilde{\xi} _j)|^2-2 \langle \lam_jy_{\xi_j}(\tilde{\xi} _j), \lam_j y_{\xi_j}(x) \rangle_{\R^2} \Big) \notag\\
				&\,\times\Big(\frac{\pa \delta_j}{\pa \lam_j}  -\frac{ 2}{\lam_j} \Big). \label{eq:b21}
			\end{align}
			Similarly, using Lemmas \ref{lem:projbubble_xi_pointwise}  and \ref{lem:A.5},
			we deduce that 
			\[	
			\Big|\Big\langle\tilde{w}, \frac{\pa P\delta_i}{\pa \lam_i} -\frac{\pa P\tilde{\delta}_i}{\pa\tilde{\lam} _i} \Big\rangle\Big| =O\Big(\frac{ |a_i|+ |b_i|}{\lam_i}\|\tilde{w}\|\Big)=o\Big(\frac{ |a_i|+ |b_i|}{\lam_i}\Big),
			\]	in view of $\|w\|=o(1)$.
			Providing that if $\xi_i\in \pa\Si $, then $\tilde{\xi} _i\in \pa\Si $ with $y_{\xi_i}(\tilde{\xi} _i)_2=0$,	Lemma \ref{lem:projbubble_lamderivative__pointwise}   and \eqref{eq:b03}   yield that 
			\begin{align}
				\Big	\langle P\delta_j-P\tilde{\delta}_j, \frac{\pa P\delta_i}{\pa \lam_i} 	\Big\rangle =&\,  \int_{\Si } (-\Delta_g)( P\delta_j-P\tilde{\delta}_j) \frac{\pa  P\delta_i}{\pa \lam_i} \,\d v_g\notag \\
				=& \, \int_{\Si } (\chi_{j}   e^{ \delta_j-\varphi_{j}}- \tilde\chi_{j} e^{\tilde{\delta}_j-\tilde\varphi_{j}}) \Big( \frac{ 4\chi_{i}}{\lam_i(1+\lam_i^2|y_{\xi_i}|^2)} +O\Big(\frac{\ln \lam_i
				}{\lam_i^3}\Big)\Big)\,\d v_g
				\notag \\
				=&\, \left\{\begin{aligned}
					&\frac{4\kap ( \xi_i)b_i}{3\lam_i} 
					+o\Big(\frac{|a_i|+|b_i|}{\lam_i}\Big)&&\text{ if }\, j=i, \\
					&	o\Big(\frac{(|a_j|+|b_j|) }{\lam_i^2}\Big)&& \text{ if }\, i\neq j,
				\end{aligned}\right.\notag
			\end{align}
			in view of $\int_{\R^2}\frac{1}{(1+|y|^2)^3} \, \d y =\frac \pi 2$,  $ \int_{\R^2}\frac{2}{(1+|y|^2)^4} \, \d y =\frac{2\pi}{3}$, and 
			\begin{equation}
				\int_{\Si }\chi_{j}   e^{ \delta_j-\varphi_{j}}\frac{ 4\chi_{i}}{\lam_i(1+\lam_i^2|y_{\xi_i}|^2)}\,\d v_g 
				= O\Big(\frac{1}{\lam_i}\Big). \label{eq:b24}	
			\end{equation}
			Using the estimates \eqref{eq:de_or_2}-\eqref{eq:b24},  we deduce  that, for $\xi_i\in \pa\Si $,
			\begin{equation}
				\label{eq:b2_final}
				\frac{	|b_i|} {\lam_i} =O\Big(\sum_{j} \frac{|a_j+|b_j|}{\lam_i^2}\Big)  + o\Big( \sum_{j} \frac{|\iota_j|}{\lam_i^2}+ \frac{|a_i|}{\lam_i}+ \frac{|b_i|}{\lam_i} \Big).
			\end{equation}
			Combing \eqref{eq:b0_final}, \eqref{eq:b1_final} and \eqref{eq:b2_final}, we obtain
			\[ \sum_{j} \Big( |\iota_j|\ln \lam_j+ (|a_j|+|b_j|) \lam_j\Big)=o\Big( |\iota_j|\ln \lam_j+ (|a_j|+|b_j|) \lam_j\Big).  \]
			This implies that there exists a sufficiently small constant $\var_0>0$ such that for any $ \var\in (0,\var_0)$, it  holds 
			\[ a_i=0, \quad b_i=0\quad \text{ and }\quad  \iota_i=0,\]
			which contradicts to the assumption $(\al,\xi,\lam)\neq (\tilde{\al},\tilde{\xi} ,\tilde{\lam}  )$. 
			Hence, the minimizer of problem \eqref{eq:min} is unique in  $\tilde{B}^{\eta} _{4\var}$, and the estimates \eqref{eq:est_alpha_1} complete the proof.
		\end{proof}

		\section{Gradient expansions in the neighborhood at infinity}\label{app:C}
		In this appendix, we provide the details of the proof omitted in Section \ref{sec:3.2}.
		To simplify notation,  we set
		$\chi_i= \chi( |y_{\xi_i}|/ r_0)$, $\varphi_i=\varphi_{\xi_i}$, $\hat{\varphi}_i=\hat{\varphi}_{\xi_i}$, $\delta_{i}=\delta_{\xi_i,\lam_i}$ and $P\delta_{i}=P\delta_{\xi_i,\lam_i}$ for $i=1,\ldots,p+q$. 
		\begin{proof}[Proof of Proposition \ref{prop:lamexpansion}]
			It follows from \eqref{orthogonality} that
			\begin{equation}\label{eq:lam_expansion}
				\Big\langle \nabla J_{\rho }(u),\lam_i\frac{\pa P\delta_{i}}{\pa \lam_i}\Big\rangle=
				\sum_{j} \al_j\Big\langle P\delta_{j}, \lam_i\frac{\pa P\delta_{i}}{\pa \lam_i}\Big\rangle -\frac{\rho \int_{\Si } V e^u \lam_i\frac{\pa P\delta_{i}}{\pa \lam_i} \,\d v_g}{\int_{\Si } V e^u \,\d v_g}.
			\end{equation}
			We begin by estimating the first term on the R.H.S. of \eqref{eq:lam_expansion}. By 	using Lemma \ref{lem:projbubble_derivative_integral} and \eqref{eq:apha_i_estimate}, we obtain 	\begin{align}
				\sum_{j} \al_j\Big\langle P\delta_{j}, \lam_i\frac{\pa P\delta_{i}}{\pa \lam_i}\Big\rangle = &\,\al_i\Big(2\kap(\xi_i)-\frac{\kap^2( \xi_i)}{|\Si|_g}\frac{\ln \lam_i}{\lam_i^2}+O\Big(\frac{1}{\lam_i^2}\Big)\Big)\nonumber\\&\,+\sum_{j\neq i} \al_j\Big(-\frac{\kap( \xi_i)\kap( \xi_j)}{|\Si|_g}\frac{\ln \lam_i}{\lam_i^2}+O\Big(\frac{1}{\lam_i^2}\Big)\Big)\nonumber
				\\=&\,2\al_i\kap( \xi_i)- \sum_{j}  \frac{\al_j\kap( \xi_i)\kap( \xi_j)}{|\Si |_g} \frac{\ln \lam_i}{\lam_i^2}+O\Big(\sum_{j} \frac{\al_j}{\lam_i^2}\Big)\nonumber\\=&\,2\al_i\kap( \xi_i)- \rho_* \frac{\kap( \xi_i)}{|\Si |_g}\frac{\ln \lam_i}{\lam_i^2} +O\Big(\sum_{j}|\al_j-1|\frac{\ln\lam_i}{\lam_i^2}+\sum_j \frac{1}{\lam_j^2}\Big). \label{prop:lamexpansionfirstpart}
			\end{align}	
			Next, we proceed to estimate the last term on the R.H.S. of \eqref{eq:lam_expansion}.  Lemma \ref{lem:projbubble_lamderivative__pointwise} implies that 	
			\begin{align*}
				\int_\Si V e^u \lam_i \frac{\pa  P\delta_{i}}{\pa \lam_i}\, \d v_g=&\, 
				\int_\Si  V e^{u} \Big( \frac{4\chi_i }{1+\lam_i^2|y_{\xi_i}|^2}-\frac{\kap( \xi_i)}{|\Si |_g}\frac{\ln \lam_i}{\lam_i^2} +O\Big(\frac{1}{\lam_i^2}\Big) \Big) \,\d v_g \nonumber\\
				=&\,\int_{U_{r_0}(\xi_i)} \frac{4V e^{u}}{1+\lam_i^2|y_{\xi_i}|^2}   \,\d v_g -\frac{\kap( \xi_i)}{|\Si |_g} \frac{ \ln  \lam_i}{\lam_i^2}\int_\Si  V e^{u} \,\d v_g \\
				&\,+O\Big( \frac{1}{\lam_i^2}\int_\Si  V e^{u} \,\d v_g+\frac 1 {\lam_i^2}\Big),
			\end{align*}
			which  yields that 
			\[
			\frac{ \int_{\Si }  Ve^{u} \lam_i\frac{\pa P\delta_{i}  }{\pa \lam_i } \,\d v_g}{\int_{\Si } Ve^{u} \,\d v_g}= \frac{\int_{U_{r_0}(\xi_i)} \frac{4Ve^{u}}{1+\lam_i^2|y_{\xi_i}|^2}   \,\d v_g }{\int_{\Si } Ve^{u} \,\d v_g } - 	 \frac{ \kap( \xi_i)}{|\Si |_g} \frac{ \ln  \lam_i}{\lam_i^2} +O\Big( \frac{1}{\lam_i^2}\Big). 
			\]
			Now, we decompose 
			\begin{align*}
				\int_{U_{r_0}(\xi_i)} \frac{4Ve^{u}}{1+\lam_i^2|y_{\xi_i}|^2}   \,\d v_g	 =&\,\int_{U_{r_0}(\xi_i)} \frac{4Ve^{u_0}}{1+\lam_i^2|y_{\xi_i}|^2}   \,\d v_g+\int_{U_{r_0}(\xi_i)} \frac{4Ve^{u_0}}{1+\lam_i^2|y_{\xi_i}|^2}w   \,\d v_g\\&\,+\int_{U_{r_0}(\xi_i)} \frac{4Ve^{u_0}}{1+\lam_i^2|y_{\xi_i}|^2}(e^w -1-w)   \,\d v_g    
				\\:=&\,I_1+I_2+I_3.
			\end{align*}
			The rest of the proof proceeds in three steps.
			
			\textbf{Step 1: Estimate of  $I_1$.} 	Using Lemma \ref{lem:Ve^{u_0}_pointwise} we have 
			\[	
			I_1	=\int_{U_{r_0}(\xi_i)}\frac{4\lam_i^{4\al_i}\cF^i_{\xi} \rg^i_{\xi}}{ ( 1+\lam_i^2|y_{\xi_i}|^2)^{2\al_i+1}}\Big(1  +\sum_{j} \frac{\al_j\kap( \xi_j)}{2|\Si |_g} \frac{\ln  \lam_j}{\lam_j^2}+ O\Big(\sum_{j} \frac{\al_j}{\lam_j^2} \Big)\Big)\,\d v_g.
			\]
			By applying Taylor expansion and \eqref{varphixi}, we derive that 
			\begin{align}
				\int_{U_{r_0}(\xi_i)} \frac{4\lam_i^{4\al_i}\cF^i_{\xi} \rg^i_{\xi}}{ ( 1+\lam_i^2|y_{\xi_i}|^2)^{2\al_i+1}} \,\d v_g =&\, (\cF^i_{\xi} \rg^i_{\xi})(\xi_i) \int_{B^{\xi_i}_{r_0}} \frac {4\lam_i^{4\al_i}} { ( 1+\lam_i^2|y|^2)^{2\al_i+1}}\,\d y \nonumber\\&\,+\int_{B^{\xi_i}_{r_0}} \frac {4\lam_i^{4\al_i}} { ( 1+\lam_i^2|y|^2)^{2\al_i+1}}\nabla_y (e^{\hat{\varphi}_i}(\cF^i_{\xi} \rg^i_{\xi})(y_{\xi_i}^{-1}(y)))\Big|_{y=0}\cdot y \,\d y \nonumber\\&\,+O\Big( \int_{B^{\xi_i}_{r_0}} \frac {\lam_i^{4\al_i}|y|^2} { ( 1+\lam_i^2|y|^2)^{2\al_i+1}}\,\d y\Big)\nonumber\\
				=&\, I_1^{(1)}+I_1^{(2)}+I_1^{(3)}. \label{Taylor2alpha+1}
			\end{align}
			A direct calculation gives that 
			\begin{align*}
				I_1^{(1)}=&\,\frac{\kap( \xi_i)}{4\al_i}(\cF^i_{\xi} \rg^i_{\xi})(\xi_i)\lam_{i}^{4\al_i-2}+O\Big(\frac{1}{\lam_i^2}\Big),\\
				I_1^{(2)}=&\,4\lam_i^{4\al_i-3}\mathfrak{f}_2(\xi_i) \Big(B\Big(\frac{3}{2},2\al_i-\frac{1}{2}\Big)+O\Big(\frac{1}{\lam_i^{4\al_i-1}}\Big)\Big),\\	I_1^{(3)}=&\,O\Big(\frac{\kap(\xi_i)}{8} B(2,2\al_i-1)\lam_i^{4\al_i-4}+\frac 1 {\lam_{i}^2}\Big),
			\end{align*}
			where $\mathfrak{f}_2$ is defined in  \eqref{def:frak_f2} and $B(a,b)= \int_0^1 t^{a-1}(1-t)^{b-1} \, \d t$  is the Beta function.  Hence, we have 
			\begin{align*}
				\int_{U_{r_0}(\xi_i)} \frac{4\lam_i^{4\al_i}\cF^i_{\xi} \rg^i_{\xi}}{ ( 1+\lam_i^2|y_{\xi_i}|^2)^{2\al_i+1}} \,\d v_g=&\,	\frac{\kap( \xi_i)}{4\al_i}(\cF^i_{\xi} \rg^i_{\xi})(\xi_i)\lam_{i}^{4\al_i-2}+\frac{\pi}{2}\lam_i^{4\al_i-3}\mathfrak{f}_2(\xi_i) \\&\,+O\Big(\lam_i^{4\al_i-4}\Big(\frac{|\al_i-1|}{\lam_i^2}+\frac{|\mathfrak{f}_2(\xi_i)|}{\lam_i}\Big)\Big)
			\end{align*}	
			in view of  \begin{align*}
				B\Big(\frac{3}{2},2\al_i-\frac{1}{2}\Big)=&\,B\Big(\frac{3}{2},\frac{3}{2}\Big)+O(|\al_i-1|)=\frac{\pi}{8}+O(|\al_i-1|),\\	B(2,2\al_i-1)=&\,B(2,1)+O(|\al_i-1|)=\frac{1}{2}+O(|\al_i-1|).
			\end{align*}
			By	combining the above estimates in this step, we have 
			\begin{align}
				I_1=&\, \frac{\kap( \xi_i)}{4\al_i}(\cF^i_{\xi} \rg^i_{\xi})(\xi_i)\lam_{i}^{4\al_i-2}\Big(1+\sum_{j} \frac{\al_j\kap(\xi_j)}{2|\Si|_g}\frac{\ln \lam_j}{\lam_j^2}\Big) +\frac{\pi}{2}
				\lam_i^{4\al_i-3}\mathfrak{f}_2(\xi_i)\nonumber\\&\,+O\Big(\lam_i^{4\al_i-4}\Big(\frac{|\al_i-1|}{\lam_i^2}+\frac{|\mathfrak{f}_2(\xi_i)|}{\lam_i}\Big)\Big).\label{prop:lamexpansionI1fina}	
			\end{align}

			\textbf{Step 2: Estimate of  $I_2$.}   
			Applying  Lemma \ref{lem:Ve^{u_0}_pointwise}, we get 
			\[	
			I_2=\int_{U_{r_0}(\xi_i)} 	\frac{4\lam_i^{4\al_i}\cF^i_{\xi}\rg^i_{\xi}}{ ( 1+\lam_i^2|y_{\xi_i}|^2)^{2\al_i+1}}w\Big(1  +O\Big(\sum_{j} \frac{\ln  \lam_j}{\lam_j^2}\Big) \Big)  \,\d v_g.
			\]	 
			Using Taylor expansion, \eqref{Phibound} and Lemma \ref{lem:A.5}, the leading term in $I_2$ can be rewritten as 
			\begin{align*}
				\int_{U_{r_0}(\xi_i)} 	\frac{4\lam_i^{4\al_i}\cF^i_{\xi} \rg^i_{\xi}  }{ ( 1+\lam_i^2|y_{\xi_i}|^2)^{2\al_i+1}}w\,\d v_g=&\,\frac{1}{2}\lam_i^{4\al_i-2}(\cF^i_{\xi} \rg^i_{\xi})(\xi_i)\int_{B^{\xi_i}_{r_0}} \frac{e^{\tilde\delta_{i}  }}{1+\lam_i^2|y|^2}w\,\d y\\&\,+ O\Big(\lam_i^{4\al_i-2}\int_{U_{r_0}(\xi_i)}e^{\delta_{i} -\varphi_i}|w|(|\Phi _i -1|+|y_{\xi_i}|)\, \d v_g\Big)\\=&\,\frac{1}{2}\lam_i^{4\al_i-2}(\cF^i_{\xi} \rg^i_{\xi})(\xi_i)\int_{U_{r_0}(\xi_i)} \frac{e^{\delta_{i}  -\varphi_{i}}}{1+\lam_i^2|y_{\xi_i}|^2}w\,\d v_g\\&\,+ O\Big(\lam_i^{4\al_i-2}\Big(|\al_i-1|+\frac{1}{\lam_i}\Big)\|w\| \Big).	
			\end{align*}	
			Furthermore,  \eqref{projectedbubble} and  \eqref{orthogonality} yield that $\int_{\Si}\chi_{i} we^{\delta_{i}  -\varphi_{i}}\, \d v_g=0$ and	
			\[	
			\int_{\Si}\chi_{i}  \frac{\pa \delta_{i}  }{\pa \lam_i  }we^{\delta_{i}  -\varphi_{i}}\, \d v_g=\int_{\Si}\chi_{i}  \frac{2}{\lam_i}\frac{1-\lam_i^2|y_{\xi_i}|^2}{1+\lam_i^2|y_{\xi}|^2}w e^{\delta_{i}  -\varphi_{i}}\, \d v_g=0, 
			\]	 	
			which  imply that	
			\[
			\int_{\Si}\chi_{i}  \frac{ e^{\delta_{i}  -\varphi_{i}}}{1+\lam_i^2|y_{\xi_i}|^2}w\, \d v_g=0.
			\]
			Then, it follows from  the Sobolev embedding theorem that 
			\[
			\int_{B^{\xi_i}_{r_0}} \frac{e^{\tilde\delta_{i}  }}{1+\lam_i^2|y|^2}w\,\d y=-\int_{B^{\xi_i}_{2r_0}\backslash B^{\xi_i}_{r_0}} \chi_i\frac{e^{\tilde\delta_{i}  }}{1+\lam_i^2|y|^2}w\,\d y=O\Big(\frac{\|w\|}{\lam_i^{4}}\Big).
			\]
			By	combining the above estimates in this step, we have 
			\begin{equation}\label{prop:lamexpansionI2fina}
				I_2=	O\Big(\lam_i^{4\al_i-2}\Big(|\al_i-1|+\frac{1}{\lam_i}\Big)\|w\| \Big)=O\Big(\Big(|\al_i-1|+\frac{1}{\lam_i}\Big)\sum_{j}\lam_j^{4\al_j-2}\|w\| \Big).	
			\end{equation}

			\textbf{Step 3: Estimate of  $I_3$.} By Lemma   \ref{lem:Ve^{u_0}_pointwise} and \eqref{def:Phi}, one has
			\[
			I_3=O\Big( \lam_i^{4\al_i-3}\int_{U_{r_0}(\xi_i)}  e^{3\delta_{i} /2}\Phi_i  |e^w - w- 1| \,\d v_g\Big).
			\]
			Using the same argument in Lemma A.4 of  \cite{ABL2017}, we obtain 
			\[
			\int_{U_{r_0}(\xi_i)}  e^{3\delta_{i} /2} \Phi_i  |e^w - w- 1| \,\d v_g\leq C \Big(\int_{U_{r_0}(\xi_i)}e^{3\delta_{i} /2}\Phi_i^2\, \d v_g\Big)^{1/2}\|w\|^2=O(\|w\|^2).
			\]
			By	combining the above estimates in this step, we derive that 
			\begin{equation}\label{prop:lamexpansionI3fina}
				I_3=O( \lam_i^{4\al_i-3}\|w\|^2).
			\end{equation}

			Finally, we conclude from the  estimates \eqref{prop:lamexpansionfirstpart}--\eqref{prop:lamexpansionI3fina} and  Lemma \ref{lem:Ve^u_integral} that 
			\begin{align*}
				\Big\langle \nabla J_{\rho }(u),\lam_i\frac{\pa P\delta_{i} }{\pa \lam_i }\Big\rangle =&\,2\al_i \kap(\xi_i)   (\tau_i-\mu+\mu\tau_i)-\frac{\pi}{2}  \frac{\rho  \mathfrak{f}_2(\xi_i) }{\int_{\Si} Ve^{u} \,\d v_g} \lam_{i}^{4 \al_i-3}\\
				&\,+\mu\rho_*\frac{\kap(\xi_i)}{|\Si|_g}\frac{\ln\lam_i}{\lam_i^2}-(1+\mu)(1-\tau_i)\sum_{j}  \frac{\kap(\xi_i)\kap(\xi_j)}{|\Si|_g}\frac{\ln \lam_j}{\lam_j^2}\\
				&\,+O\Big(|\al_i-1|^2+\|w\|^2+\frac{1}{\lam_{i}^2}+ \frac{(\sum_{j\in}|\al_j-1|+|\mu|+|\tau_i-\mu +\mu\tau_i|)\ln \lam_i}{\lam_i^2}\Big),
			\end{align*}
			in view of 
			\begin{gather*}
				\frac{1}{\al_i^2}-\frac{1}{2\al_i-1}=O(|\al_i-1|^2),\\1-\frac{(\cF^i_{\xi} \rg^i_{\xi})(\xi_i)}{8\al_i^2}\frac{\rho_*}{\int_{\Si}Ve^u\,\d v_g}\lam_i^{4\al_i-2}=\tau_i+O(|\al_i-1|^2|).
			\end{gather*}
			This completes the proof of  \eqref{prop:lamexpansion-1}. The estimate \eqref{eq:est_tau_i} then follows directly from    \eqref{prop:lamexpansion-1},  $\|\nabla J_{\rho_*}(u)\|< \var$, $\|w\|<\var$ and the fact  $\|\lam_i\frac{\pa P\delta_{i}}{\pa \lam_i} \|=O(1)$.
			%
		\end{proof}
		
		\begin{proof}[Proof of Proposition \ref{prop:bubble_expansion}]
			It follows from \eqref{orthogonality} that
			\begin{equation}\label{eq:expansion_projectedbubble}
				\Big\langle\nabla J_{\rho}(u), \frac{P \delta_{i} }{\ln \lam_i}\Big\rangle=\sum_{j }  \al_j\Big\langle P \delta_{j} ,\frac{ P \delta_{i} }{\ln \lam_i}\Big\rangle-\frac{\rho \int_{\Si} Ve^{u} \frac{P  \delta_{i} }{\ln \lam_i}\, \d v_g}{\int_{\Si} Ve^{u} \, \d v_g}  .
			\end{equation}
			We begin by estimating the first term on the R.H.S. of  \eqref{eq:expansion_projectedbubble}. Lemma \ref{lem:projbubble_integral_interaction} yields that 
			\begin{align}
				\sum_{j } \al_j \Big\la P\delta_{j}  ,\frac{ P\delta_{i} }{\ln \lam_i}\Big\ra
				=&
				\,\frac{\al_i }{\ln \lam_i} (4 \kap(\xi_i)\ln \lam _i+\kap^2(\xi_i) R^g(\xi_i)-2\kap(\xi_i)) \nonumber\\&\,+ \sum_{j\neq i} \frac{\al_j}{\ln \lam_i}\kap(\xi_j)\kap(\xi_i) G^g(\xi_i,\xi_j)
				+O\Big( \sum_{j }  \frac{1}{\ln \lam_i}\frac{\ln \lam_j}{\lam_j^{2}}\Big).\label{eq:expansion_projectedbubble_first}
			\end{align}
			Next, we  proceed to estimate the second term on the R.H.S. of  \eqref{eq:expansion_projectedbubble}. To this end,   we decompose
			\begin{align*}
				\int_{\Si} Ve^{u} \frac{P  \delta_{i} }{\ln \lam_i}\, \d v_g=&\,	\int_{\Si} V e^{u_0} \frac{P\delta_{i}  }{\ln \lam_i}   \, \d v_g  + \int_{\Si} V e^{u_0} \frac{P\delta_{i}  }{\ln \lam_i}  w \, \d v_g  +\int_{\Si} V e^{u_0} \frac{P\delta_{i}  }{\ln \lam_i} (e^w-1-w) \, \d v_g 
				\\	:=&\, I_1+I_2+I_3.
			\end{align*}
			The estimates for $I_1$, $I_2$ and $I_3$ will be carried out in the following steps.

			\textbf{Step 1: Estimate of $I_1$.} 
			By Lemmas \ref{lem:projbubble_pointwise} and \ref{lem:Ve^{u_0}_pointwise}, we get   
			\begin{align*}
				I_1 \times \ln \lam_i
				=&\,\sum_{j }  \int_{U_{r_0}(\xi_j)} \frac{ \lam_{j}^{4\al_j}\cF^j_{\xi} \rg^j_{\xi}( 1+ O(\sum_{s }\frac{\ln \lam_s}{\lam_s^{2}}))}{(1+\lam_j^2|y_{\xi_j}|^2)^{2\al_j}} \Big( 2\chi_{i}\ln\Big(\frac{\lam_{i}^2}{1+\lam_{i}^2|y_{\xi_i}|^2}\Big) \\&\,+ \kap(\xi_i) H^g(\cdot, \xi_i) + O\Big(\frac{\ln \lam_i}{\lam_i^{2}}\Big)\Big) \, \d v_g + O(1) \\
				=&\,\int_{U_{r_0}(\xi_i)} \frac{ \lam_{i}^{4\al_i}\cF^i_{\xi} \rg^i_{\xi}}{(1+\lam_i^2|y_{\xi_i}|^2)^{2\al_i}}(4\ln \lam_{i} -2\ln(1+\lam_{i}^2|y_{\xi_i}|^2))   \, \d v_g\nonumber\\&\,+ \sum_{j } \int_{U_{r_0}(\xi_j)} \frac{ \kap(\xi_i) \lam_j^{4\al_j} \cF^j_{\xi} \rg^j_{\xi} H^g(\cdot,\xi_i) }{(1+\lam_{j}^2|y_{\xi_j}|^2)^{2\al_j}}  \, \d v_g  \\
				&\,
				+O\Big(  \sum_{j } 
				\int_{U_{r_0}(\xi_j)}
				\frac{\lam_j^{4\al_j} \cF^j_{\xi} \rg^j_{\xi}  }{(1+\lam_{j}^2|y_{\xi_j}|^2)^{2\al_j}} \frac{\ln \lam_i}{\lam_i^{2}} \, \d v_g + 1 \Big)\\
				:=&\, I^{(1)}_1+ I^{(2)}_1+  I^{(3)}_1.
			\end{align*}
			We decompose
			\begin{align*}
				I^{(1)}_1=&\,\int_{U_{r_0}(\xi_i)} \frac{ 4\lam_{i}^{4\al_i}\ln \lam_{i}\cF^i_{\xi} \rg^i_{\xi}}{(1+\lam_i^2|y_{\xi_i}|^2)^{2\al_i}}\, \d v_g-\int_{U_{r_0}(\xi_i)} \frac{ 2\lam_{i}^{4\al_i}\ln(1+\lam_{i}^2|y_{\xi_i}|^2)\cF^i_{\xi} \rg^i_{\xi}}{(1+\lam_i^2|y_{\xi_i}|^2)^{2\al_i}} \, \d v_g
				\\:=&\, I^{(1,1)}+I^{(1,2)}.
			\end{align*}
			It follows from  \eqref{eq:est_Fg_0rd} that 	
			\begin{align*}
				I^{(1,1)}=	&\,\frac{\kap ( \xi_i)(\cF^i_{\xi}\rg^i_{\xi})(\xi_i) }{2(2\al_i-1)}\lam_{i}^{4\al_i-2}\ln\lam_i+4 B\Big(\frac{3}{2},2\al_i-\frac{3}{2}\Big)\mathfrak{f}_2(\xi_i)\lam_i^{4\al_i-3}\ln\lam_i\nonumber\\
				&\,+O\Big(\lam_i^{4\al_i-2}\Big(\frac{\ln^2 \lam_i}{\lam_i^2}+\frac{|\al_i-1|\ln\lam_i}{\lam_i^{3/2}}\Big)+\ln\lam_i\Big).
			\end{align*}
			Using Taylor expansion and  \eqref{varphixi}, we derive that 
			\begin{align*}
				I^{(1,2)} =&\,-2(\cF^i_{\xi} \rg^i_{\xi})(\xi_i) \int_{B^{\xi_i}_{r_0}} \frac {\lam_i^{4\al_i} \ln( 1+\lam_i^2|y|^2)} { ( 1+\lam_i^2|y|^2)^{2\al_i}}\,\d y \nonumber\\&\,-2\int_{B^{\xi_i}_{r_0}} \frac {\lam_i^{4\al_i} \ln( 1+\lam_i^2|y|^2)} { ( 1+\lam_i^2|y|^2)^{2\al_i}}\nabla_y (e^{\hat\varphi_i}\cF^i_{\xi} \rg^i_{\xi}(y_{\xi_i}^{-1}(y)))\Big|_{y=0}\cdot y \,\d y\\&\,+O\Big( \int_{B^{\xi_i}_{r_0}} \frac {\lam_i^{4\al_i}|y|^2 \ln( 1+\lam_i^2|y|^2)} { ( 1+\lam_i^2|y|^2)^{2\al_i}}\,\d y\Big)\\
				=&\,-\frac{\kap(\xi_i) (\cF^i_{\xi} \rg^i_{\xi})(\xi_i)}{4(2\al_i-1)^2}  \lam_i^{4\al_i-2}- 2\mathfrak{B}_{\al_i} \mathfrak{f}_2(\xi_i)  \lam_i^{4\al_i-3} \\
				&\,+O\Big(\lam_i^{4\al_i-2}\Big(\frac{\ln^2\lam_i}{\lam_i^2}+\frac{|\al_i-1|}{\lam_{i}^{3/2}} \Big)+\ln\lam_i\Big),	
			\end{align*}
			where $\mathfrak{B}_{\al_i}:=\int_{0}^{+\infty}\frac{\sqrt{t}\ln (1+t)}{(1+t)^{2\al_i}}\, \d t$,  and we used the estimates
			\begin{align*}
				\int_{B^{\xi_i}_{r_0}} \frac {\lam_i^{4\al_i} \ln( 1+\lam_i^2|y|^2)} { ( 1+\lam_i^2|y|^2)^{2\al_i}}\,\d y=&\,\frac{\kap(\xi_i)}{8}\lam_i^{4\al_i-2}\Big(\frac{1}{(2\al_i-1)^2}+O\Big(\frac{\ln\lam_i}{\lam_i^{4\al_i-2}}\Big)\Big),\\\mathfrak{f}_2(\xi_i)\int_{B^{\xi_i}_{r_0}} \frac {\lam_i^{4\al_i} \ln( 1+\lam_i^2|y|^2)y_2} { ( 1+\lam_i^2|y|^2)^{2\al_i}} \,\d y=&\,\mathfrak{f}_2(\xi_i)\Big(\lam_i^{4\al_i-3}\int_{0}^{+\infty}\frac{\sqrt{t}\ln (1+t)}{(1+t)^{2\al_i}}\,\d t+O(\ln\lam)\Big),
			\end{align*}
			and 
			\[
			\int_{B^{\xi_i}_{r_0}} \frac {\lam_i^{4\al_i}|y|^2 \ln( 1+\lam_i^2|y|^2)} { ( 1+\lam_i^2|y|^2)^{2\al_i}}\,\d y=O(\lam_i^{4\al_i-4}\ln\lam_i(\ln\lam_i+|\al_i-1|\sqrt{\lam_i})).	
			\]
			Thus, we get
			\begin{align*}
				I^{(1)}_1=&\, \frac{\kap( \xi_i)(\cF^i_{\xi} \rg^i_{\xi})(\xi_i)}{2(2\al_i-1)} \lam_{i}^{4\al_i-2}\ln \lam_i+4 B\Big(\frac{3}{2},2\al_i-\frac{3}{2}\Big)\mathfrak{f}_2(\xi_i) \lam_i^{4\al_i-3} \ln \lam_i \\
				&\,- \frac{\kap(\xi_i)(\cF^i_{\xi} \rg^i_{\xi})(\xi_i)}{4(2\al_i-1)^2}   \lam_i^{4\al_i-2}  -2 \mathfrak{B}_{\al_i} \mathfrak{f}_2(\xi_i)   \lam_i^{4\al_i-3}\\
				&\,+ O\Big(\lam_i^{4\al_i-2}\Big(\frac{\ln^2\lam_i}{\lam_i^2}+\frac{|\al_i-1|\ln \lam_{i}}{\lam_{i}^{3/2}} \Big)+\ln\lam_i\Big). 
			\end{align*}
			Note that   $\pa_{\nu_g} G^g(x,\xi_i)=0$ for any $x\in  \pa\Si\setminus\{\xi_i\}$, and for $\xi_i\in \pa\Si$, it holds  $\pa_{\nu_g} H^g(x,\xi_i)|_{x=\xi_i}=0$. Moreover, by the definition of $H^g(x,\xi_i)$, we have $H^g(x,\xi_i)= G^g(x,\xi_i)$ for any $x\in U_{r_0}(\xi_j)$ with $j\neq i$.  Proceeding as in the estimate \eqref{eq:est_Fg_0rd}, we have 
			\begin{align*}
				I^{(2)}_1=&\, \sum_{j }\int_{B^{\xi_j}_{r_0}} \frac{ \kap(\xi_i) \lam_j^{4\al_j} (H^g(\xi_j,\xi_i)(\cF^j_{\xi} \rg^j_{\xi})(\xi_j))   }{(1+\lam_{j}^2|y|^2)^{2\al_j}}  \, \d y 
				\\
				&\,+ \sum_{j }  \int_{B^{\xi_j}_{r_0}}\frac{ \kap(\xi_i) \lam_j^{4\al_j} \nabla_y(( e^{\hat\varphi_j}  \cF^j_{\xi} \rg^j_{\xi} H^g(\cdot,\xi_i))\circ y_{\xi_j}^{-1})(0)\cdot y  }{(1+\lam_{j}^2|y_{\xi_j}|^2)^{2\al_j}}  \, \d y \\
				&\,+ O\Big(\sum_{j }  \int_{B^{\xi_j}_{r_0}} \frac {\lam_j^{4\al_j}|y|^2 } { ( 1+\lam_j^2|y|^2)^{2\al_j}}\,\d y\Big)\\
				=&\, \sum_{j } \frac{ \kap(\xi_i)\kap(\xi_j) }{ 8(2\al_j-1)}( H^g(\xi_j,\xi_i)(\cF^i_{\xi}\rg^i_{\xi})(\xi_j) )\lam_j^{4\al_j-2}\\&\,+ \sum_{j } B\Big(\frac{3}{2},2\al_j-\frac{3}{2}\Big) \kap(\xi_i)  H^g(\xi_j,\xi_i)\mathfrak{f}_2(\xi_j) \lam_j^{4\al_j-3} \\
				&\,+ O\Big( \sum_{j } \lam_{j}^{4\al_j-2} 
				\Big( \frac{\ln \lam_j}{\lam_{j}^2}+ \frac{ |\al_j-1|}{\lam_{j}^{3/2}}\Big) +1\Big).
			\end{align*}
			From  \eqref{eq:est_Fg_0rd} we also know that 
			\[
			I^{(3)}_1=O\Big(\sum_{j } \lam_j^{4\al_j-2}\frac{\ln \lam_i}{\lam_i^2}+ 1\Big).
			\]
			Summing the estimates  for $I_1^{(1)}$, $I_1^{(2)}$ and $I_1^{(3)}$, we obtain 
			\begin{align}	
				I_1=& \,\frac{\kap( \xi_i)(\cF^i_{\xi} \rg^i_{\xi})(\xi_i)}{2(2\al_i-1)} \lam_{i}^{4\al_i-2}+4 B\Big(\frac{3}{2},2\al_i-\frac{3}{2}\Big)\mathfrak{f}_2(\xi_i) \lam_i^{4\al_i-3} \nonumber\\
				&\,- \frac{\kap(\xi_i)(\cF^i_{\xi} \rg^i_{\xi})(\xi_i)}{4(2\al_i-1)^2}   \frac{\lam_i^{4\al_i-2}}{\ln \lam_i}  -2 \mathfrak{B}_{\al_i}  \mathfrak{f}_2(\xi_i)\frac{ \lam_i^{4\al_i-3}}{\ln \lam_i}\nonumber\\
				& +\sum_{j } \frac{ \kap(\xi_i)\kap(\xi_j) }{ 8(2\al_j-1)}( H^g(\xi_j,\xi_i)(\cF^i_{\xi} \rg^i_{\xi})(\xi_j))\frac{\lam_j^{4\al_i-2}}{\ln \lam_i}\nonumber\\&\,+ \sum_{j }  B\Big(\frac{3}{2},2\al_j-\frac{3}{2}\Big) \kap(\xi_i)H^g(\xi_j,\xi_i)\mathfrak{f}_2(\xi_j) \frac{\lam_j^{4\al_j-3}}{\ln \lam_{i}} \nonumber\\
				&\,+ O\Big(1+\sum_{j } \lam_j^{4\al_j-2} \Big(\frac{\ln\lam_j}{\lam_j^2}+\frac{|\al_j-1|}{\lam_{j}^{3/2}} \Big)\Big). \label{prop:bubble_expansionI1final}
			\end{align}

			\textbf{Step 2: Estimate of $I_2$.}  
			By applying Lemmas \ref{lem:projbubble_pointwise} and \ref{lem:w_error_integral_estimates}, together with  the H\"older inequality and  the Sobolev embedding theorem, we deduce that  
			\begin{align*}
				\int_{\Si}e^{u_0}|w|\,\d v_g=&\,\sum_{j } 	\int_{U_{r_0}(\xi_j)} e^{u_0}|w|\,\d v_g+\int_{\Si\setminus \cup_{j } U_{r_0}(\xi_j)} e^{u_0}|w|\,\d v_g\\=&\,O\Big(\sum_{j }  \int_{U_{r_0}(\xi_j)} \frac{\lam_j^{4\al_j}}{(1+\lam_j|y_{\xi_j}|^2)^{2\al_j}}|w|\,\d v_g\Big)+O(\|w\|)
				\\=&\,O\Big(\sum_{j }   \lam_j^{4\al_j-2}\Big(\int_{U_{r_0}(\xi_j)} e^{\delta_{j} }|w|^2\,\d v_g\Big)^{1/2}\Big(\int_{U_{r_0}(\xi_j)} \frac{\lam_j^2}{(1+\lam_j^2|y_{\xi_j}|^2)^{4\al_j-2}} \,\d v_g\Big)^{1/2}\Big)\\
				&\,+O(\|w\|)\nonumber
				\\=&\,O\Big(\sum_{j }  \lam_j^{4\al_j-2}\|w\|\Big).
			\end{align*}
			As a result, using Lemma \ref{lem:projbubble_pointwise} again, we obtain  
			\begin{align*}
				I_2= &\,\frac1 {\ln \lam_{i}} \int_{U_{r_0}(\xi_i)} V e^{u_0} P\delta_{i}  w\,\d v_g+O\Big( \frac1 {\ln \lam_{i}}\int_{\Si \setminus U_{r_0}(\xi_i)} e^{u_0}  | w|\,\d v_g\Big)\\=&\, \frac1 {\ln \lam_{i}}\int_{U_{r_0}(\xi_i)} V e^{u_0} P\delta_{i}  w \,\d v_g+O\Big(\sum_{j }  \frac{\lam_j^{4\al_j-2} } {\ln \lam_{i}}\|w\|\Big).
			\end{align*}
			Moreover, we deduce from  Lemmas \ref{lem:projbubble_pointwise} and \ref{lem:Ve^{u_0}_pointwise} that
			\begin{align*}
				\int_{U_{r_0}(\xi_i)} V e^{u_0} P\delta_{i}  w\,\d v_g=&\,	\int_{U_{r_0}(\xi_i)} \frac{\lam_i^{4\al_i}\cF^i_{\xi}\rg^i_{\xi}}{ ( 1+\lam_i^2|y_{\xi_i}|^2)^{2\al_i}}\Big(1  +\sum_{j }  \frac{\al_j\kap( \xi_j)}{2|\Si |_g} \frac{\ln  \lam_j}{\lam_j^2}+ O\Big( \sum_{j } \frac{\al_j}{\lam_j^2} \Big)\Big) \\&\,\times \Big(2\ln \Big(\frac{\lam_i^2}{1+\lam_i^2|y_{\xi_i}|^2}\Big)+O(1)\Big)w \,\d v_g
				\\=&\,\frac{\lam_i^{4\al_i-2}}{2}\ln\lam_i\int_{U_{r_0}(\xi_i)}e^{\delta_{i} }\Phi_i \cF^i_{\xi}\rg^i_{\xi}w\,\d v_g\\&\,+O\Big(	\int_{U_{r_0}(\xi_i)}\frac{\lam_i^{4\al_i}\ln (1+\lam_i^2|y_{\xi_i}|^2)}{ ( 1+\lam_i^2|y_{\xi_i}|^2)^{2\al_i}}|w|\, \d v_g\Big)\\&\,+O\Big(	\int_{U_{r_0}(\xi_i)}\frac{\lam_i^{4\al_i}}{ ( 1+\lam_i^2|y_{\xi_i}|^2)^{2\al_i}}|w|\, \d v_g\Big).
			\end{align*}
			By	using Taylor expansion, \eqref{Phibound}, \eqref{orthogonality} and Lemma \ref{lem:A.5}, we have
			\begin{align*}
				\int_{U_{r_0}(\xi_i)}e^{\delta_{i} }\Phi_i \cF^i_{\xi}\rg^i_{\xi}w\,\d v_g=&\,\int_{U_{r_0}(\xi_i)}e^{\delta_{i} } \cF^i_{\xi}\rg^i_{\xi}w\,\d v_g+O\Big(\int_{U_{r_0}(\xi_i)}e^{\delta_{i} }|\Phi_i-1| |w|\,\d v_g\Big)
				\\=&\,(\cF^i_{\xi}\rg^i_{\xi})(\xi_i)\int_{U_{r_0}(\xi_i)} e^{\delta_{i} } w\,\d v_g\\&\,+O\Big(\int_{U_{r_0}(\xi_i)}e^{\delta_{i} }(|y_{\xi_i}|+|\al_i-1|\sqrt{\lam_i|y_{\xi_i}|}) |w|\,\d v_g\Big)
				\\=&\,O\Big(\frac{\|w\|}{\lam_i^2}\Big)+O\Big(\Big(\frac{1}{\lam_i}+|\al_i-1|\Big)\|w\|\Big)
				\\=&\, O\Big(\Big(\frac{1}{\lam_i}+|\al_i-1|\Big)\|w\|\Big).
			\end{align*}
			Applying the H\"older inequality  together with Lemma \ref{lem:w_error_integral_estimates}  gives that 
			\[
			O\Big(\int_{U_{r_0}(\xi_i)}\frac{\lam_i^{4\al_i}\ln (1+\lam_i^2|y_{\xi_i}|^2)}{ ( 1+\lam_i^2|y_{\xi_i}|^2)^{2\al_i}}|w|\, \d v_g+\int_{U_{r_0}(\xi_i)}\frac{\lam_i^{4\al_i}}{ ( 1+\lam_i^2|y_{\xi_i}|^2)^{2\al_i}}|w|\, \d v_g\Big)=O(\lam_i^{4\al_i-2}\|w\|).
			\]
			By combining the estimates obtained in this step, we derive that 
			\begin{equation}\label{prop:bubble_expansionI2final}
				I_2=
				O\Big(\Big(|\al_i-1|+\frac{1}{\ln\lam_i }\Big)\sum_{j }  \lam_j^{4\al_j-2}\|w\|\Big).
			\end{equation}

			\textbf{Step 3: Estimate of $I_3$.}
			Lemma \ref{lem:projbubble_pointwise} implies that 
			\[	I_3=O\Big(\frac 1{\ln \lam_{i}}\int_{\Si} V e^{u_0}|e^w-1-w|\, \d v_g\Big).\]	
			By the same argument  as in Step 3 of the proof of Lemma \ref{lem:Ve^u_integral}, we obtain
			\begin{equation}\label{prop:bubble_expansionI3final}
				I_3=O\Big(\sum_{j}\frac{\lam_j^{4\al_j-2}}{\ln \lam_i} \|w\|^2\Big).
			\end{equation}

			Finally, we conclude  from   the  estimates \eqref{eq:expansion_projectedbubble_first}--\eqref{prop:bubble_expansionI3final} and Lemma  \ref{lem:Ve^u_integral} that  
			\begin{align*}
				\Big\la\nabla J_{\rho} (u), \frac{P \delta_i  }{\ln \lam_i }\Big\ra=&\, 4\kap (\xi_i)( (\al _i-1)+(\tau_i-\mu+\mu\tau_i))\nonumber\\
				&\,- \Big(4 B\Big(\frac{3}{2},2\al_i-\frac{3}{2}\Big)  +\frac{2\mathfrak{B}_{\al _i}}{\ln \lam_{i}} 
				\Big) \frac{\rho \mathfrak{f}_2(\xi_i)}{\int_\Si  V e^u  \, \d v_g} \lam_i^{4\al_i-3}\nonumber \\
				&\,-\sum_{j }  B\Big(\frac{3}{2},2\al_j-\frac{3}{2}\Big) \frac{\rho  H^g(\xi_j,\xi_i)\mathfrak{f}_2(\xi_j) }{\int_\Si   V e^u \, \d v_g } \frac{\lam_{j}^{4\al _j-3}}{\ln \lam _i}\nonumber\\
				&
				\,+O \Big(\sum_{j }  \frac{|\al _j-1|}{\ln \lam _i}+\sum_{j }  \frac{ |\tau_j-\mu+\mu\tau_j|}{\ln \lam _i}+\Big(\frac 1 {\ln \lam_{i}}+ |\al _i-1|\Big)\|w\|+\|w\|^2\Big)\nonumber\\
				&\,+ O \Big(\frac{\ln \lam _i}{\lam_{i}^2}+\frac 1 {\lam _i^{4\al _i-2}}\Big).
			\end{align*}
			Combining Lemma \ref{lem:Ve^u_integral} with \eqref{eq:B(3/2,1/2)} and  $\mathfrak{B}_{\al _i}=\mathfrak{B}_{1}+O(|\al_i-1|)\to \frac{\pi}{2}(1+2\ln 2)$ as $\al_i\to 1$, 
			the desired estimate	 \eqref{eq:est_PU_0_gradient}  follows.
		\end{proof}

		\begin{proof}[Proof of Proposition \ref{prop:3}]
			It follows from \eqref{orthogonality} that
			\begin{equation}
				\label{eq:expansion_projectedbubble_xi}
				\Big\langle\nabla J_{\rho }(u), \frac{1}{\lam_i}\frac{\pa P\delta_{i}  }{\pa (\xi_i)_j}\Big\rangle=\sum_{s}  \al_s\Big\langle P \delta_{s} , \frac{1}{\lam_i}\frac{\pa P\delta_{i}  }{\pa (\xi_i)_j}\Big\rangle-\frac{\rho \int_{\Si } Ve^{u} \frac{1}{\lam_i}\frac{\pa P\delta_{i}  }{\pa (\xi_i)_j}\, \d v_g}{\int_{\Si } Ve^{u} \,\d v_g}.
			\end{equation}
			We begin by estimating the first term on the R.H.S. of  \eqref{eq:expansion_projectedbubble_xi}. Using Lemma \ref{lem:projbubble_derivative_integral}, we have
			\[
			\sum_{s}  \al_s\Big\langle P \delta_{s} , \frac{1}{\lam_i}\frac{\pa P\delta_{i}  }{\pa (\xi_i)_j}\Big\rangle= \al_i\frac{\kap^2( \xi_i)}{2\lam_i} \frac{ \pa  R^g(\xi_i) }{\pa (\xi_i)_j}+\sum_{s\neq i} \al_s \frac{\kap( \xi_i)\kap(\xi_s)}{\lam_i} \frac{ \pa  G^g(\xi_i,\xi_s) }{\pa (\xi_i)_j}+O\Big(\sum_{s}\frac{\ln \lam_i }{\lam_i^{3} }  \Big).
			\]
			Next, we proceed to estimate the second term in the R.H.S. of \eqref{eq:expansion_projectedbubble_xi}.	 To this end,   we decompose
			\begin{align*}
				\int_{\Si } \frac{Ve^{u}}{\lam_i}\frac{\pa P\delta_{i}  }{\pa (\xi_i)_j} \,\d v_g=&	\,\int_{\Si }   \frac{V e^{u_0}}{\lam_i} \frac{\pa P\delta_{i}  }{\pa (\xi_i)_j} \,\d v_g+	\int_{\Si }  \frac{V e^{u_0}}{\lam_i} \frac{\pa P\delta_{i}  }{\pa (\xi_i)_j} w\,\d v_g
				\\&\,+	\int_{\Si }  \frac{V e^{u_0}}{\lam_i} \frac{\pa P\delta_{i}  }{\pa (\xi_i)_j} (e^w-w-1) \,\d v_g\\:=&\, I_1+I_2+I_3.	
			\end{align*}
			The estimates for $I_1$, $I_2$ and $I_3$ will be carried out in the following steps.
			
			\textbf{Step 1: Estimate of $I_1$.} 
			Using  Lemmas \ref{lem:projbubble_xi_pointwise} and \ref{lem:Ve^{u_0}_pointwise}, we get
			\begin{align*}
				I_1= & \,\sum_{s}  \int_{U_{r_0}(\xi_s)} \frac{\lam_{s}^{4\al_s}\mathcal{F}_{\xi}^s \mathrm{g}_{\xi}^s(1+O(\frac{\ln \lam_s}{\lam_s^2}))}{(1+\lam_s^2|y_{\xi_s}|)^{2 \al_s}} \frac{1}{\lam_i}\Big(\frac{-4\chi_{i} \lam_{i}^2 (y_{\xi_i})_j}{1+\lam_i^2|y_{\xi_i}|^2}\\&\,+\kap(\xi_i)\frac{ \pa H^g(\cdot, \xi_i)}{\pa (\xi_i)_j}+O\Big( \frac{\ln \lam_i}{\lam_i^2}\Big)\Big)\, \d v_g 
				+O\Big(\frac{1}{\lam_i}\Big) \\
				=&\, -4 \int_{U_{r_0}(\xi_i)} \frac{\lam_i^{4\al_i+1}\cF_{\xi}^i\rg_{\xi}^i(y_{\xi_i})_j}{(1+\lam_i^2|y_{\xi_i}|^2)^{2\al_i+1}}\, \d v_g \nonumber\\&\,+\sum_{s}  \int_{U_{r_0}(\xi_s)}   \frac{\lam_s^{4\al_s}\cF_{\xi}^s\rg_{\xi}^s   }{(1+\lam_s^2|y_{\xi_s}|^2)^{2\al_s}}\frac{\kap(\xi_i)}{\lam_i}\frac{ \pa H^g(\cdot, \xi_i)}{\pa (\xi_i)_j}\, \d v_g  \nonumber  \\
				&\,+O\Big(\sum_{s}  \frac{\ln \lam_i}{\lam_i^3} \int_{U_{r_0}(\xi_s)}   \frac{\lam_s^{4\al_s}\cF_{\xi}^s\rg_{\xi}^s }{(1+\lam_s^2|y_{\xi_s}|^2)^{2\al_s}}\, \d v_g  +\frac 1 {\lam_i} \Big)\\
				:=&\, I^{(1)}_1+I^{(2)}_1+O(I^{(3)}_1).
			\end{align*}
			By Taylor expansion and \eqref{varphixi}, we obtain 
			\begin{align*}
				I_1^{(1)}	=&\, -4(\cF^i_{\xi}\rg^i_{\xi})(\xi_i) \int_{B_{r_0}^{\xi_i}} 	\frac{ \lam_i^{4\al_i+1}   y_j
				}{ ( 1+\lam_i^2|y|^2)^{2\al_i+1}}  \,\d y\\&\, -4\int_{B_{r_0}^{\xi_i}} 	\frac{  \lam_i^{4\al_i+1} y_j 
				}{ ( 1+\lam_i^2|y|^2)^{2\al_i+1}} \nabla_y (e^{\hat\varphi_{i}(y)}(\cF^i_{\xi} \rg^i_{\xi})(y_{\xi_i}^{-1}(y))\Big|_{y=0}\cdot y \,\d y\\&\,+ O\Big(\int_{B_{r_0}^{\xi_i}} 	\frac{ \lam_i^{4\al_i+1} |y|^3 
				}{ ( 1+\lam_i^2|y|^2)^{2\al_i+1}} \,\d y\Big)\\
				=&\,- \frac{\kap(\xi_i)}{8\al_i(2\al_i-1)}\frac{\pa(e^{\varphi_{i}}\cF_{\xi}^i \rg_{\xi}^i )(\xi_i)}{\pa (\xi_i)_j }\lam_i^{4\al_i-3}
				+	O(\lam_i^{4\al_i-4}),
			\end{align*}
			where we used the estimates 
			\begin{gather*}
				\int_{B_{r_0}^{\xi_i}} 	\frac{   y_j 
				}{ ( 1+\lam_i^2|y|^2)^{2\al_i+1}}\, \d y=0, \quad \forall\, j\in\{1,\ii(\xi_i)\} ,\\\int_{B_{r_0}^{\xi_i}} 	\frac{ 4 \lam_i^{4\al_i+1} y_j^2 
				}{ ( 1+\lam_i^2|y|^2)^{2\al_i+1}}\, \d y=\frac{\kap(\xi_i)}{8}\frac{\lam_i^{4\al_i-3}}{\al_i(2\al_i-1)}+O\Big(\frac 1 \lam_i\Big),
			\end{gather*}
			and
			\[ 
			\int_{B_{r_0}^{\xi_i}} 	\frac{ \lam_i^{4\al_i+1} |y|^3 
			}{ ( 1+\lam_i^2|y|^2)^{2\al_i+1}} \,\d y =O(\lam_i^{4\al_i-4} ).
			\]
			Similarly, 	by using that  for any $x\in  \pa\Si\setminus\{\xi_s\}$, $\pa_{\nu_g} G^g(x,\xi_s)=0$ and for $\xi_s\in \pa\Si$, $\pa_{\nu_g} H^g(x,\xi_i)|_{x=\xi_i}=0$,    $H^g(x,\xi_s)= G^g(x,\xi_s)$ for any $x\in U_{r_0}(\xi_s)$ with $s\neq i$, we have
			\begin{align*}
				I^{(2)}_1
				=&\, \sum_{s}  \frac{\kap(\xi_i)}{\lam_i} (\cF^s_{\xi} \rg^s_{\xi})(\xi_s)	\frac{\pa  H^g(x,\xi_i)}{\pa (\xi_i)_j}\Big|_{x=\xi_s}	\int_{U_{r_0}(\xi_s)} 	\frac{\lam_s^{4\al_s} 
				}{( 1+\lam_s^2|y_{\xi_s}|^2)^{2\al_s}} \,\d v_g  \\&\,+O\Big(\sum_{s}\frac{\lam_s^{4\al_s-4}(\ln\lam_s+|\al_s-1|\sqrt{\al_s})}{\lam_i}\Big) \\
				=& \,\sum_{s}  \frac{1}{\lam_{i}}\frac {\kap( \xi_i)\kap(\xi_s)} {8(2\al_s-1)}\frac{\pa  H^g(x,\xi_i)}{\pa (\xi_i)_j}\Big|_{x=\xi_s} (\cF^s_{\xi} \rg^s_{\xi})(\xi_s)
				\lam_s^{4\al_s-2}\\&\,+O\Big(\sum_{s}\frac{\lam_s^{4\al_s-4}(\ln\lam_s+|\al_s-1|\sqrt{\al_s})}{\lam_i}\Big),
			\end{align*}	
			where we used the estimates 
			\[	
			\int_{B_{r_0}^{\xi_s}} 	\frac{\lam_s^{4\al_s} 
			}{( 1+\lam_s^2|y|^2)^{2\al_s}} \,\d y=\frac{\kap(\xi_s)}{4}\lam_s^{4\al_s-2}\Big(\frac{1}{2(2\al_i-1)}+O\Big(\frac{1}{\lam_s^{4\al_s-2}}+\frac{1}{\lam_s^2}\Big)\Big)
			\]
			and
			\[
			\int_{B_{r_0}^{\xi_s}} 	\frac{  \lam_s^{4\al_s} |y|^2 
			}{ ( 1+\lam_s^2|y|^2)^{2\al_s}}\, \d y=O\Big(\sum_{s}\lam_s^{4\al_s-4}(\ln\lam_s+|\al_s-1|\sqrt{\al_s})\Big).
			\]
			By \eqref{eq:est_Fg_0rd}, we have 
			\[ 
			I^{(3)}_1= O\Big(\sum_{s}  \lam_s^{4\al_s-2}\frac{\ln\lam_i}{\lam_i^3} +\frac{1}{\lam_i}\Big).
			\]
			Observe that for any $s\in\{1,\ldots,p+q\}$, $H^g(x,\xi_s)=G^g(x,\xi_s)$ for any $x\notin U_{2r_0}(\xi_s)$.
			Therefore, we have 
			\begin{align*}
				I_1=&\, -\frac{\kap(\xi_i)}{8\al_i(2\al_i-1)} \frac{\pa(e^{\varphi_{i}}\cF_{\xi}^i \rg_{\xi}^i )(\xi_i)}{\pa (\xi_i)_j}\lam_i^{4\al_i-3}+
				\frac {\kap ^2(\xi_i)(\cF^i_{\xi}\rg^i_{\xi})(\xi_i)}{16(2\al_i-1)}  
				\frac{\pa R(\xi_i)}{\pa(\xi_i)_j}\lam_i^{4\al_i-3}\\ &\,+\sum_{s\neq i} \frac{\kap( \xi_s)\kap( \xi_i)(\cF^s_{\xi} \rg^s_{\xi} )(\xi_s)}{8(2\al_s-1)\lam_i}  \frac{\pa  G^g(\xi_s,\xi_i)}{\pa (\xi_i)_j}\lam_s^{4\al_s-2}
				\\
				&\,+	O(\lam_i^{4\al_i-4}\ln  \lam_{i}). 
			\end{align*}
			Since $\rg_{\xi}^s(\xi_s)=1+O(|\al_s-1|)$
			and \[ 
			\kap(\xi_i)\frac{\pa ( e^{\varphi_{i}}\cF_{\xi}^i \rg_{\xi}^i)(\xi_i)}{\pa (\xi_i)_j}= (\cF^i_{\xi}\rg^i_{\xi})(\xi_i)\Big(\frac{1}{2}\frac{\pa \ff(\xi)}{\pa(\xi_i)_j}+ O(|\al_i-1|)\Big) ,
			\]
			we have 
			\begin{align*}
				I_1=&\,
				\frac {\kap ^2(\xi_i)(\cF^i_{\xi}\rg^i_{\xi})(\xi_i)}{16(2\al_i-1)}  
				\frac{\pa R(\xi_i)}{\pa(\xi_i)_j}\lam_i^{4\al_i-3}\\&\, +\sum_{s\neq i} \frac{\kap( \xi_s)\kap( \xi_i)(\cF^s_{\xi} \rg^s_{\xi} )(\xi_s)}{8(2\al_s-1)\lam_i}  \frac{\pa  G^g(\xi_s,\xi_i)}{\pa (\xi_i)_j}\lam_s^{4\al_s-2}\\
				&\,+ \frac{\kap(\xi_i)(\cF^i_{\xi}\rg^i_{\xi})(\xi_i))}{16\al_i(2\al_i-1)}\frac{\pa \ff}{\pa(\xi_i)_j }\lam_i^{4\al_i-3}+	O\Big(\lam_i^{4\al_i-2}\Big(\frac{|\al_i-1|}{\lam_i}+\frac{\ln  \lam_{i}}{\lam_i^{2}}\Big)\Big).
			\end{align*}
			
			\textbf{Step 2: Estimate of $I_2$.} 	It follows from   Lemmas \ref{lem:projbubble_xi_pointwise}, \ref{lem:Ve^{u_0}_pointwise} and the Sobolev embedding theorem that   \begin{align*}
				I_2
				=& \, \int_{U_{r_0}(\xi_i)} \frac{\lam_i^{4\al_i-1}\cF_{\xi}^i\rg_{\xi}^i}{(1+\lam_i^2|y_{\xi_i}|^2)^{2\al_i}}\frac{\pa \delta_i}{\pa (\xi_i)_j}w\, \d v_g\\&\, +\sum_{s}  \int_{U_{r_0}(\xi_s)}   \frac{\lam_s^{4\al_s}\cF_{\xi}^s\rg_{\xi}^s   }{(1+\lam_s^2|y_{\xi_s}|^2)^{2\al_s}}\frac{\kap(\xi_i)}{\lam_i}\frac{ \pa H^g(\cdot, \xi_i)}{\pa (\xi_i)_j}w\, \d v_g  \nonumber  \\
				&\,+O\Big(\sum_{s}  \frac{\ln \lam_i}{\lam_i^3} \int_{U_{r_0}(\xi_s)}   \frac{\lam_s^{4\al_s}\cF_{\xi}^s\rg_{\xi}^s |w|}{(1+\lam_s^2|y_{\xi_s}|^2)^{2\al_s}}\, \d v_g  +\frac{\|w\|} {\lam_i} \Big).
			\end{align*}
			By \eqref{projectedbubble}, \eqref{orthogonality},  the H\"{o}lder inequality, and the Sobolev embedding theorem,  we have 
			\begin{align*}
				&\int_{U_{r_0}(\xi_i)} \frac{\lam_i^{4\al_i-1}  }{(1+\lam_i^2|y_{\xi_i}|^2)^{2\al_i}} \Big(e^{-\varphi_{i}} \frac{\pa \delta_{i}}{\pa (\xi_i)_j} \Big)w\, \d v_g \nonumber\\
				=&\,\lam_i^{4\al_i-3}\Big\la  \frac{\pa P\delta_{i}}{\pa (\xi_i)_j} ,w\Big\ra -\int_{\Si\setminus U_{r_0}(\xi_i)} \frac{\lam_i^{4\al_i-1}  }{(1+\lam_i^2|y_{\xi_i}|^2)^{2}}\Big(e^{-\varphi_{i}} \frac{\pa \delta_{i}}{\pa (\xi_i)_j} \Big)w\, \d v_g\nonumber\\
				&\, -\lam_i^{4\al_i-3}\int_\Sigma \chi_{i}e^{-\varphi_{i}} \frac{\pa \delta_{i}}{\pa(\xi_i)_j}  \, \d v_g \int_{\Si}w  \, \d v_g \\=& O(\lam_i^{4\al_i-5}\|w\|)
			\end{align*}
			and 
			\[
			\sum_{s} \frac{1}{\lam_i} \int_{U_{r_0}(\xi_s)}   \frac{\lam_s^{4\al_s} |w|}{(1+\lam_i^2|y_{\xi_i}|^2)^{2}}\big( \lam_{i} |\al_i-1|\sqrt{\lam_{s}|y_{\xi_s}|}+ 1\big)\, \d v_g=O\Big(\sum_{s}  \frac{\lam_s^{4\al_j-2}}{\lam_i}(1+\lam_i |\al_i-1|)\|w\|\Big).
			\]	
			By the above estimates, Lemma \ref{lem:w_error_integral_estimates} and  \eqref{Phibound}, we deduce that  
			\[  	
			I_2= O\Big(\sum_{j}  \lam_j^{4\al_j-2}\Big(|\al_j-1|\|w\|+ \frac{\|w\|}{\lam_{i}}\Big )+ \frac{\|w\|}{\lam_i}\Big).   
			\]

			\textbf{Step 3: Estimate of $I_3$.}	
			Lemma  \ref{lem:projbubble_xi_pointwise} implies that 
			\[
			|I_3|=O\Big(\int_{\Si } V e^{u_0}|e^w-w-1| \,\d v_g\Big) .
			\]	
			By the same argument  as in Step 3 of the proof of Lemma \ref{lem:Ve^u_integral}, we obtain 
			\begin{equation}\label{prop2:I3}
				I_3  = O\Big(\sum_{s}  \lam_s^{4\al_s-2} \|w\|^2\Big).
			\end{equation}

			Finally, combining all the estimates above  with Lemma \ref{lem:Ve^u_integral}, we conclude the proof of  \eqref{eq:gradient_pa_xi_i}. 
		\end{proof}

		\begin{proof}[Proof of Proposition \ref{prop:J(u)expansion}]
			By applying Lemmas  \ref{lem:w_error_integral_estimates} and \ref{lem:Ve^{u_0}_pointwise}, we obtain
			\begin{align*}	
				\int_\Si  V e^{u_0}\Big( e^w-1-w-\frac{1}{2} w^2\Big) \,\d v_g =&O\Big(\sum_{j} \int_{U_{r_0}(\xi_j)} e^{\delta_{j}} |\Phi_j  | |w|^3 \,\d v_g +\|w\|^3\Big)
				\\	=&O\Big( \sum_{j} \lam_j^{4\al_j-2} \|w\|^3\Big).
			\end{align*}
			Since $w\in E_{\lam,\xi}^{p,q}$  and by using \eqref{lem:A.8-I1}, we deduce that
			\begin{align*}
				&J_{\rho }(u)-J_{\rho}( u_0)\\=& \,\frac 1 2\|w\|^2-\rho \ln \Big( 1 +\frac{\int_{\Si } V e^{u_0} w \,\d v_g}{\int_{\Si } V e^{u_0}\,\d v_g} +\frac 1 2 \frac{\int_{\Si } V e^{u_0} w^2 \,\d v_g}{\int_{\Si } V e^{u_0} \,\d v_g} + O\Big( \sum_{j} \frac{\lam_j^{4\al_j-2} \|w\|^3} {\int_{\Si } V e^{u_0} \,\d v_g} \Big) \Big)\\
				=& \,\frac 1 2\|w\|^2 -\rho \Big(\frac{\int_{\Si } V e^{u_0} w \,\d v_g}{\int_{\Si } V e^{u_0}\,\d v_g} +\frac{1}{2}\frac{\int_{\Si } V e^{u_0} w^2 \,\d v_g}{\int_{\Si } V e^{u_0} \,\d v_g}-\frac{1}{2} \Big(\frac{\int_{\Si } V e^{u_0} w \,\d v_g}{\int_{\Si } V e^{u_0}\,\d v_g} \Big)^2 + O(  \|w\|^3)\Big)\\
				=&\,f(w)+\frac{1}{2} Q(w,w)+O(\var\|w\|^2).
			\end{align*}
			The estimate \eqref{eq:J_lam} is concluded. 
		\end{proof}
		
		\begin{proof}[Proof of Proposition \ref{prop:Q_0positive}]
			By combining Lemma \ref{lem:A.7}, the proof in Lemma \ref{lem:Ve^u_integral}  and \eqref{tau_iestimate},  we obtain 
			\[ \rho \Big(\frac{\int_{\Si } V e^{u_0} w  \,\d v_g }{\int_{\Si } V e^{u_0} \,\d v_g }\Big)^2=O\Big(\sum_{i} \|w\|^2\Big( |\al_i-1|^2 + \frac 1 {\lam_i^2}\Big)\Big)=O(\var^2\|w\|^2). \]
			Using Lemmas \ref{Phibound}-\ref{lem:Ve^{u_0}_pointwise}, we further deduce that 
			\begin{align*}
				\int_{\Si } V e^{u_0} w^2 \,\d v_g =&\,
				\sum_{i} \int_{U_{r_0}(\xi_i)} \frac{\cF^i_{\xi} \rg^i_{\xi}}{ ( 1+\lam_i^2|y_{\xi_i}|^2)^{2\al_i}}\Big( 1+O\Big(\frac{\ln \lam_i}{\lam_i^2} \Big)\Big)w^2 \,\d v_g+ O( \|w\|^2)\\
				=&\,  \sum_{i} \frac{(\cF^i_{\xi} \rg^i_{\xi})(\xi_i) \lam_i^{4\al_i-2}}{8}\int_{U_{r_0}(\xi_i)} e^ {\delta_{i}}\Phi _i  w^2 \,\d v_g \\&\,+ O\Big( \lam_i^{4\al_i-2}\int_{U_{r_0}(\xi_i)} e^{\delta_{i}}\Phi _i  |y_{\xi_i}|w^2 \,\d v_g\Big)
				+O(\|w\|^2)\\
				=& \sum_{i} \frac{(\cF^i_{\xi} \rg^i_{\xi})(\xi_i)\lam_i^{4\al_i-2}}{8}\int_{U_{r_0}(\xi_i)} e^ {\delta_{i}}\Phi _i  w^2 \,\d v_g+O\Big( \sum_{i} \lam_i^{4\al_i-3}\|w\|^2\Big)\\
				=& \sum_{i} \frac{(\cF^i_{\xi} \rg^i_{\xi})(\xi_i) \lam_i^{4\al_i-2}}{8}\int_{\Si}\chi_i e^ {\delta_{i}} w^2 \,\d v_g+O\Big( \sum_{i} \lam_i^{4\al_i-2}\Big( \frac{1}{\lam_i}+|\al _i-1|\Big)\|w\|^2\Big) ,
			\end{align*}
			where $\Phi _i $ is defined in  \eqref{def:Phi}.
			Moreover, from the proof of Lemma \ref{lem:Ve^u_integral}, it follows that 
			\begin{align*}
				\int_{\Si } V e^{u_0} \,\d v_g=&\,\sum_{i} \lam_i^{4\al_i-2}\Big( \frac{\kap( \xi_i)(\cF^i_{\xi} \rg^i_{\xi})(\xi_i)}{8(2\al_i-1)} +\frac{\pi}{2}\lam_i   \mathfrak{f}_2(\xi_i)\Big)\\
				&\,+ O\Big(\sum_{i}(\ln \lam_i+|\al_i-1|)\Big)  +O(1)
			\end{align*}
			and 
			\[\int_{\Si } V e^{u_0} \,\d v_g- \int_{\Si } V e^u \,\d v_g= 	 O\Big(\sum_{i} \lam_i^{4\al_i-2} \Big(|\al_i-1|^2 + \frac{1}{\lam_i^2} +\|w\|^2\Big)\Big)  + O(\|w\|).  \]
			Thus, using  \eqref{def:tau_i} and \eqref{tau_iestimate},  we conclude that 
			\begin{align*}
				\frac{ \int_{\Si } V e^{u_0} w^2 \,\d v_g}{\int_{\Si } V e^{u_0} \,\d v_g}
				=&\, \sum_{i} \int_{U_{r_0}(\xi_i)} e^ {\delta_{i}} w^2 \,\d v_g+ O\Big( \sum_{i}\Big( |\tau_i|+|\al_i-1|+\frac{1}{\lam_i}\Big)\|w\|^2\Big)\\
				=&\,\sum_{i} \int_{\Si } \chi_{i} e^ {\delta_{i}} w^2 \,\d v_g+ O\Big( \sum_{i}\Big( |\tau_i|+|\al_i-1|+\frac{1}{\lam_i}\Big)\|w\|^2\Big)\\
				=& \, \sum_{i} \int_{\Si } \chi_{i} e^ {\delta_{i}} w^2 \,\d v_g+O(\var\|w\|^2).
			\end{align*}
			It remains to prove $Q_0(w,w)$ is coercive  in $E^{p,q}_{\lam,\xi}$, which  can be established via a standard blow-up argument (see, for instance, \cite[Lemma 6.4]{CL2003}), and we omit the details.
		\end{proof}
		



		
		\bigskip
		
		\noindent M. Ahmedou
		
		\noindent  Mathematisches Institut der Justus-Liebig-Universit\"at Giessen \\ 
		Arndtsrasse 2, D-35392 Giessen, Germany \\
		Email: \textsf{\href{mailto:Mohameden.Ahmedou@math.uni-giessen.de}{Mohameden.Ahmedou@math.uni-giessen.de}}

		\medskip
		\noindent Z. Hu
		
		\noindent  School of Mathematical Sciences, Shanghai Jiao Tong University\\
		800 Dongchuan, RD, Minhang District, Shanghai, 200240, China \\ 
		Email: \textsf{\href{mailto:zhengni_hu2021@outlook.com}{zhengni\_hu2021@outlook.com}}
		
		\medskip
		\noindent H. Wang
		
		\noindent School of Mathematical Sciences, Laboratory of Mathematics and Complex Systems, MOE\\
		Beijing Normal University, Beijing 100875,  China\\
		Email: \textsf{\href{mailto:hmw@mail.bnu.edu.cn}{hmw@mail.bnu.edu.cn}}

	\end{document}